\documentclass[reqno,11pt]{amsart}
\usepackage{amssymb}

\usepackage{mathbbold}
\usepackage{amscd}
\usepackage{amsfonts}
\usepackage{amssymb, amsmath}
\usepackage[all,cmtip]{xy}
\usepackage{enumerate}
\usepackage{graphicx}
\usepackage{epic}
\usepackage{overpic}
\usepackage{galois}
\usepackage{subfig}
\usepackage{caption}
\usepackage{varioref}

\def\w{\widetilde}
\def\wh{\widehat}
\def\s{\smallsmile}
\def\f{\smallfrown}
\numberwithin{figure}{section}

\newcommand{\kk}{\mathbf{k}}

\newcommand{\ZZ}{\mathcal {Z}}

\newcommand{\RR}{\mathcal {R}}
\newcommand{\Ss}{\mathcal {S}}
\newcommand{\NN}{\mathcal {N}}
\newcommand{\VV}{\mathcal {V}}

\newcommand{\Aa}{\mathcal {A}}
\newcommand{\BB}{\mathcal {B}}
\newcommand{\CC}{\mathcal {C}}
\newcommand{\HH}{\mathcal {H}}

\newtheorem{thm}{Theorem}[section]
\newtheorem{prop}[thm]{Proposition}
\newtheorem{lem}[thm]{Lemma}
\newtheorem{cor}[thm]{Corollary}

\newtheorem*{K-S}{Krull-Schmidt Theorem}

\theoremstyle{definition}
\newtheorem{Def}[thm]{Definition}
\newtheorem{conj}[thm]{Conjecture}
\newtheorem{exmp}[thm]{Example}
\newtheorem{cons}[thm]{Construction}

\theoremstyle{remark}

\newtheorem{rem}[thm]{Remark}
\newtheorem{prob}[thm]{Problem}
\newtheorem*{conv}{\indent \rm Convention}
\newtheorem*{que}{Question}
\newtheorem{Que}[]{Question}
\numberwithin{equation}{section}

\usepackage[colorlinks]{hyperref}
\hypersetup{citecolor=blue,linkcolor=blue}
\begin{document}
\title[On the cohomology of moment-angle complexes...]{On the cohomology of moment-angle complexes associated to Gorenstein* complexes}
\author[F.~Fan \& X.~Wang]{Feifei Fan and Xiangjun Wang}
\thanks{The authors are supported by the National Natural Science Foundation of China (NSFC no. 11261062, 11471167)}
\address{Feifei Fan, School of Mathematical Sciences and LPMC, Nankai University, Tianjin 300071, P.~R.~China}
\email{fanfeifei@mail.nankai.edu.cn}
\address{Xiangjun Wang, School of Mathematical Sciences and LPMC, Nankai University, Tianjin 300071, P.~R.~China}
\email{xjwang@nankai.edu.cn}
\keywords{moment-angle complex; Goreistein* complex; indecomposoble ring; prime manifold; cohomological rigidity}
\subjclass[2010]{Primary 13F55, 52B05, 52B10, 55U10; Secondary 05A19, 05E40, 57R19}
\maketitle
\begin{abstract}
The main goal of this article is to study the cohomology rings and their applications of moment-angle complexes associated to Gorenstein* complexes, especially, the applications in combinatorial commutative algebra and combinatorics. First, we give a topological characterization of Gorenstein* complexes in terms of Alexander duality (as an application we give a topological proof of Stanley's Theorem). Next we give some cohomological transformation formulae of $\mathcal {Z}_{K}$, which are induced by some combinatorial operations on the Gorenstein* complex $K$, such as the connected sum operation and stellar subdivisions. We also prove that $\mathcal {Z}_{K}$ is a prime manifold whenever $K$ is a flag $2$-sphere by proving the indecomposability of their cohomology rings. Then we use these results to give the unique decomposition of the cohomology rings of moment-angle manifolds associated to simplicial $2$-spheres, and explain how to use it to detect the cohomological rigidity problem of these moment-angle manifolds.
\end{abstract}
\section{Introduction and main results}
In 1990's Davis and Januszkiewicz \cite{DJ91} introduced quasi-toric manifolds over a simple polytope $\mathcal {P}$, a topological generalization of projective toric varieties which were being studied intensively by algebraic geometers. They observed that every quasi-toric manifold is the quotient of the
manifold $\mathcal {Z}_{\mathcal {P}}$ constructed from the same polytope $\mathcal {P}$ (now called \emph{moment-angle manifold}) by the free action of a real torus. Buchstaber and Panov \cite{BP00} generalized this construction to any simplicial complex $K$,
and named it the \emph{moment-angle complex associated to $K$}. It has been actively
studied in toric topology and has many connections with symplectic and algebraic geometry, and combinatorics.

Throughout this paper, we use the common notation $\kk$ for the ground ring, which is always assumed
to be the ring $\mathbb{Z}$ of integers or a field, and whenever there is no confusion, we use the tensor product notation $\otimes$ for $\otimes_\kk$. Given a simplicial complex $K$, there is an associated ring
known as the Stanley-Reisner ring of $K$, denoted $\kk(K)$.
The ring $\kk(K)$ is a quotient of a finitely generated polynomial ring $\kk[v_1,\dots,v_m]$ with generators
$v_i$ for each vertex of $K$. Hochster \cite{H75}, in purely algebraic work,
calculated the Tor-modules $\mathrm{Tor}_{\kk[v_1,\dots,v_m]}(\kk(K),\kk)$ in terms of the full subcomplexes of $K$.
Buchstaber-Panov \cite{BP00} proved
that the cohomology algebra of $\mathcal {Z}_{K}$ is isomorphic to $\mathrm{Tor}_{\kk[v_1,\dots,v_m]}(\kk(K),\kk)$. Using the cellular cochain algebra model for $\mathcal {Z}_{K}$,
Baskakov \cite{B02} gave an explicit formula for the cup product in $H^*(\mathcal {Z}_{K},\kk)$, in terms
of pairings between full subcomplexes (However, there is no proof in this very short paper. A complete proof of this is given by Bosio-Meersseman \cite{BM06}). Other works on the cohomology ring of $\mathcal {Z}_{K}$ can be found in
\cite{BBCG1,BBCG2,BM06,DS07,M03,M06,P08,WZ13}.

This paper is a study of the ring structure of $H^*(\ZZ_K)$ and its applications to detecting the topology of $\ZZ_K$. More precisely, the paper is organized as follows.

\S\ref{sec:1} is a review of notations and previous results on moment-angle complexes. In \S\ref{sec:2}, we introduce the Gorenstein* complexes and some of their algebraic and topological properties. We prove that
(Theorem \ref{thm:dual}): a simplicial complex $K$ is a generalized homology sphere if and only if it satisfies Alexander duality between full subcomplexes. Then by using the Poincar\'e duality of Gorenstein rings and the theory of moment-angle complexes, we give a topological proof for Stanley's Theorem (Theorem \ref{thm:3}).
In \S\ref{sec:3},  we give a formula to calculate the cohomology ring of the moment angle complex
associated to the connected sum of two Gorenstein* complexes. The main result of this section is:
\begin{thm}[Theorem \ref{thm:4}]\label{thm:a}
Let $K_1$ and $K_2$ be two $(n-1)$-dimensional ($n\geq 2$) Grenstein* complex over $\kk$ with $m_1$ and $m_2$ vertices respectively,
and let $K=K_1\#K_2$. Then the reduced cohomology ring of $\mathcal {Z}_{K}$ is given by the isomorphism
\[\w H^*(\mathcal {Z}_{K};\kk)\cong \mathcal {R}(K_1,K_2;\kk)/\mathcal {I}(K_1,K_2;\kk),\]
where
\[
\mathcal {R}(K_1,K_2;\kk)=G^{m_2-n}(\w {H}^*(\mathcal {Z}_{K_1};\kk))\times G^{m_1-n}(\w {H}^*(\mathcal {Z}_{K_2};\kk))\times\w {H}^*(M;\kk),
\]
(see Construction \ref{cons:1} for the operation $G$)
\[M=\overset{m_1+m_2-2n}{\underset{i=2}{\sharp}}\lambda(i)(S^{i+1}\times S^{m_1+m_2-i-1}),\]
\[\lambda(i)=\tbinom{m_1+m_2-2n}{i}-\tbinom{m_1-n}{i}-\tbinom{m_2-n}{i},\]
$\mathcal {I}(K_1,K_2;\kk)$ is an ideal of $\mathcal {R}(K_1,K_2;\kk)$ generated by
\[([Z_1],0,0)-(0,0,[M])\, \text{ and }\, (0,[Z_2],0)-(0,0,[M]),\] where $[Z_1]$ (resp. $[Z_2]$, $[M]$) is the top dimensional generator of
 $G^{m_2-n}(\w {H}^*(\mathcal {Z}_{K_1};\kk))$ (resp. $G^{m_1-n}(\w {H}^*(\mathcal {Z}_{K_2};\kk))$, $\w {H}^*(M;\kk)$).
\end{thm}

Theorem \ref{thm:a} leads to a conjecture on the topology of $\mathcal {Z}_{K_1\#K_2}$:
\begin{conj}\label{conj:a}
If $K_1,\,K_2$ are generalized homology spheres, then $\mathcal {Z}_{K_1\#K_2}$ is homeomorphic to
\[\mathcal {G}^{m_2-n}(\mathcal {Z}_{K_1})\#\mathcal {G}^{m_1-n}(\mathcal {Z}_{K_2})
\#\overset{m_1+m_2-2n}{\underset{i=2}{\sharp}}\lambda(i)(S^{i+1}\times S^{m_1+m_2-i-1}),\]
where $\mathcal {G}$ is an operation on manifolds (See Definition \ref{def:2}) first defined and studied by Gonz\'alez Acu\~na.
\end{conj}

In \S \ref{sec:4}, we study the cohomology ring of the moment angle complex
associated to the stellar subdivision $\mathrm{S}_\sigma K$ of a given simplicial sphere $K$. Although in general it is difficult to describe the relation between
$H^*(\mathcal {Z}_{\mathrm{S}_\sigma K})$ and $H^*(\ZZ_K)$, there is a simple formula to calculate $H^*(\mathcal {Z}_{\mathrm{S}_\sigma K})$ when the simplex $\sigma$ satisfies some local conditions.
That is:
\begin{thm}[Theorem \ref{thm:5}]\label{thm:b}
Let $K$ be a Corenstein* complex of dimension $n-1$ with $m$ vertices,
$\sigma\in K$ be a simplex of dimension $q$. Let $\mathcal {V}$ be the vertex set of $\mathrm{link}_K\sigma$, $s=|\mathcal {V}|+|\sigma|$. If for any $I\subset \mathcal {V}$, the inclusion map
$\varphi_I:(\mathrm{link}_K\sigma)_I\to K_I$ is nullhomotopic, and if one of the following additional conditions are satisfied:
\begin{enumerate}[(a)]
\item $\kk$ is a field.
\item $\kk=\mathbb{Z}$ and $H^*(\mathcal {Z}_{\mathrm{link}_K\sigma})$ is torsion free.
\end{enumerate}
Then the cohomology ring of $\mathcal {Z}_{\mathrm{S}_\sigma K}$ is given by the isomorphism
\[\w H^*(\mathcal {Z}_{\mathrm{S}_\sigma K};\kk)\cong\w H^*(\mathcal {G}(\mathcal {Z}_K)\#\, Y;\kk),\]
where
\begin{align*}
&Y={\underset{i+j\geq1}{\sharp}}f_i\cdot\tbinom{m-s}{j}(S^{i+j+2}\times S^{m+n-i-j-1}),\\
&f_i=\mathrm{rank}\,H^i(\mathcal {Z}_{\mathrm{link}_K\sigma}).
\end{align*}
\end{thm}

In \S \ref{sec:5}, we study the moment-angle manifolds associated to flag $2$-spheres, and get an important algebraic property of $H^*(\mathcal {Z}_K)$ when $K$ is a flag $2$-sphere, which implies
a topological property of $\mathcal {Z}_K$. They are shown by the following

\begin{thm}[Theorem \ref{thm:7}]
Let $K$ be a flag $2$-sphere. Then $\w {H}^*(\mathcal {Z}_{K})/([\mathcal {Z}_{K}])$ is a graded indecomposable ring, where $[\mathcal {Z}_{K}]$ is the top class of $\w {H}^*(\mathcal {Z}_{K})$.
\end{thm}

\begin{thm}[Theorem \ref{thm:8}]
Let $K$ be a flag $2$-sphere. Then $\mathcal {Z}_{K}$ is a prime manifold.
\end{thm}

In \S \ref{sec:6}, we further discuss the cohomology ring $H^*(\mathcal {Z}_{K})$ when $K$ is a simplicial $2$-sphere. By using the results in \S \ref{sec:3} and \S \ref{sec:5}, we get the following
\begin{thm}[Theorem \ref{thm:9}]\label{thm:c}
Let $K$ and $K'$ be simplicial $2$-spheres. If
\[K=K_1\#K_2\#,\dots,\#K_n,\quad K'=K'_1\#K'_2\#,\dots,\#K'_{n'}\]
such that each $K_i$ and $K_i'$ are irreducible, and if $H^*(\mathcal {Z}_{K};\kk)\cong H^*(\mathcal {Z}_{K'};\kk)$ (as graded rings), then
$n=n'$ and there is a permutation $j\curvearrowright j'$ such that
\[H^*(\mathcal {Z}_{K_j};\kk)/([\mathcal {Z}_{K_j}])\cong H^*(\mathcal {Z}_{K'_{j'}};\kk)/([\mathcal {Z}_{K'_{j'}}]),\ 1\leq j\leq n.\]
\end{thm}
This theorem together with Conjecture \ref{conj:a} provides a heuristic method for solving the cohomological rigidity problem of moment-angle manifolds associated to simplicial $2$-spheres.

\section{Preliminaries}\label{sec:1}
\subsection{simplicial complexes and face rings}
Let $\mathcal {S}$ be a finite set. Given a subset $\sigma\subseteq \mathcal {S}$, we denote its cardinality by $|\sigma|$.
\begin{Def}
An \emph{abstract simplicial complex} on a set $\mathcal {S}$ is a collection $K=\{\sigma\}$ of subsets of $\mathcal {S}$
such that for each $\sigma\in K$ all subsets of $\sigma$ (including $\varnothing$) also belong to $K$. An subset $\sigma\in K$ is called a \emph{simplex} of $K$.
A maximal simplex is also called a \emph{facet}.

One-element simplices are called \emph{vertices} of $K$. If $K$ contains all one-element subsets of $\mathcal {S}$, then we say that $K$ is a simplicial complex on the vertex set $\mathcal {S}$.

It is sometimes convenient to consider simplicial complexes $K$ whose vertex sets
are proper subsets of $\mathcal {S}$. In this case we refer to a one-element subset of $\mathcal {S}$ which is
not a vertex of $K$ as a \emph{ghost vertex}.

The \emph{dimension} of a simplex $\sigma\in K$ is $\mathrm{dim}\sigma=|\sigma|-1$. The dimension
of an abstract simplicial complex is the maximal dimension of its simplices. A
simplicial complex $K$ is \emph{pure} if all its facets have the same dimension.

\end{Def}

In most construction of this paper it is safe to fix an ordering in $\mathcal {S}$ and identify $\mathcal {S}$ with the index set $[m]=\{1,\dots,m\}$. We denote by $\Delta^{m}$ the abstract simplicial complex $2^{[m+1]}$ consisting of all subsets of $[m+1]$,
and denote by $\partial\Delta^{m}$ the simplicial complex $2^{[m+1]}\setminus[m+1]$ (the boundary of $\Delta^{m}$).

\begin{Def}
Let $K_1$ and $K_2$ be simplicial complexes on sets $\mathcal {S}_1$ and $\mathcal {S}_2$ respectively. The \emph{join} of $K_1$ and $K_2$ is the simplicial complex
\[K_1*K_2=\{\sigma\subseteq \mathcal {S}_1\sqcup \mathcal {S}_2: \sigma=\sigma_1\cup\sigma_2,\, \sigma_1\in K_1, \sigma_2\in K_2\}.\]
\end{Def}

\begin{Def}\label{def:1}
Let $K_1$, $K_2$ be two pure $n$-dimensional simplicial complexes on sets $\mathcal {S}_1$, $\mathcal {S}_2$ respectively.
Suppose we are given two facets $\sigma_1\in K_1,\ \sigma_2\in K_2$.
Fix an identification of $\sigma_1$ and $\sigma_2$ (by identifying their vertices),
and denote by $\mathcal {S}_1\cup_\sigma\mathcal {S}_2$ the union of $\mathcal {S}_1$ and $\mathcal {S}_2$ with
$\sigma_1$ and $\sigma_2$ identified.  Then the simplicial complex $(K_1\cup K_2)\setminus\sigma$ is called a \emph{connected sum} of $K_1$ and $K_2$. It depends on the way of choosing the two simplices and identifying their vertices. Let $\mathcal {C}(K_1\#K_2)$ denote the set of
connected sums of $K_1$ and $K_2$.
\end{Def}

\begin{Def}
Let $K$ be an abstract simplicial complex on $[m]$.
\begin{enumerate}[(1)]
\item Given a subset $I\subseteq \mathcal [m]$, define $K_I\subseteq K$ to be the \emph{full sub-complex} of $K$ consisting of all simplices of $K$
which have all of their vertices in $I$, that is
\[K_I=\{\sigma\cap I\mid \sigma\in K\}.\]

\item For a simplex $\sigma$ of $K$, the \emph{link} and the \emph{star} of $\sigma$ are the simplicial subcomplexes
\begin{align*}
\mathrm{link}_K\sigma&=\{\tau\in K:\sigma\cup\tau\in K,\,\sigma\cap\tau=\varnothing\};\\
\mathrm{star}_K\sigma&=\{\tau\in K:\sigma\cup\tau\in K\}.
\end{align*}

\item The simplicial complex $\Delta^0*K$ (the join of $K$ and a point) is called the \emph{cone} over $K$.
\end{enumerate}
\end{Def}

\begin{Def}
A simplicial complex $K$ is called a \emph{flag complex} if each of its missing faces has two vertices.
\end{Def}

\begin{Def}
Let $K$ be a simplicial complex with vertex set $[m]$.
A \emph{missing face} of $K$ is a sequence $(i_1,\dots,i_k)\subseteq[m]$ such that $(i_1,\dots,i_k)\not\in K$,
but every proper subsequence of $(i_1,\dots,i_k)$ is a simplex of $K$. Denote by $MF(K)$ the set of all missing faces of $K$.
\end{Def}

We use the common notation $\kk$ to denote a field or the ring $\mathbb{Z}$ of integers. Let $\kk[m]=\kk[v_1,\dots,v_2]$ denote the graded polynomial algebra on $m$ variables,
deg$(v_i)=2$.
\begin{Def}
The \emph{face ring} (also known as the \emph{Stanley-Reisner ring}) of a simplicial complex $K$ on a set [m] is the quotient ring
\[\kk(K)=\kk[m]/\mathcal {I}_K\]
where $\mathcal {I}_K$ is the ideal generated by those square free monomials $v_{i_1}\cdots v_{i_s}$
for which $(i_1,\dots,i_s)$ is not a simplex in $K$. We refer to $\mathcal {I}_K$ as the Stanley-Reisner ideal of $K$.
\end{Def}

\subsection{moment-angle complexes}
We start with a generalization of the notion of moment-angle complex.
\begin{Def}
Let $(\underline{X},\underline{A})$ denote a collection of based CW pairs $\{(X_i,A_i,x_i)\}_{i=1}^m$. Given
a simplicial complex $K$ on $[m]$ (may have ghost vertices), the \emph{polyhedral product} determined by $(\underline{X},\underline{A})$ and $K$ is defined  to be:
\[\mathcal {Z}_K^{[m]}(\underline{X},\underline{A})=\bigcup_{\sigma\in K}D(\sigma),\]
where \[D(\sigma)=\prod_{i=1}^m Y_i\quad \text{and}\quad Y_i=\begin{cases}X_i\quad \text{if}\quad i\in\sigma,\\
A_i\quad \text{if}\quad i\not\in\sigma.\end{cases}\]
If $[m]$ is just the vertex set of $K$, the notation $\mathcal {Z}_K^{[m]}(\underline{X},\underline{A})$ is simplified to $\mathcal {Z}_{K}(\underline{X},\underline{A}))$.

In the case when $(\underline{X},\underline{A})=(D^2, S^1)$, we obtain the usual
\emph{moment-angle complex}, $\mathcal {Z}_K^{[m]}=\mathcal {Z}_K^{[m]}(D^2,S^1)$, and in the case $(\underline{X},\underline{A})=(D^1, S^0)$, $\RR\ZZ_K^{[m]}=\ZZ_K^{[m]}(D^1,S^0)$ is called the \emph{real moment-angle complex}.
%These two classes of spaces are the main goal of our study.
\end{Def}

Now suppose that $K$ is an $(n-1)$-dimentional simplicial sphere (a triangulation of a sphere) with $m$ vertices. Then, as shown by Buchstaber and Panov \cite{BP00},
the moment-angle complex $\mathcal {Z}_{K}$ is a manifold of dimension $n+m$. Furthermore, Cai \cite{C14} gave a necessary and sufficient condition for $\mathcal {Z}_{K}$ to be a topological manifold.

\begin{Def}\label{def:A1}
A space $X$ is a \emph{homology $n$-manifold} if it has the same
local homology groups as $\mathbb{R}^n$, i.e., if for each $x\in X$,
\[H_i(X,X-x)=
\begin{cases}
\mathbb{Z}\quad&\text{ if } i=n,\\
0\quad&\text{ otherwise. }
\end{cases}
\]
$X$ is \emph{orientable} if there is a function $x\mapsto\mu_x$ assigning to each $x\in X$ a choice of
generator $\mu_x$ of $H_n(X,X-x)$ satisfies that for each $x\in X$ there is an open neighborhood $U$ and a class $\alpha\in H_n(X,X-U)$ such that the map induced by the inclusion $(X,X-U)\to(X,X-x)$ maps $\alpha$ to $\mu_x$ for each $x\in U$.

A space $X$ is a \emph{generalized homology $n$-sphere} (for short, a ``GHS'') if it is a homology $n$-manifold with the same homology as $S^n$ (`generalized' because a homology sphere is usually assumed to be a manifold).
\end{Def}
Definitions \ref{def:A1} can be extended in an obvious fashion to define a homology manifold or GHS over $\kk$. For a simplicial complex $K$,

\begin{thm}{\cite[Corollary 2.10]{C14}}\label{thm:mfd}
Let $K$ be a simplicial complex with $m$ vertices. Then the moment-angle complex $\ZZ_K$ is a topological $(n+m)$-manifold if and only if $|K|$ is a generalized homology $(n-1)$-sphere.
\end{thm}
In the above case, $\ZZ_K$ is referred to as a \emph{moment-angle manifold}.
In particular, if $K$ is a polytopal sphere (a triangulated sphere isomorphic to the boundary complex of a simplicial polytope),
or more generally a starshaped sphere (see Definition \ref{def:b}), then $\mathcal {Z}_{K}$ admits a smooth structure \cite{PU12}.

\begin{Def}\label{def:b}
A simplicial sphere $K$ of dimension $n$ is said to be \emph{starshaped} if there is a geometric realization $|K|$ of $K$ in $\mathbb{R}^n$ and a point $p\in \mathbb{R}^n$
with the property that each ray emanating from $p$ meets $|K|$ in exactly one point.
\end{Def}

\subsection{Cohomology rings of moment-angle complexes}
In this paper we mainly use the following results to calculate the cohomology ring of $\mathcal {Z}_{K}$,
which is proved by Buchstaber and Panov \cite[Theorems 7.6]{BP02} for the case over a field, \cite{BBP04} for the general case. Another proof of Theorem \ref{thm:1} for the case over $\mathbb{Z}$ was given by Franz \cite{M06}.

\begin{thm}[Buchstaber-Panov, {\cite[Theorem 4.7]{P08}}]\label{thm:1}
Let $K$ be a abstract simplicial complex on the set $[m]$. Then the cohomology
algebra of the moment-angle complex $\mathcal {Z}_{K}^{[m]}$ is given by the isomorphisms
\begin{align*}
H^*(\mathcal {Z}_{K}^{[m]}; \kk)\cong \mathrm{Tor}^{*,\,*}_{\kk[m]}(\kk(K),\kk)\cong\bigoplus_{I\subseteq [m]} \w {H}^*(K_I;\kk)
\end{align*}
where
\[H^p(\mathcal {Z}_{K}^{[m]}; \kk)\cong \bigoplus_{\substack{J\in[m]\\-i+2|J|=p}}\mathrm{Tor}^{-i,\,2J}_{\kk[m]}(\kk(K),\kk)\]
and
\[\mathrm{Tor}^{-i,\,2J}_{\kk[m]}(\kk(K),\kk)\cong \w {H}^{|J|-i-1}(K_J;\kk).\]
We assume $\w H^{-1}(\varnothing;\kk)=\kk$ above.
\end{thm}

\begin{Def}
Since $\mathrm{Tor}^{*,\,*}_{\kk[m]}(\kk(K),\kk)$ is a graded algebra, the isomorphisms of Theorem \ref{thm:1} turn the direct sum
$\bigoplus_{I\subseteq [m]} \w {H}^*(K_I;\kk)$
into a graded $\kk$-algebra, called the \emph{Hochster algebra} of $K$ on [m] and denoted $\mathcal {H}^*_{[m]}(K;\kk)$
(if $[m]$ is the vertex set of $K$, simplify it to $\mathcal {H}^*(K;\kk)$),
where $\mathcal {H}^i_{[m]}(K;\kk)=\bigoplus_{J\subseteq [m]}\w H^{i-1}(K_J;\kk)$. Let
\[\w {\mathcal {H}}^*_{[m]}(K)=\bigoplus_{J\subseteq [m],\, J\neq\varnothing}\w H^*(K_J),\]
called the \emph{reduced Hochster algebra} of $K$.

Given two elements $u,v\in\mathcal {H}^*_{[m]}(K)$, denote by $u*v$ the product of $u$ and $v$ in $\mathcal {H}^*_{[m]}(K)$, called the \emph{star product}.
\end{Def}

So $\mathcal {H}^*_{[m]}(K;\kk)$ is an augmented $\kk$-algebra with $\w{\mathcal {H}}^*_{[m]}(K;\kk)$ as its augmentation ideal.
Theorem \ref{thm:1} gives that
\[H^*(\mathcal {Z}_{K}^{[m]}; \kk)\cong \mathcal {H}^*_{[m]}(K;\kk),\]
and
\[\w H^*(\mathcal {Z}_{K}^{[m]};\kk)\cong \w{\mathcal {H}}^*_{[m]}(K;\kk).\]

Let $\w C^q(K;\kk)$ denote the $q$th reduced simplicial cochain group of $K$ with coefficients in $\kk$. For a oriented (ordered) simplex $\sigma=(i_1,\dots,i_p)$ of $K$, denote still by $\sigma\in\w C^{p-1}(K;\kk)$ the basis cochain corresponding to $\sigma$; it takes value $1$ on $\sigma$ and vanishes on all other simplices.
\begin{Def}\label{def:cup}
Define a $\kk$-bilinear \emph{union product} in the simplicial cochains of full subcomplexes of $K$
\begin{align*}
\sqcup :\w C^{p-1}(K_I;\kk)\otimes \w C^{q-1}(K_J;\kk)&\to \w C^{p+q-1}(K_{I\cup J};\kk), \quad p,q\geq 0\\
\sigma\otimes\tau&\mapsto\sigma\sqcup\tau
\end{align*}
by setting $\sigma\sqcup \tau$ to be the juxtaposition of $\sigma$ and $\tau$ if $I\cap J=\varnothing$ and $\sigma\cup\tau$ is a simplex of $K_{I\cup J}$; zero otherwise.

Similarly, the \emph{excision product} in the simplicial chains and cochains of full subcomplexes is defined by
\begin{align*}
\sqcap :\w C_{p+q-1}(K_I;\kk)\otimes \w C^{p-1}(K_J;\kk)&\to \w C_{q-1}(K_{I\setminus J};\kk), \quad p,q\geq 0\\
\sigma\otimes\tau&\mapsto\sigma\sqcap\tau
\end{align*}
Here $\sigma\sqcap\tau=\varepsilon_{\sigma,\tau}(\sigma\setminus\tau)$ if $J\subset I$, $\tau\subset\sigma$ and $\sigma\setminus\tau\subset I\setminus J$; zero otherwise, and $\varepsilon_{\sigma,\tau}$ is the sign of the permutation sending $\tau\sqcup(\sigma\setminus\tau)$ to $\sigma$.
\end{Def}
It is easily verified that the union product of cochains
induces a union product of cohomology classes in the Hochster algebra of $K$:
\[\sqcup :\HH^{p}(K;\kk)\otimes \HH^{q}(K;\kk)\to \HH^{p+q}(K;\kk), \quad p,q\geq 0.\]
Similarly, there is an induced excision product in homology and cohomology of the full subcomplexes of $K$. Union and excision product are related by the formula \[\psi(c\sqcap\phi)=(\phi\sqcup\psi)(c)\] for $c\in\w C_{p+q-1}(K_I)$, $\phi\in\w C^{p-1}(K_J)$ and $\psi\in\w C^{q-1}(K_{I\setminus J})$.
Intuitively, the union product (resp. excision produc) is an analogue of cup product (resp. cap product) in cohomology (resp. homology and cohomology) of a space.
Actually, the union and excision product for $K$ do respectively correspond (up to a sign) to the cup and cap product for $\ZZ_k$.  In \S\ref{sec:2} we will futher show that like the Poincar\'e duality of manifolds is given by the cap product map, the Alexander duality of Gorenstein* complexes is given by the excision product map.

Baskakov in \cite{B02}
asserted that the star product and the union product are identical. Howerver there is a sign difference between these two products. This was pointed out by Bosio and Meersseman \cite[Theorem 10.1]{BM06} (Buchstaber and Panov also indicated this defect \cite{BP15}).

\begin{thm}[{\cite[Proposition 3.2.10]{BP15}}]\label{thm:e}
Let $K$ be a simplicial complex on $[m]$.  Given cohomology classes
$[u]\in \w H^{p-1}(K_I)$ and $[v]\in \w H^{q-1}(K_J)$ with $I,J\in[m]$, then
\[
[u]*[v]=(-1)^{|I|q+\theta (I,J)}[u]\sqcup [v],
\]
where $\theta (I,J)$ is defined by $\theta (I,J)=\sum_{i\in I}\theta (i,J)$, and $\theta (i,J)$ is the number of elements $j\in J$,
such that $j<i$.
\end{thm}

In fact the formula given in \cite[Proposition 3.2.10]{BP15} (also \cite[Theorem 10.1]{BM06}) is not exactly the same as in Theorem \ref{thm:e}. They showed that if
$\sigma=(l_1,\dots,l_p)\in \w C^{p-1}(K_I)$ and $\tau=(m_1,\dots,m_q)\in \w C^{q-1}(K_J)$ are two cochain simplices with
$l_1<\dots<l_p,\, m_1<\dots<m_q$, then
\[
\sigma*\tau=
\begin{cases}
(-1)^\zeta\,\sigma\uplus\tau&\ \text{ if } I\cap J=\varnothing,\\
0&\  \text{ otherwise.}
\end{cases}\]
where \[\zeta=\theta(\sigma, I)+\theta(\tau, J)+\theta(\sigma\cup \tau, I\cup J)+\theta(I\setminus\sigma,J\setminus\tau),\]
and $\sigma\uplus\tau$ is the cochain simplex of $\w C^{p+q-1}(K_I*K_J)$ corresponding to $\sigma\sqcup\tau$ written in increasing order. Since $\zeta$
depends on both $I,J$ and $\sigma,\tau$, it is not convenient to use it to describe the relation between union product and star product. Now we prove that $\zeta$ is equal to the sign given in Theorem \ref{thm:e}.

\begin{proof}[Proof of Theorem \ref{thm:e}]
First note that \[\theta(\sigma\cup\tau, I\cup J)=\theta(\sigma,I)+\theta(\sigma,J)+\theta(\tau,I)+\theta(\tau,J),\]
and
\[\theta(I\setminus\sigma,J\setminus\tau)=\theta(I\setminus\sigma,J)-\theta(I\setminus\sigma,\tau).\]
Then
\[\zeta\equiv \theta(I,J)+\theta(\tau,I)+\theta(I\setminus\sigma,\tau)\quad \text{mod } 2\]

On the other hand, it is easy to see that
\[\sigma\uplus\tau=(-1)^{\theta(\sigma,\tau)}\sigma\sqcup\tau.\]
So if we can prove $\theta(\tau,I)+\theta(I,\tau)\equiv |I|q$ mod $2$, then the formula holds. Note that
\[\theta(I,\tau)=\sum_{v\in I}(|\tau|-g(v,\tau)),\]
where $g(v,\tau)$ is the cardinality of the set $\{u\in\tau\mid u>v\}$. Let $f(u,v)=1$ if $u>v$ and zero otherwise.
Then we have \[\theta(\tau,I)=\sum_{u\in\tau, v\in I}f(u,v)=\sum_{v\in I}g(v,\tau)\]
Combine the equations above and the fact that $|\tau|=q$, we get the desired equation.
\end{proof}

\subsection{functorial properties of moment-angle complexes}
Let $K$ be a simplicial complex on $[m]$, then $\mathcal{Z}_K^{[m]}$ can be seen as a subspace of the \emph{unit polydisk} $(D^2)^m$ of $\mathbb{C}^m$, where
\[(D^2)^m=\{(z_1,\dots,z_m)\in\mathbb{C}^m: |z_i|\leq 1,\ i=1,\dots,m\}\]

A set map $\varphi:[l]\to[m]$ induces a map between polydisks
\[\psi:\prod_1^l(D^2)^l\to\prod_1^m(D^2)^m,\quad (x_1,\dots,x_l)\to(y_1,\dots,y_m),\]
where
\[y_j=\prod_{i\in \varphi^{(-1)}(j)}x_i\quad \text{ for } j=1,\dots,m,\]
and set $y_j=1$ if $\varphi^{(-1)}(j)=\varnothing$.

\begin{prop}[{\cite[Lemma 4.2]{P08}}]\label{prop:0}
Let $\varphi: K_1\to K_2$ be a simplicial map between simplicial complexes $K_1$ and $K_2$ on $[m_1]$ and $[m_2]$ respectively. Then there is a equivariant map
$\varphi_\mathcal {Z}:\mathcal{Z}_{K_1}^{[m_1]}\to \mathcal{Z}_{K_2}^{[m_2]}$, which is the restriction of $\psi$ (defined as above) to $\mathcal{Z}_{K_1}^{[m_1]}$.
\end{prop}

So there is a covariant functor $\mathcal {Z}: K\mapsto \mathcal {Z}_K^{[m]}$ from the category of finite simplicial
complexes and simplicial maps to the category of spaces with torus actions
and equivariant maps (the \emph{moment-angle complex functor}).

Meanwhile, (the notation is as in Proposition \ref{prop:0}) define a homomorphism $f:\kk[m_2]\to \kk[m_1]$ by
\[f(u_j):=\sum_{i\in \varphi^{(-1)}(j)}v_i.\]

\begin{prop}[{\cite[Proposition 2.2]{P08}}]
$f$ induces a homomorphism $\kk(K_2)\to \kk(K_1)$, and then induces a homomorphism
\[\varphi^*_{\mathrm{Tor}}:\mathrm{Tor}^{*,\,*}_{\kk[m_2]}(\kk(K_2),\kk)\to \mathrm{Tor}^{*,\,*}_{\kk[m_1]}(\kk(K_1),\kk),\]
where $\varphi^*_{\mathrm{Tor}}$ is induced by the change of rings.
\end{prop}

Thus we have a contravariant functor
\[\textrm{Tor-alg}:K\mapsto \mathrm{Tor}^{*,\,*}_{\kk[m]}(\kk(K),\kk).\]
from simplicial complexes to bigraded $k$-algebras.

On the other hand, Baskakov in \cite{B02} defined a homomorphism
\[\varphi^*_\mathcal {H}: \mathcal {H}^{*}_{[m_2]}(K_2)\to \mathcal {H}^{*}_{[m_1]}(K_1),\]
which is generated by the homomorphisms
\[\varphi^*_{I,J}:\w H^*((K_2)_J)\to \w H^*((K_1)_I)\ \text{ for all } I\subset  [m_1],\ J\subset [m_2],\]
where $\varphi^*_{I,J}=\varphi^*\mid(K_1)_I:\w H^*((K_2)_J)\to\w H^*((K_1)_I)$ if $\varphi(I)=J$ and $|I|=|J|$,  and $\varphi^*_{I,J}=0$
in all other cases. In particular, if $K_1$ is a subcomplex of $K_2$, then $\varphi^*_\mathcal {H}=\bigoplus_{I\subset [m_1]}\varphi^*_I$, where
$\varphi^*_I=\varphi^*\mid(K_1)_I:\w H^*((K_2)_I)\to \w H^*((K_1)_I)$. So we have another contravariant functor:
\[\mathcal {H}:K\mapsto \mathcal {H}^*_{[m]}(K).\]
\begin{conv}
We use the simplified notation $\varphi^*$ to denote $\varphi^*_\mathcal {H}$ whenever it creates no confusion.
\end{conv}

The relation between these functors is given by the following
\begin{thm}[{\cite[Theorem 4.7, Theorem 5.1]{P08}}]\label{thm:0}
Let $H^*:X\to H^*(X;\kk)$ be the cohomology functor. Then there are natural equivalences induced by the isomorphisms in Theorem \ref{thm:1}, that is
\[H^*\circ \mathcal {Z}= \mathrm{Tor\text{-}alg}= \mathcal {H}.\]
\end{thm}

\section{Gorenstein* complexes and Alexander duality}\label{sec:2}
In commutative algebra, \emph{Gorenstein rings} are an important class of rings because of their nice properties such as finite injective dimension and self-duality. In the language of homological algebra, a Gorenstein ring is a \emph{Cohen-Macaulay ring of type one} (cf. \cite{S96,BH98}).
If the face ring (over a field $\kk$) of a simplicial complex $K$ is Gorenstein and $K$ is not a cone, then it is called \emph{Gorenstein* over $\kk$}.
The Gorenstein property has several topological and algebraic
interpretations (some of which we list below) and is very important for both topological
and combinatorial applications of the face rings.

The Koszul homology of a Gorenstein* ring behaves like the cohomology algebra of a manifold: it satisfies \emph{Poincar\'e duality}.
This fundamental result was proven by Avramov and Golod \cite{AG71} for general Gorenstein rings. Here we state the graded version of their theorem in
the case of face rings.

\begin{Def}
A finite dimensional graded $\kk$-algebra $A=\bigoplus_{i=0}^d A^i$ ($\kk$ is a field) will be called a \emph{Poincar\'e algebra} if $\mathrm{dim}_{\kk}\,A^d=1$, and the $\kk$-homomorphisms
\begin{align*}
A^i&\to\mathrm{Hom}_\kk(A^{d-i}, A^d),\\
a&\mapsto\phi_a, \ \ \ \text{where}\  \phi_a(b)=ab
\end{align*}
are isomorphisms for $0\leqslant i\leqslant d$.
\end{Def}

\begin{thm}[{\cite[Avramov-Golod]{AG71}}, {\cite[Theorem 3.4.4]{BP15}}]\label{thm:AG}
For a simplicial complex $K$, the face ring $\kk(K)$ over a field $\kk$ is a Gorenstein ring
if and only if the algebra $T^*=\bigoplus_{i=0}^{d}T^i$, where $T^i=\mathrm{Tor}_{\kk[m]}^{-i}(\kk(K),\kk)$,
is a Poincar\'e algebra.
%The isomorphisms in the definition of Poincar\'e algebra are
%\[\mathcal {H}^{i}(K)\xrightarrow{\cong}\mathrm{Hom}_\kk(\mathcal {H}^{n-i}(K), \mathcal {H}^{n}(K)).\]
\end{thm}

A theorem due to Stanley gives a purely topological characterisation of Gorenstien* complexes.

\begin{thm}[{\cite[Theorem 5.1]{S96}}]\label{thm:3}
A simplicial complex $K$ is Gorenstein* over a field $\kk$ if and only if for any simplex $\sigma\in K$ (including $\sigma=\varnothing$)
\[\w H_i(\mathrm{link}_K\sigma;\kk)=
\begin{cases}
\kk&\quad\text{if }\  i=\mathrm{dim(link}_K\sigma);\\
0&\quad\text{otherwise,}
\end{cases}
\]
or equivalently, $|K|$ is a GHS over $\kk$.
\end{thm}

The proof of Stanley's theorem is based on the calculation of the \emph{local cohomology} of face rings, which was first given by Hochster for general simplicial complexes. Here we will give a more topological proof of Theorem \ref{thm:3} by using the connection between Poincar\'e duality of moment-angle manifolds and Alexander duality of GHS's.

Notice that a $\kk$-orientable homology manifold over $\kk$ clearly satisfies Poincar\'e duality with coefficients in $\kk$.
Since the proof of Alexander duality is a standard combination of Poincar\'e duality and excision, together with the homology information of $S^n$, then it can be naturally generalized to GHS's. Moreover, the Alexander duality has a more explicit description for the triangulation of a GHS:

\begin{thm}[Alexander duality]\label{thm:dual}
Let $K$ be a simplicial complex with vertex set $[m]$. Then $|K|$ is a GHS over $\kk$ of dimension $n$ if and only if
\[\w H^i(K_I;\kk)\cong \w H_{n-i-1}(K_{[m]\setminus I};\kk)\quad\text{for all } I\subset[m],\ -1\leq i\leq n.\]
Moreover, if the above isomorphisms hold, then they are induced by the excision product map:
\[D:\w H^i(K_I;\kk)\xrightarrow{[K]\sqcap}\w H_{n-i-1}(K_{[m]\setminus I};\kk),\]
where $[K]\in\w H_n(K;\kk)$ is a fundamental class for $K$.
\end{thm}

\begin{proof}
The `only if' part of the  theorem follows from the Alexander duality and the fact that the homotopy equivalence $|K|-|K_I|\simeq|K_{[m]\setminus I}|$ holds for any simplicial complex $K$ on $[m]$ and any $I\subset[m]$.

For the `if' part, first we show by induction on the vertex number of $\sigma$ that if $\sigma\cup I\in K$ for some $I\neq\varnothing$ with $I\cap\sigma=\varnothing$ then $\w H_*(\mathrm{link}_{K_{\wh I}}\sigma)=0$, where $\wh I=[m]\setminus I$.  The case that $\sigma=\varnothing$ is true since $\w H_*(\mathrm{link}_{K_{\wh I}}\varnothing)=\w H_*(K_{\wh I})=\w H^*(K_{I})=0$ by duality in the theorem.
For the induction step, suppose $\sigma=(v_1,\dots,v_k),\sigma'=(v_2,\dots,v_k)$ and $I'=I\cup\{v_1\}$, then it is easy to verify that
\[\mathrm{link}_{K_{\wh I}}\sigma'=\mathrm{link}_{K_{\wh I'}}{\sigma'}\cup_{\mathrm{link}_{K_{\wh I}}\sigma}(\{v_1\}*\mathrm{link}_{K_{\wh I}}\sigma).\]
Consider the Mayer-Vietoris sequence
\begin{align*}
\cdots\xrightarrow{\partial}\w H_i(\mathrm{link}_{K_{\wh I}}\sigma)\to\w H_i(\mathrm{link}_{K_{\wh I'}}\sigma')\oplus\w H_i(\{v_1\}*\mathrm{link}_{K_{\wh I}}\sigma)&\to\w H_i(\mathrm{link}_{K_{\wh I}}\sigma')\xrightarrow{\partial}\\
&\w H_{i-1}(\mathrm{link}_{K_{\wh I}}\sigma)\to\cdots\
\end{align*}
Since $\w H_*(\{v_1\}*\mathrm{link}_{K'}\sigma)=0$, and $\w H_*(\mathrm{link}_{K_{\wh I'}}\sigma')=H_*(\mathrm{link}_{K_{\wh I}}\sigma')=0$ ($\sigma'\cup I', \sigma'\cup I\in K$) by induction, then $\w H_*(\mathrm{link}_{K'}\sigma)=0$, finishing the induction step.
Next, we show that $\w H_*(\mathrm{link}_K\sigma;\kk)\cong\w H_*(S^{n-|\sigma|};\kk)$ for each $\sigma\in K$ by induction on the vertex number of $\sigma\in K$. The case $\sigma=\varnothing$ is clearly true. For the induction step, suppose $\sigma=(v_1,\dots,v_k),\sigma'=(v_2,\dots,v_k)$ and $K'=K_{[m]\setminus\{v_1\}}$ then as above
\[\mathrm{link}_K\sigma'=\mathrm{link}_{K'}{\sigma'}\cup_{\mathrm{link}_K\sigma}(\{v_1\}*\mathrm{link}_K\sigma).\]
Now consider the Mayer-Vietoris sequence
\begin{align*}
\cdots\xrightarrow{\partial}\w H_i(\mathrm{link}_K\sigma)\to\w H_i(\mathrm{link}_{K'}\sigma')\oplus\w H_i(\{v_1\}*\mathrm{link}_K\sigma)&\to\w H_i(\mathrm{link}_K\sigma')\xrightarrow{\partial}\\
&\w H_{i-1}(\mathrm{link}_K\sigma)\to\cdots\
\end{align*}
By induction, $\w H_*(\mathrm{link}_K\sigma')=\w H_{n-|\sigma'|}(\mathrm{link}_K\sigma')=\kk$. $\w H_*(\{v_1\}*\mathrm{link}_{K}\sigma)=0$ is clear. The preceding argument shows that $\w H_*(\mathrm{link}_{K'}\sigma')=0$. Thus the exactness of the Mayer-Vietoris sequence implies that $\w H_*(\mathrm{link}_K\sigma)\cong\w H_*(S^{n-|\sigma|})$, finishing the induction step. It remains to show that $K$ is pure of dimension $n$ (this implies $\mathrm{dim\,link}_K\sigma=n-|\sigma|$). To see this, let $\sigma$ be a facet, then $\mathrm{link}_K\sigma=\varnothing$, so $\w H_*(\varnothing)=\w H_*(S^{n-|\sigma|})$ gives the desired result dim\,$\sigma=n$.

To show that the Alexander duality is induced by the excision product, we first note that when the Alexander duality holds, $|K|$ is a GHS over $\kk$.
Hence from Theorem \ref{thm:mfd}
(it can be extended in an obvious fashion to the case that $|K|$ is a GHS over $\kk$), we have that $\ZZ_K$ is a homology $(m+n+1)$-manifold over $\kk$. Let $[\ZZ_K]$ be a fundamental class of $\ZZ_K$ over $\kk$. Then by Poincar\'e duality, there are isomorphisms
\[H^i(\ZZ_K;\kk)\xrightarrow{[\ZZ_K]\f}H_{m+n+1-i}(\ZZ_K;\kk),\quad 0\leq i\leq m+n+1.\]
Since the cap product for $\ZZ_k$ corresponds (up to sign) to the excision product for $K$, and the top homology class $[\ZZ_k]$ corresponds to the top homology class $[K]$, this is equivalent to saying  that the maps
\[D:\w H^i(K_I;\kk)\xrightarrow{[K]\sqcap}\w H_{n-i-1}(K_{[m]\setminus I};\kk),\quad I\subset[m],\ -1\leq i\leq n.\]
are isomorphisms. The proof is finished.
\end{proof}
Actually Theorem \ref{thm:dual} gives another equivalent condition  for $K$ to be Gorenstein*. Meanwhile we can  make use of it together with Theorem \ref{thm:AG} to give a topological proof for Theorem \ref{thm:3}:

\begin{proof}[Proof of Theorem \ref{thm:3}]
We may assume the vertex number of $K$ is $m$ and dim\,$K=n-1$.
For the `if' part, note that $\mathrm{dim}_\kk\,T^{m-n}\geq1$ by Theorem \ref{thm:1}. We assert that $\mathrm{dim}_\kk\,T^{m-n}=1$ and $T^i=0$ for $i>m-n$. Otherwise, there exists $I\subsetneq[m]$ and $i\geq m-n$ such that $\w H^{|I|-i-1}(K_I)\neq0$. Hence from Alexander duality of $K$, $\w H_{n+i-|I|-1}(K_{[m]\setminus I})\neq0$. Since $i\geq m-n$, we have $n+i-|I|-1\geq m-|I|-1$.
However a simplicial complex with $m-|I|$ vertices is either a simplex or of dimension less than $m-|I|-1$. Therefore, $\w H_{n+i-|I|-1}(K_{[m]\setminus I})=0$, a contradiction. On the other hand, Theorem \ref{thm:dual} says that the Alexander dual of $K$ is induced by the excision product map $[K]\sqcap$. It is equivalent to saying that the bilinear form $T^i\times T^{m-n-i}\to T^{m-n}$ is nondegenerate because of the relation between union and excision product. Thus $T^*=\bigoplus_{i=0}^{m-n}T^i$ is a Poincar\'e algebra. Theorem \ref{thm:AG} gives the desired result.

To prove the `only if' part, first note that $T^{d}=\kk$ for some $d>0$ by Theorem \ref{thm:AG}. We assert that $T^{d}=\w H^{m-d-1}(K)$.
If this is not true, then $T^{d}=\w H^{|I|-d-1}(K_{I})$ for some $I\neq[m]$ by Theorem \ref{thm:1}.
This implies that $\w H^*(K_J)=0$ for any $J\not\subset I$ since the bilinear form $T^i\times T^{d-i}\to T^{d}$ is nondegenerate.
It turns out that any missing face of $K$ belongs to $I$, and it is equivalent to saying that $K=\Delta^{m-|I|-1}*K_I$, which is contradict to the assumption that $K$ is not a cone. Now, by the Poincar\'e duality of $T^*$ and the isomorphism $T^*\cong \HH^*(K;\kk)$, we immediately have that $\w H^{i}(K_{I};\kk)\cong\w H^{m-d-i-2}(K_{[m]\setminus I};\kk)$ for $-1\leq i\leq m-d-1,\,I\subset[m]$.
On the other hand, $\w H^*(K_{[m]\setminus I};\kk)\cong\w H_*(K_{[m]\setminus I};\kk)$ for the assumption that $\kk$ is a field. Thus it follows from Theorem \ref{thm:dual} that $|K|$ is a GHS over $\kk$ of dimension $m-d-1$ (so $n=d-m-1$), finishing the proof.
\end{proof}

For convention, we simply call a simplicial complex to be \emph{Gorenstein*} if the formula in Theorem \ref{thm:3} holds for $\mathbb{Z}$ coefficient. The following two results are easily deduced from the topological characterisation of Gorenstien* complexes.

\begin{prop}[{\cite[Theorem 6.1]{N05}}]\label{thm:6}
If a simplicial complex $K$ is Gorenstein* over $\kk$, then $K$ is pure, and for any simplex $\sigma\in K$, $\mathrm{link}_K\sigma$ is also
Gorenstein* over $\kk$.
\end{prop}

\begin{prop}[{\cite[Corollary 4.8]{MM02}}]\label{thm:2}
Let $K_1$ and $K_2$ be two Gorenstein* complexes over $\kk$ with the same dimension. Then $K_1\#K_2$ is also Gorenstein* over $\kk$.
\end{prop}

\section{$H^*(\mathcal {Z}_{K_1\#K_2})$ for Gorenstein* complexes}\label{sec:3}
As we have seen in Definition \ref{def:1}, for two pure simplicial complexes $K_1$ and $K_2$ with the same dimension, the combinatorial equivalent classes in $\CC(K_1\#K_2)$ are not unique in general. Nevertheless, if $K_1$ and $K_2$ are both Gorenstein*, we can give an universal formula for $H^*(\ZZ_{K_1\#K_2})$ in terms of $H^*(\ZZ_{K_1})$ and $H^*(\ZZ_{K_2})$. Firstly, we introduce an algebraic operation on graded algebras.
\begin{cons}\label{cons:1}
For a finitely graded commutative $\kk$-algebra $A=\bigoplus_{i=0}^d A^i$, define a operation $G$ on $A$, called the \emph{algebraic gyration of $A$}, as
\[G(A)=(A\otimes \Lambda_\kk[v])/(A^d\otimes 1),\ \mathrm{deg}(v)=1.\]
\end{cons}
%Note that if $A$ has no identity, then $G(A)$ is not the same as $A/A^d\otimes\Lambda_\kk[v]$.
Let $G^m$ denote the composition of $m$ $G$'s. Then it is easily verified that
\[G^m(A)=(A\otimes\Lambda_\kk[v_1,\dots,v_{m}])/(\bigoplus_{I\neq[m]}A^d\otimes v_I),\]
where $v_I$ is a simplified notation for $\prod_{i\in I}v_i$.

\begin{thm}\label{thm:4}

Let $K_1$ and $K_2$ be two $(n-1)$-dimensional ($n\geq 2$) Grenstein* complex over $\kk$ with $m_1$ and $m_2$ vertices respectively,
and let $K=K_1\#K_2$. Then the reduced cohomology ring of $\mathcal {Z}_{K}$ is given by the isomorphism
\[\w H^*(\mathcal {Z}_{K};\kk)\cong \mathcal {R}(K_1,K_2;\kk)/\mathcal {I}(K_1,K_2;\kk),\]
where
\[
\mathcal {R}(K_1,K_2;\kk)=G^{m_2-n}(\w {H}^*(\mathcal {Z}_{K_1};\kk))\times G^{m_1-n}(\w {H}^*(\mathcal {Z}_{K_2};\kk))\times\w {H}^*(M;\kk),
\]
\[M=\overset{m_1+m_2-2n}{\underset{i=2}{\sharp}}\lambda(i)(S^{i+1}\times S^{m_1+m_2-i-1}),\]
\[\lambda(i)=\tbinom{m_1+m_2-2n}{i}-\tbinom{m_1-n}{i}-\tbinom{m_2-n}{i},\]
$\mathcal {I}(K_1,K_2;\kk)$ is an ideal of $\mathcal {R}(K_1,K_2;\kk)$ generated by
\[([Z_1],0,0)-(0,0,[M])\, \text{ and }\, (0,[Z_2],0)-(0,0,[M]),\] where $[Z_1]$ (resp. $[Z_2]$, $[M]$) is the top dimensional generator of
 $G^{m_2-n}(\w {H}^*(\mathcal {Z}_{K_1};\kk))$ (resp. $G^{m_1-n}(\w {H}^*(\mathcal {Z}_{K_2};\kk))$, $\w {H}^*(M;\kk)$).
\end{thm}
Before proving this main theorem, we first give several more general results about the cohomology of moment-angle complexes.
\begin{prop}\label{prop:1}
Suppose we are given two simplices $\Delta^{m_1-1}$ and $\Delta^{m_2-1}$. Let $L=\Delta^{m_1-1}\cup_{\sigma} \Delta^{m_2-1}$ denote the simplicial complex obtained
from $\Delta^{m_1-1}$ and $\Delta^{m_2-1}$ by gluing along a common simplex $\sigma$ of dimension $n-1$. Then
\[\w {H}^*(\mathcal {Z}_{L};\kk)\cong \w H^*(\bigvee_{i\geq2} \lambda_i\, S^{i+1};\kk),\]
where
\[\lambda(i)=\tbinom{m_1+m_2-2n}{i}-\tbinom{m_1-n}{i}-\tbinom{m_2-n}{i}.\]
\end{prop}
\begin{proof}
Set $\Ss_1$ and $\Ss_2$ are the vertex sets of $\Delta^{m_1-1}$ and $\Delta^{m_2-1}$ respectively.
From the construction of $L$ we get that
\[
\w {H}^i(L_I;\kk)=
\begin{cases}
\kk& \ \text{if} \ i=0\ \text{and}\ I\cap\sigma=\varnothing,\,I\cap \mathcal {S}_1\neq \varnothing,\,I\cap \mathcal {S}_2\neq \varnothing,\\
0& \ \text{otherwise.}
\end{cases}
\]
So the additive isomorphism follows from theorem \ref{thm:1} immediately. To verify the ring structure,
given $[u]\in \w {H}^0(L_{I_1})$ and $[v]\in \w {H}^{0}(L_{I_2})$ with $I=I_1\cup I_2$, $I_1\cap I_2=\varnothing$,
then $[u]*[v]\in \w {H}^1(L_I)=0$. This proves the proposition.
\end{proof}

\begin{prop}\label{prop:2}
Let $K_1$ be a simplicial complex with $m_1$ vertices, and let $L_1=K_1\cup_\sigma\Delta^{m_2-1}$ denote the simplicial complex obtained from $K_1$ and $\Delta^{m_2-1}$ by gluing along a common simplex $\sigma$, and let $L=\Delta^{m_1-1}\cup_{\sigma} \Delta^{m_2-1}$. Then
\[\w {H}^*(\mathcal {Z}_{L_1};\kk)\cong (\w {H}^*(\mathcal {Z}_{K_1};\kk)\otimes\Lambda_\kk[v_1,\dots,v_{m_2-|\sigma|}])
\times\w {H}^*(\mathcal {Z}_{L};\kk)\]
where $\Lambda_\kk[v_1,\dots,v_{m_2-|\sigma|}]$ denotes the graded exterior algebra over $\kk$ with $m_2-|\sigma|$ generators;
$\mathrm{deg}(v_i)=1$ for $1\leq i\leq m_2-|\sigma|$.
\end{prop}

\begin{proof}
The coefficient ring $\kk$ will play no special role in the argument so we shall omit it from the notation.
Suppose $m=m_1+m_2-|\sigma|$. Set
\begin{align*}
[m_1]&=\{u_1,\dots,u_{m_1-|\sigma|},w_1,\dots,w_{|\sigma|}\},\\
[m_2]&=\{v_1,\dots,v_{m_2-|\sigma|},w_1,\dots,w_{|\sigma|}\}, \text{ and}\\
[m]&=\{u_1,\dots,u_{m_1-|\sigma|},w_1,\dots,w_{|\sigma|},v_1,\dots,v_{m_2-|\sigma|}\}
\end{align*}
be the vertex sets of $K_1$, $\Delta^{m_2-1}$ and $L_1$ respectively.
Consider simplicial inclusions: $K_1\stackrel{i_1}{\hookrightarrow} L_1\stackrel{i_2}{\hookrightarrow} L$, which induce inclusions
$\mathcal{Z}_{K_1}^{[m]}\stackrel{\phi_1}{\hookrightarrow}\mathcal{Z}_{L_1}\stackrel{\phi_2}{\hookrightarrow}\mathcal {Z}_{L}$.
By Theorem \ref{thm:0}, we have a commutative diagram

\[\begin{CD}
\w H^*(\mathcal{Z}_{L})@>\phi_2^*  >>\w H^*(\mathcal{Z}_{L_1})@>\phi_1^*  >>\w H^*(\mathcal {Z}_{K_1}^{[m]})\\
@V\cong VV @V\cong VV @V\cong VV\\
\w {\mathcal {H}}^*(L)@>i_2^*>>
\w {\mathcal {H}}^*(L_1)@>i_1^* >>
\w {\mathcal {H}}^*_{[m]}(K_1)
\end{CD}\]

From the construction of $L$ and $L_1$ it is easy to see that for any $I\subseteq [m]$, $I\neq\varnothing$
\begin{equation}\label{eq:2}
\begin{split}
(L_1)_I&\cong L_I\simeq pt,\ \text{ if }I\cap [m_1]=\varnothing,\\
(L_1)_I&\simeq(K_1)_I\vee  L_I,\ \text{ if }I\cap [m_1]\neq \varnothing,
\end{split}
\end{equation}
and in the second case we have a commutative diagram
\[\begin{CD}
 (K_1)_I@>(i_1)_I  >>(L_1)_I@>(i_2)_I>>L_I\\
@| @V\simeq  VV @|\\
 (K_1)_I@>j_I>>(K_1)_I\vee  L_I@>p_I>>L_I
\end{CD}\]
where $j_I$ and $p_I$ are the natural inclusion and projection maps respectively.
Therefore $\w {\mathcal {H}}^*(L_1)$ additively splits as $A_1\oplus A_2$, where
\[A_1\cong i^*_1(A_1)= \mathrm{Im}\,i_1^*;\quad A_2=\mathrm{Im}\,i_2^*=\mathrm{Ker}\,i^*_1\cong \w {\mathcal {H}}^*(L).\]
%(Note that if $I\cap [m_1]=\varnothing$, then $\w {H}^*({(K_1)}_I)=\w {H}^{-1}(\varnothing)=\kk$, but $\w {H}^*({(L_1)}_I)=0$, so $i_1^*$ is not surjective).
Actually, this splitting is also multiplicative. Since $A_2$ is an ideal of $\w {\mathcal {H}}^*(L_1)$,
we need only to verify that $A_1$ is a subalgebra and for any
$\alpha_1\in A_1$ and $\alpha_2\in A_2$, $\alpha_1*\alpha_2=0$.
Given two elements $\alpha$, $\alpha'\in A_1$, then $\alpha*\alpha'\in \bigoplus_{i\geq 2}\w {\mathcal {H}}^i(L_1)$.
However, from Proposition \ref{prop:1}, we have $\w {\mathcal {H}}^*(L)=\w {\mathcal {H}}^1(L)$, so $\alpha*\alpha'\in A_1$. $A_1$ is an algebra.
Meanwhile, for any $\alpha_1\in A_1$ and $\alpha_2\in A_2$, $\alpha_1*\alpha_2=0\in A_2$ for the same dimension reason. Thus $\w {\mathcal {H}}^*(L_1)=A_1\times A_2$.

%Since $(K_1)_I=\varnothing$ for $I\cap[m]=\varnothing$,
According to Theorem \ref{thm:e}, a straightforward calculation shows that
\begin{enumerate}[(1)]
\item In $\HH^*_{[m]}(K_1)$, the subalgebra $\underset{I\cap[m_1]=\varnothing}\bigoplus \w {H}^*({(K_1)}_I)\cong \Lambda_\kk [v_1,\dots,v_{m_2-|\sigma|}$.
Each generator $v_i$ of the exterior algebra corresponds to a generator of
$\w H^{-1}((K_1)_{\{v_i\}})=\w H^{-1}(\varnothing)=\kk$.
\item For $I=I_1\cup I_2$ with $I_1\subseteq  [m_1]$, $I_2\cap [m_1]=\varnothing$, the homomorphism
\begin{align*}
\w {H}^*((K_1)_{I_1})\otimes\w {H}^*((K_1)_{I_2})&\to \w {H}^*((K_1)_I)\\
\alpha\otimes\beta&\mapsto\alpha*\beta
\end{align*}
is an isomorphism.
\end{enumerate}
Therefore we have an isomorphism
\[\bigoplus_{I\subseteq [m]} \w {H}^*({(K_1)}_I)\cong \big(\bigoplus_{I\subseteq [m_1]} \w {H}^*({(K_1)}_I)\big)\otimes
\big(\bigoplus_{I\cap[m_1]=\varnothing} \w {H}^*({(K_1)}_I)\big)\]
From formula (\ref{eq:2}), it is easy to see that
\begin{align*}
\mathrm{Im}\,i_1^*=\bigoplus_{\substack{I\cap[m_1]\neq\varnothing}} \w {H}^*({(K_1)}_I)\cong\big(\bigoplus_{\substack{I\subseteq [m_1]\\I\neq\varnothing}} \w {H}^*({(K_1)}_I)\big)
\otimes\big(\bigoplus_{I\cap[m_1]=\varnothing} \w {H}^*({(K_1)}_I)\big)
\end{align*}
Then the proposition follows at once.
\end{proof}

\begin{prop}\label{prop:3}
Let $K_1$ and $K_2$ be simplicial complexes with $m_1$ and $m_2$ vertices respectively. Let $K=K_1\cup_\sigma K_2$ denote a simplicial complex obtained from
$K_1$ and $K_2$ by gluing along a common simplex $\sigma$, and let $L=\Delta^{m_1-1}\cup_{\sigma} \Delta^{m_2-1}$. Then
\begin{align*}
\w {H}^*(\mathcal {Z}_{K};\kk)\cong &(\w {H}^*(\mathcal {Z}_{K_1};\kk)\otimes\Lambda_\kk[v_1,\dots,v_{m_2-|\sigma|}])\\
&\times(\w {H}^*(\mathcal {Z}_{K_2};\kk)\otimes\Lambda_\kk[u_1,\dots,u_{m_1-|\sigma|}])\times\w {H}^*(\mathcal {Z}_{L};\kk)
\end{align*}
\end{prop}
\begin{proof}
We use the same notations $[m_1]$, $[m_2]$ and $[m]$ as in the proof of Proposition \ref{prop:2} to denote the vertex sets of
$K_1$, $K_2$ and $K$ respectively. Let $L_1=K_1\cup_\sigma\Delta^{m_2-1}$, $L_2=\Delta^{m_1-1}\cup_\sigma K_2$,
then we have a commutative diagram of simplicial inclusions
\[
\xymatrix{K_1 \ar[rd]^{i_1}&&L_2\ar[rd]^{l_2}\\
&K\ar[ru]^{j_2}\ar[rd]^{j_1}&&L\\
K_2\ar[ru]^{i_2}&&L_1\ar[ru]^{l_1}}
\]
which induces a commutative diagram of algebras
\[
\xymatrix{\w{\mathcal {H}}_{[m]}^*(K_1) &&\w{\mathcal {H}}^*(L_2)\ar[ld]_{j_2^*}\\
&\w{\mathcal {H}}^*(K)\ar[lu]_{i_1^*}\ar[ld]_{i_2^*}&&\w{\mathcal {H}}^*(L)\ar[lu]_{l_2^*}\ar[ld]_{l_1^*}\\
\w{\mathcal {H}}_{[m]}^*(K_2)&&\w{\mathcal {H}}^*(L_1)\ar[lu]_{j_1^*}}
\]
By the construction of these simplicial complexes,
for any $I\subseteq [m]$ we have
\begin{equation}\label{eq:1}
\begin{split}
K_I&\cong(L_1)_I\cong(K_1)_I\ \text{ if } I\cap [m_2]=\varnothing,\\
K_I&\cong(L_2)_I\cong(K_2)_I\ \text{ if } I\cap [m_1]=\varnothing,\\
K_I&\simeq(K_1)_I\vee  (L_2)_I\simeq(K_2)_I\vee(L_1)_I\ \text{ otherwise}.
\end{split}
\end{equation}
So $\w {\mathcal {H}}^*(K)$
additively splits as $A_1\oplus A_2$ (also as $B_1\oplus B_2$), where
\[A_1\cong \mathrm{Im}\,i_1^*=i_1^*(A_1)\ (B_1\cong \mathrm{Im}\,i_2^*=i_2^*(B_1)),\text{ and}\]
\[A_2=\mathrm{Im}\,j_2^*=\mathrm{Ker}\,i_1^*\cong \w {\mathcal {H}}^*(L_2)\
(B_2=\mathrm{Im}\,j_1^*=\mathrm{Ker}\,i_2^*\cong \w {\mathcal {H}}^*(L_1)).\]
From formula (\ref{eq:1}), we also have $\mathrm{Im}\,i_1^*=\mathrm{Im}\,(j_1i_1)^*$. From the proof of Proposition \ref{prop:2},
we know that $\w {\mathcal {H}}^*(L_1)$ multiplicatively splits as $A_1'\oplus A_2'$, where
\[A_1'\cong\mathrm{Im}\,(j_1i_1)^*\cong \w {H}^*(\mathcal {Z}_{K_1})\otimes\Lambda_\kk [v_1,\dots,v_{m_2-|\sigma|}],\]
\[A_2'=\mathrm{Im}\,l^*_1\cong \w {\mathcal {H}}^*(L).\]
So we can take $A_1=j_1^*(A_1')$.

Now we will show that the  splitting $\w {\mathcal {H}}^*(K)=A_1\oplus A_2$ is a multiplicative splitting.
It is clear that $A_1$ is a subalgebra and $A_2$ is an ideal.
For any $\alpha_1\in A_1$,  $\alpha_2\in A_2$, $\alpha_1*\alpha_2\in A_2$ and
$i_2^*(\alpha_1*\alpha_2)=i_2^*(\alpha_1)*i_2^*(\alpha_2)=0$ ($A_1\subseteq \mathrm{Im}\,j_1^*=\mathrm{Ker}\,i_2^*$).
Thus $\alpha_1*\alpha_2\in \mathrm{Ker}\,i^*_2=\mathrm{Im}\,j_1^*$.
Note that $\mathrm{Im}\,j_1^*=j_1^*(A_1')\oplus j_1^*(A_2')=A_1\oplus j_1^*(A_2')$. So $\alpha_1*\alpha_2\in j_1^*(A_2')$.
By Proposition \ref{prop:1}, the elements of $A_2'$ all come from $\w {\mathcal {H}}^1(L_1)$, but $\alpha_1*\alpha_2\in \bigoplus_{i\geq 2}\w {\mathcal {H}}^i(K)$. It follows that $\alpha_1*\alpha_2=0$.
Therefore the above splitting is multiplicative, and the result follows.
\end{proof}

Now let us complete the proof of Theorem \ref{thm:4}.

\begin{proof}[Proof of Theorem \ref{thm:4}]
The cohomology groups with coefficients in $\kk$ will be implicit throughout the proof. We use the notation
$[m_1]$, $[m_2]$ and $[m]$, as in the proof of Proposition \ref{prop:3}, to denote the vertex sets of $K_1$, $K_2$ and $K$ respectively.
Suppose $\sigma$ is the selected common facet of $K_1$ and $K_2$ corresponding to $K$. Set $K'=K_1\cup_\sigma K_2$,
then there is a simplicial inclusion $i:K\hookrightarrow K'$, which induces a inclusion $\phi:\mathcal {Z}_{K}\hookrightarrow \mathcal {Z}_{K'}$. Consider the homomorphisms of cohomology groups
\[i_I^*: \w H^j(K'_I)\to\w H^j(K_I),\ I\subseteq [m],\ -1\leq j\leq n-1.\]
We will analyze $i^*_I$ in five cases:
\begin{enumerate}[(i)]
\item $\sigma\nsubseteq I$. In this case $K'_I=K_I$, so $i_I^*$ is an isomorphism for $-1\leq j\leq n-1$.

\item \label{case:2}$\sigma\subseteq I$ and $[m_1]\nsubseteq I,\,[m_2]\nsubseteq I$.
In this case $K_I$ is a proper subcomplex of $K'_I$, and $K'_I\simeq (K_1)_I\vee (K_2)_I$.
Consider the long exact sequence of cohomology groups
\[\cdots\to \w H^{j-1}(K_I)\xrightarrow{\delta}\w H^j(K'_I,K_I)\to\w H^j(K'_I)\xrightarrow{i_I^*}\w H^j(K_I)\xrightarrow{\delta}\cdots\]
By excision theorem
\[\w H^j(K'_I,K_I)\cong\w H^j(\Delta^{n-1},\partial\Delta^{n-1})=
\begin{cases}
\kk\ \ \text{ if } j=n-1;\\
0\ \ \text{ otherwise.}
\end{cases}\]
$\w H^{n-1}(K'_I)=\w H^{n-1}(K_I)=0$ by Theorem \ref{thm:dual}. Therefore from the long exact sequence above, we have that $i^*_I$ is an isomorphism
for $j\neq n-2$ and $\w H^{n-2}(K_I)$ splits as $\mathrm{Im}\,i_I^*\oplus \kk$.

\item \label{case:3}$[m_1]\subseteq I,\,[m_2]\nsubseteq I$.
The only difference to (\ref{case:2}) is that $\w H^{n-1}(K'_I)\cong \kk$ in this case.
A similar analysis shows that $i^*_I$ is an isomorphism for $j<n-1$ and trivial for $j=n-1$.
From the proof of Proposition \ref{prop:3}, a generator of $\w H^{n-1}(K'_I)\cong \kk$ corresponds to
$[\mathcal {Z}_{K_1}]\otimes v_{I\setminus[m_1]}$ in the formula of Proposition \ref{prop:1},
where $[\mathcal {Z}_{K_1}]$ is a generator of $\w {H}^{m_1+n}(\mathcal {Z}_{K_1})\cong \kk$.

\item$[m_2]\subseteq I,\,[m_1]\nsubseteq I$.  $i^*_I$ is exactly the same as in (\ref{case:3}).
In this case a generator of $\w H^{n-1}(K'_I)$ corresponds to
$[\mathcal {Z}_{K_2}]\otimes u_{I\setminus[m_2]}$ in the formula of Proposition \ref{prop:1},
where $[\mathcal {Z}_{K_2}]$ is a generator of $\w {H}^{m_2+n}(\mathcal {Z}_{K_2})\cong \kk$.

\item \label{case:5}$I=[m]$.
In this case
\[\w H^{n-1}(K'_I)=\w H^{n-1}(K')\cong \kk\oplus \kk,\ \w H^{n-1}(K_I)=\w H^{n-1}(K)\cong \kk,\]
\[\w H^{j}(K'_I)=\w H^{j}(K_I)=0\ \text{ for }\ j<n-1.\]
Let $\xi_1$ and $\xi_2$ be two generators of $\w H^{n-1}(K')$ corresponding respectively to $[\mathcal {Z}_{K_1}]\otimes v_{[m_2-n]}$ and
$[\mathcal {Z}_{K_2}]\otimes u_{[m_1-n]}$ in the formula of Proposition \ref{prop:1}. It is easy to see that
$i^*(\xi_1)=i^*(\xi_2)$.
\end{enumerate}

Combining the above arguments, we have that as $\kk$-module
\[\w H^*(\mathcal {Z}_K)=\phi^*(\w H^*(\mathcal {Z}_{K'}))\oplus B,\]
where
\[B\cong \w H^*(\bigvee_{I\subseteq \mathcal {M}} S_I^{|I|+n-1}),\quad \mathcal {M}=\{I\in [m]:\sigma\subseteq I,\,[m_1]\nsubseteq I,\,[m_2]\nsubseteq I\}.\]
Each sphere summand $S_I^{|I|+n-1}$ above corresponds to a $\kk$ direct summand in $\w H^{n-2}(K_I)$ (see case \eqref{case:2}).
Denote by $\beta_I$ a generator of this $\kk$ direct summand.
On the other hand, by Proposition \ref{prop:3} $\w H^*(\mathcal {Z}_{K'})= A_1\times A_2\times A$ (as algebra), where
\[A_1\cong \w H^*(\mathcal {Z}_{K_1})\otimes\Lambda_\kk[v_1,\dots,v_{m_2-|\sigma|}];\quad
A_2\cong \w {H}^*(\mathcal {Z}_{K_2})\otimes\Lambda_\kk[u_1,\dots,u_{m_1-|\sigma|}];\]
\[A\cong\w {H}^*(\bigvee_{J\in \mathcal {N}} S_J^{|J|+1}),\quad \mathcal {N}=\{I\subseteq [m]:I\cap\sigma=\varnothing,\,I\cap[m_1]\neq\varnothing,\,I\cap[m_2]\neq\varnothing\}.\]
Each sphere summand $S_J^{|J|+1}$ above corresponds to a $\kk$ direct summand in $\w H^0(K'_J)$,
denote by $\alpha_J$ a generator of this $\kk$ direct summand. It is straightforward to see that $\mathcal {M}$ and $\mathcal {N}$ are in one-to-one correspondence\,: $I\mapsto[m]\setminus I$. The previous arguments imply that
$\phi^*(A)\cong A$,
and $\phi^*(A_1\times A_2)\cong (A_1\times A_2)/\mathcal {I}$ where $\mathcal {I}$ is a ideal generated by
\[([\mathcal {Z}_{K_1}]\otimes\alpha,0)\  \text{ for } \alpha\in\bigoplus_{i=0}^{m_2-n-1} \Lambda^i_\kk[v_1,\dots,v_{m_2-n}]\ \text{ and}\]
\[(0,[\mathcal {Z}_{K_2}]\otimes\beta)\ \text{ for } \beta\in \bigoplus_{j=0}^{m_1-n-1}\Lambda^j_\kk[u_1,\dots,u_{m_1-n}]\ \text{ and}\]
\[([\mathcal {Z}_{K_1}]\otimes v_{[m_2-n]},0)-(0,[\mathcal {Z}_{K_2}]\otimes u_{[m_1-n]}).\]

An easy observation shows that the isomorphism in the theorem holds for $\kk$-module homomorphism.
Now let us complete the proof by verifying the ring structure of $\w H^*(\mathcal {Z}_K)$. $\phi^*(\w {H}^*(\mathcal {Z}_{K'}))$ is clearly an algebra. For any two
generators $\beta_{I},\beta_{I'}\in B$, $\varnothing\neq\sigma\subset I\cap I'$, so $\beta_{I}*\beta_{I'}=0$ in $\w {\mathcal {H}}^*(K)$.
and then $B$ is an algebra with trivial product structure.  Thus we need only to verify the multiplication between $B$ and $\phi^*(\w {H}^*(\mathcal {Z}_{K'}))=i^*(\w {\mathcal {H}}^*(K'))$.

If $\kk$ is a field, then by Proposition \ref{thm:2}
and Theorem \ref{thm:AG}, $\mathcal {H}^{*}(K)$ is a Poincar\'e $\kk$-algebra.
First we assert that $\beta_I\in B$ can be chosen properly such that
$\beta_I*i^*(\alpha)=0$ for any $\alpha\in A_1\times A_2$.
%(here $A_1\times A_2$ is viewed as a subalgebra of $\w {\mathcal {H}}^*(K')$).
To see this, note first that if $\alpha\in A_1$
such that $\beta_I*i^*(\alpha)\neq 0$, then for the dimension reason, $\alpha\in A_1^1$ (define $A_1^i=A_1\cap \w {\mathcal {H}}^i(K')$). By arguments in case (\ref{case:5}) $A_1^n\cong i^*(A_1^n)=\w {\mathcal {H}}^n(K)$.
Then for a chosen $\beta_I$ there is a $\kk$-homomorphism defined by
\begin{align*}
\phi_{\beta_I}: A_1^1&\to A_1^{n},\\
a&\mapsto (i^*)^{-1}(\beta_I*i^*(a)).
\end{align*}
Since $K_1$ is Gorenstein* over $\kk$,
$\w {\mathcal {H}}^*(K_1)$ (without concerning $\mathcal {H}^0(K_1)$)
is a Poincar\'e algebra, and so is $A_1$. Thus there exists an element $a_1\in A_1^{n-1}$, such that the $\kk$-homomorphism
\begin{align*}
\phi_{a_1}: A_1^1&\to A_1^n,\\
a&\mapsto a_1*a
\end{align*}
is equal to $\phi_{\beta_I}$, and so
\begin{align*}
(\beta_I-i^*(a_1))*i^*(a)&=i^*(i^*)^{-1}(\beta_I*i^*(a))-i^*(a_1*a)\\
&=i^*((i^*)^{-1}(\beta_I*i^*(a))-a_1*a)=0
\end{align*}
for any $a\in A_1$. Similarly there exists $a_2\in A_2^{n-1}$,
such that $(\beta_I-i^*(a_2))*i^*(a)=0$
for any $a\in A_2$.
Replacing $\beta_I$ by $\beta_I-i^*(a_1)-i^*(a_2)$ we get  the desired generator.
So we can make $B*i^*(A_1\times A_2)=0$.

It remains to verify the multiplication between $i^*(A)$ and $B$.
For each generator
$\beta_I\in B\subset \w{\mathcal {H}}^{n-1}(K)$ and $\alpha_J\in A\subset \w{\mathcal {H}}^1(K)$, we have that $\beta_I*i^*(\alpha_J)=0$ if $J\neq [m]\setminus I$.
This is because $\w H^{n-1}(K_S)=0$ if $S\neq [m]$. On the other hand, since $B*i^*(A_1\times A_2)=0$, then by the definition of Poincar\'e algebra $\beta_I*i^*(\alpha_{[m]\setminus I})$ is a generator of $\w {\mathcal {H}}^n(K)=\kk$. Thus by choosing the coefficients of  $\{\beta_I\}$ and
$\{\alpha_J\}$ properly, we have
\[\beta_I*i^*(\alpha_J)=
\begin{cases}
[\mathcal {Z}_{K_1}]\otimes v_{[m_2-n]}=[\mathcal {Z}_{K_2}]\otimes u_{[m_1-n]}\ &\text{ if } I\cup J=[m],\,I\cap J=\varnothing,\\
0&\text{ otherwise. }
\end{cases}
\]
Then the desired algebra isomorphism follows.

If $\kk=\mathbb{Z}$. Write $K^{\{0\}}=K$, $K^{\{1\}}=K_1$ and $K^{\{2\}}=K_2$. If we prove that
\begin{align*}
\varphi: \w {\mathcal {H}}^{n-1}(K^{\{i\}})&\to\mathrm{Hom}(\w {\mathcal {H}}^1(K^{\{i\}}),\w {\mathcal {H}}^n(K^{\{i\}}))\\
a&\mapsto\phi_a, \quad\quad \phi_a(b)=a*b
\end{align*}
is an isomorphism for $i=0,1,2$, then the result can be proved in the same way as the field case.
This is clearly true. Since when $K^{\{i\}}$ is Gorenstein*, $\ZZ_{K^{\{i\}}}$ is a manifold, thus satisfies Poincar\'e duality.
Then $\varphi$ is an isomorphism after factoring out the torsion. Note that $\w {\mathcal {H}}^{n-1}(K^{\{i\}})$ and $\w {\mathcal {H}}^{1}(K^{\{i\}})$ are always torsion free, the result follows.
\end{proof}

Since when $K_1$ and $K_2$ are Gorenstein*, $\mathcal {Z}_{K_1}$,  $\mathcal {Z}_{K_2}$ and $\mathcal {Z}_{K_1\#K_2}$ are both manifolds,
then we wish to obtain $\mathcal {Z}_{K_1\#K_2}$ from
$\mathcal {Z}_{K_1}$ and $\mathcal {Z}_{K_2}$ by making some geometrical operation on them. First we introduce the following operation on a manifold, which is studied in \cite{GL13}.

\begin{Def}\label{def:2}
Let $M^n$ be an $n$-manifold without boundary, and let $M^n_{-1}$ be $M^n$ minus an open ball $\overset{\circ}{D^n}$. The \emph{gyration} $\mathcal {G}(M^n)$ of $M^n$ is defined to be the manifold
\[\mathcal {G}(M^n)=\partial(M^n_{-1}\times D^2)=M^n_{-1}\times S^1\cup S^{n-1}\times D^2.\]
\end{Def}

Apparently, $\mathcal {G}(M^n)$ is obtained from the manifold $M^n\times S^1$ by an $(1,n)$-type surgery. A straightforward calculation shows the effect of gyration on cohomology.

\begin{prop}\label{prop:5}
Let $M^n$ be a closed $\kk$-orientable manifold of dimension $n$. Then there is a ring isomorohism
\[\w H^*(\mathcal {G}(M^n);\kk)\cong G(\w H^*(M^n;\kk)).\]
\end{prop}

From Theorem \ref{thm:4} and Proposition \ref{prop:5}, we immediately get the following Corollary.

\begin{cor}\label{cor:3}
Let $K_1$ and $K_2$ be two $n-1$ dimensional Gorenstein* complexes with $m_1$ and $m_2$ vertices respectively.
Then $H^*(\mathcal {Z}_{K_1\#K_2};\kk)\cong H^*(M;\kk)$, where
\[M=\mathcal {G}^{m_2-n}(\mathcal {Z}_{K_1})\#\mathcal {G}^{m_1-n}(\mathcal {Z}_{K_2})
\#\overset{m_1+m_2-2n}{\underset{i=2}{\sharp}}\lambda(i)(S^{i+1}\times S^{m_1+m_2-i-1}),\]
\[\lambda(i)=\tbinom{m_1+m_2-2n}{i}-\tbinom{m_1-n}{i}-\tbinom{m_2-n}{i},\]
and $\mathcal {G}^{r}(M)$ means iterating the gyration on a manifold $M$ by $r$ times.
\end{cor}

\begin{rem}
Actually, if one of $K_1$ and $K_2$ is $\partial\Delta^n$ in the above corollary,
then the cohomology ring isomorphism is induced by a homeomorphism of manifolds \cite{CFW16}, i.e., for an $(n-1)$-dimensional Gorenstein* complex $K$ with
$m$ vertices, we have a homeomorphism
\[\mathcal {Z}_{K\#\partial\Delta^n}\approx \mathcal {G}(\mathcal {Z}_{K})\overset{m-n}{\underset{i=1}{\sharp}}\tbinom{m-n}{i}(S^{i+2}\times S^{m+n-i-1}).\]
This is just the conjecture given by S.~Gitler and S.~L\'opez de~Medrano \cite{GL13}.
\end{rem}
For the general case we make the following
\begin{conj}\label{conj:1}
$\mathcal {Z}_{K_1\#K_2}$ is homeomorphic to
\[\mathcal {G}^{m_2-n}(\mathcal {Z}_{K_1})\#\mathcal {G}^{m_1-n}(\mathcal {Z}_{K_2})
\#\overset{m_1+m_2-2n}{\underset{i=2}{\sharp}}\lambda(i)(S^{i+1}\times S^{m_1+m_2-i-1}).\]
\end{conj}

\section{Change of $H^*(\ZZ_K)$ after a stellar subdivision on $K$}\label{sec:4}
\begin{Def}\label{def:4}
Let $\sigma\in K$ be a nonempty simplex of a simplicial complex $K$.
The \emph{stellar subdivision} of $K$ at $\sigma$ is obtained by replacing the star of $\sigma$ by the cone
over its boundary:
\[\mathrm{S}_\sigma K=(K\setminus \mathrm{star}_K\sigma)\cup \big(\mathrm{cone}(\partial\sigma*\mathrm{link}_K\sigma)\big).\]
If $\mathrm{dim}\,\sigma=0$ then $\mathrm{S}_\sigma K=K$. Otherwise the complex $\mathrm{S}_\sigma K$ acquires an additional
vertex (the vertex of the cone).
\end{Def}

\begin{Def}
Two simplicial complexes $K,L$ are called \emph{stellar equivalent}, if there is a sequence
of simplicial complexes $K=K_0,K_1,\dots,K_n=L$ such that $K_{i+1}$ is a stellar
subdivision of $K_i$ or $K_i$ is a stellar
subdivision of $K_{i+1}$ for each $i$.
\end{Def}
A classical result in $PL$ theory says that
\begin{thm}[{\cite[Theorem 4.5]{L99}}]
Two simplicial complexes are piecewise linearly
homeomorphic if and only if they are stellar equivalent.
\end{thm}

Thus if we can find a formula (depends only on dim\,$K$ and dim\,$\sigma$) to describe the relation between $H^*(\mathcal {Z}_K)$ and $H^*(\mathcal {Z}_{\mathrm{S}_\sigma K})$ for any $\sigma\in K$, then we can actually get the cohomology of all moment-angle manifolds associated to $PL$-spheres. However, this is impossible,
since the change of $H^*(\mathcal {Z}_{\mathrm{S}_\sigma K})$ from $H^*(\mathcal {Z}_K)$ does not only depend on $\mathrm{dim}\,\sigma$, but also on the position of $\sigma$ in $K$ in general. Actually, the calculation of $H^*(\mathcal {Z}_{\mathrm{S}_\sigma K})$ is very complicated even if the cohomology $H^*(\mathcal {Z}_K)$ is already known in general cases.
Nevertheless, when $\sigma$ satisfies some local conditions, we can get a simple description of $H^*(\mathcal {Z}_{\mathrm{S}_\sigma K})$ by $H^*(\mathcal {Z}_K)$. That is the following

\begin{thm}\label{thm:5}
Let $K$ be a Gorenstein* complex of dimension $n-1$ with $m$ vertices,
$\sigma\in K$ be a simplex of dimension $q<n-1$. Let $\mathcal {V}$ be the vertex set of $\mathrm{link}_K\sigma$, $s=|\mathcal {V}|+|\sigma|$. If for any $I\subset \mathcal {V}$, the inclusion map
$\varphi_I:(\mathrm{link}_K\sigma)_I\to K_I$ is nullhomotopic, and if one of the following additional conditions are satisfied:
\begin{enumerate}[(a)]
\item $\kk$ is a field.
\item $\kk=\mathbb{Z}$ and $H^*(\mathcal {Z}_{\mathrm{link}_K\sigma})$ is torsion free.
\end{enumerate}
Then the cohomology ring of $\mathcal {Z}_{\mathrm{S}_\sigma K}$ is given by the isomorphism
\[ H^*(\mathcal {Z}_{\mathrm{S}_\sigma K};\kk)\cong H^*\big(\mathcal {G}(\mathcal {Z}_K)\#\, Y;\kk\big),\]
where
\begin{align*}
&Y=\underset{i+j\geq1}{\sharp}f_i\cdot\tbinom{m-s}{j}(S^{i+j+2}\times S^{m+n-i-j-1}),\\
&f_i=\mathrm{rank}\,H^i(\mathcal {Z}_{\mathrm{link}_K\sigma}).
\end{align*}
\end{thm}

\begin{Def}
A simplicial complex $K$ is called \emph{$q$-neighborly} if any $q$ vertices of $K$ span a simplex of $K$.
\end{Def}
If $K$ is a $q$-neighborly simplicial complex, then by definition the $k$-skeleton of $K$ is homotopy equivalent to the wedge of $S^k$'s for $k<q$. Hence $\pi_k(K)=0$
for $k<q-1$.

\begin{cor}\label{cor:1}
Let $K$ be a Gorenstein* complex of dimension $n-1$ with $m$ vertices.
If $\sigma\in K$ ($\mathrm{dim}\,\sigma<n-1$) satisfies $\mathrm{link}_K\sigma\cong\partial\Delta^{n_1}*\cdots*\partial\Delta^{n_k}$,
and $K_{\VV}$ ($\mathcal {V}$ is the vertex set of $\mathrm{link}_K\sigma$) is $q$-neighborly with $q\geq1+\sum_{i=1}^k n_k$,
then $H^*(\mathcal {Z}_{\mathrm{S}_\sigma K})\cong H^*(M)$, where
\begin{align*}
M=&\,\mathcal {G}(\mathcal {Z}_K)\#\overset{m-n-k}{\underset{j=1}{\sharp}}\tbinom{m-n-k}{j}(S^{j+2}\times S^{m+n-j-1})\\
&\#\underset{\substack{I\subset[k],\\I\neq\varnothing}}{\sharp}\overset{m-n-k}{\underset{j=0}{\sharp}}\tbinom{m-n-k}{j}(S^{\lambda_I+j+2}\times S^{m+n-\lambda_I-j-1}),\quad \lambda_I=\sum_{i\in I}(2n_i+1).
\end{align*}
\end{cor}
\begin{proof}
A straightforward verification shows that $|\sigma|+|\VV|=n+k$ and $\mathrm{dim\,link}_\sigma K=n-|\sigma|-1$, so $\mathrm{dim}\,(\mathrm{link}_\sigma K)_I<q-1$ for any $I\subset\VV$. Since $K_I$ is $q$-neighborly, then $\pi_k(K)=0$ for $k<q-1$. Thus the inclusion map $(\mathrm{link}_K\sigma)_I\to K_I$ is nullhomotopic by Whitehead's theorem.
On the other hand, $\mathcal {Z}_{\mathrm{link}_K\sigma}=\prod_{i=1}^kS^{2n_i+1}$, so the conditions in Theorem \ref{thm:5} are all satisfied.
\end{proof}

\begin{exmp}
Let $K=\partial\Delta^{n_1}*\partial\Delta^{n_2}$ with $n_1,\,n_2>2$ (so $\mathcal {Z}_K\approx S^{2n_1+1}\times S^{2n_2+1}$). For any two
simplices $\sigma\in \partial\Delta^{n_1}$ and $\tau\in \partial\Delta^{n_2}$, it is easy to see that
\[\mathrm{link}_K(\sigma\cup\tau)=\partial\Delta^{n_1-|\sigma|}*\partial\Delta^{n_2-|\tau|}.\]
So if $\sigma,\,\tau\neq\varnothing$, and $|\sigma|+|\tau|<n_1+n_2$, then $\sigma\cup\tau\in K$ satisfies the condition in Theorem \ref{thm:5}. Note that
\[\mathcal {Z}_{\mathrm{link}_K(\sigma\cup\tau}=
\begin{cases}
S^{2n_1-2|\sigma|+1}\times S^{2n_2-2|\tau|+1}\ &\text{ for } |\sigma|<n_1,\ |\tau|<n_2,\\
S^{2n_1-2|\sigma|+1}\ &\text{ for } |\tau|=n_2,\\
S^{2n_2-2|\tau|+1}\ &\text{ for } |\sigma|=n_1.
\end{cases}
\]
Thus by Theorem \ref{thm:5}, we have the cohomology ring of $\mathcal {Z}_{\mathrm{S}_{\sigma*\tau} K}$ is isomorphic to the one of the connected sum of sphere products
\begin{align*}
S^{2n_1+2}\times S^{2n_2+1}\# S^{2n_1+1}\times S^{2n_2+2}\#S^{2n_1-2|\sigma|+3}\times S^{2n_2+2|\sigma|}\\\# S^{2n_2-2|\tau|+3}\times S^{2n_1+2|\tau|}\#S^{2(n_1+n_2-|\sigma|-|\tau|)+4}\times S^{2(|\sigma|+|\tau|)-1}.
\end{align*}
\end{exmp}

\begin{rem}
Actually, it can be proved that the isomorphism in Corollary \ref{cor:1} is induced by a homeomorphism of manifolds \cite{CFW16}: $\ZZ_{\mathrm{S}_\sigma K}\approx M.$
\end{rem}

The proof of Theorem \ref{thm:5} is based on a series of propositions.

\begin{prop}\label{prop:10}
Let $K$ be a simplicial complex with $m$ vertices, $\sigma\in K$, $\mathcal {V}$ be the vertex set of $\mathrm{link}_K\sigma$, $s=|\mathcal {V}|+|\sigma|$.
View $K$ as a subcomplex of $\Delta^{m-1}$, and let $L=\Delta^{m-1}\bigcup_{\mathrm{star}_K\sigma}\mathrm{cone(star}_K\sigma)$. Then the reduced cohomology ring of $\mathcal {Z}_{L}$ is given by the isomorphism
\[\w H^*(\mathcal {Z}_{L};\kk)\cong \big(\w H^*(\Sigma^2\mathcal {Z}_{\mathrm{link}_K\sigma};\kk)\otimes\Lambda_\kk[v_1,\cdots,v_{m-s}]\big)
\times \w H^*(\bigvee\limits_{j=1}^{m-s}\tbinom{m-s}{j}S^{j+2};\kk),\]
where $\Sigma$ denotes the suspension operation on spaces.
\end{prop}

\begin{proof}
Let $[m]=\{v_1,\dots,v_m\}$, $[m+1]=[m]\cup\{v\}$ and $[s]=\{v_{m-s+1},\dots,v_{m}\}$ denote the vertex sets of $K$, $L$
and $\mathrm{star}_K\sigma$ respectively.
It is easy to see that for each $I\in [m+1]$, $I\neq\varnothing$
\begin{equation}\label{eq:4}
L_I\simeq
\begin{cases}
\Sigma\,(\mathrm{link}_K\sigma)_{I\cap[s]}\ \ &\text{ if }v\in I,\ I\cap\sigma=\varnothing \text{ and } I\cap[s]\neq\varnothing,\\
S^0&\text{ if } \{v\}\subsetneq  I\text{ and }I\cap [s]=\varnothing,\\
pt&\text{ otherwise.}
\end{cases}\end{equation}
Then we get the additive isomorphism by Theorem \ref{thm:1}. Clearly the right side of the formula in the proposition has the trivial product structure. From Theorem \ref{thm:e} and the fact that $\w H^*(L_I)\neq 0\Rightarrow v\in I$, we have that the left side also has the trivial product structure. Then the proposition follows.
\end{proof}

\begin{prop}\label{prop:7}
We use the notation as in Proposition \ref{prop:10}.
Suppoe for any $I\subset \mathcal {V}$, the inclusion map
$\varphi_I:(\mathrm{link}_K\sigma)_I\to K_I$ is nullhomotopic. Let
\[K'=K\cup_{\mathrm{star}_K\sigma} \mathrm{cone(star}_K\sigma).\]
Then the cohomology ring of $\mathcal {Z}_{K'}$ is given by the isomorphism
\begin{align*}
\w H^*(\mathcal {Z}_{K'};\kk)\cong \big((\w H^*(\mathcal {Z}_K;\kk)\otimes\Lambda_\kk[v]\big)\times \w H^*(\mathcal {Z}_{L};\kk).
\end{align*}
\end{prop}
\begin{proof}
Let $[m],\, [m+1]$ and $[s]$ be as in the proof of \ref{prop:10}.
We assert that for each $I\in [m+1]$, $I\neq\varnothing$
\begin{equation}\label{eq:3}
K'_I\simeq
\begin{cases}
K_I\bigvee\Sigma((\mathrm{link}_K\sigma)_{I\cap\VV})\ \ &\text{ if }v\in I,\ I\cap\sigma=\varnothing \text{ and } I\cap\VV\neq\varnothing,\\
K_I\bigvee S^0&\text{ if } \{v\}\subsetneq  I\text{ and }I\cap [s]=\varnothing,\\
pt&\text{ if }I=\{v\},\\
K_I&\text{ otherwise.}
\end{cases}\end{equation}
The second, third and forth cases are trivial. For the first case, note that $K'_I$ is the mapping cone of the inclusion $\varphi_I: (\mathrm{link}_K\sigma)_I\to K_I$.
Therefore, the formula follows from the assumption that $\varphi_I$ is nullhomotopic.

Let $L'=K_{\mathcal {V}}\cup_{\mathrm{link}_K\sigma}\mathrm{cone}(\mathrm{link}_K\sigma)$. Then for each $I\in [m+1]$
\begin{equation}\label{eq:5}
L'_I\simeq
\begin{cases}
\Sigma((\mathrm{link}_K\sigma)_{I\cap\VV})\ \ &\text{ if }v\in I \text{ and } I\cap\VV\neq\varnothing,\\
pt\ \text{ or }\ \varnothing&\text{ otherwise.}
\end{cases}\end{equation}
Consider the following commutative diagram of simplicial inclusions
\[
\xymatrix{&L'\ar[d]^{j_1}\ar[dr]^{j_2}&\\
K\ar[r]^{i_1}&K'\ar[r]^{i_2}&L,}
\]
which induces a commutative diagram of Hochster algebras
\[
\xymatrix{&\w{\mathcal {H}}^*_{[m+1]}(L')&\\
\w{\mathcal {H}}^*_{[m+1]}(K)&\w{\mathcal {H}}^*(K')\ar[l]_-{i_1^*}\ar[u]_{j_1^*}&\w{\mathcal {H}}^*(L)\ar[ul]_{j_2^*}\ar[l]_{i_2^*},}
\]
where $i_1,\,i_2,\,j_1,\,j_2$ are natural inclusions.

By formula \eqref{eq:3} and the proof of Proposition \ref{prop:10},
$\w {\mathcal {H}}^*(K')$
additively splits as $A\oplus B$, where
\begin{align*}
&A\cong i_1^*(A)=\mathrm{Im}\,i_1^*,\\
&B=\mathrm{Im}\,i_2^*=\mathrm{Ker}\,i_1^*\cong\w{\mathcal {H}}^*(L).
\end{align*}
Now we show that the splitting is multiplicative. $B$ is an ideal of $\w {\mathcal {H}}^*(K')$ is clear.
From formulae \eqref{eq:3} and \eqref{eq:5}, $A$ can be taken as a direct summand of $\mathrm{Ker}\,j_1^*$, such that $\mathrm{Ker}\,j_1^*/A\subset\HH^1(K')$ and $A=A_1\oplus A_2$ with $A_1=\bigoplus_{v\not\in I} \w H^*(K'_I)$, $A_2\subset\bigoplus_{v\in I} \w H^*(K'_I)$.
At first, we prove that $A$ has to be an algebra. Given two elements $a,\,a'\in A$. Set $a=a_1+a_2$, $a'=a_1'+a_2'$, where
$a_1,a_1'\in A_1$ and $a_2,a_2'\in A_2$.
Clearly, $a_1*a_1'\in A_1$; $a_2*a_2'=0$. If we can prove that $a_1*a_2'$, $a_2*a_1'\in A$, then $a*a'\in A$, and so $A$ is an algebra.
This can be deduced from the fact that $a_1*a_2',\,a_2*a_1'\in\bigoplus_{i\geq2}\HH^i(K')$ and $\mathrm{Ker}\,j_1^*/A\subset\HH^1(K')$.
It remains to show that for any elements $a\in A$ and $b\in B$, $a*b=0$. Since $B$ is an ideal, $a*b\in B$. Note that $a*b\in\bigoplus_{i\geq2}\HH^i(K')$, then
$a*b=i_2^*(c)$ for some $c\in\bigoplus_{i\geq2}\HH^i(L)$. It is easy to see that $j^*_2$ restricted to $\bigoplus_{i\geq2}\HH^i(L)$ is an isomorphism, so $a*b=0$ if and only if $j^*_2(c)=j^*_1(a*b)=0$. Therefore, since $A\subset\mathrm{Ker}\,j^*_1$, then $j_1^*(a*b)=j_1^*(a)*j_1^*(b)=0$.

Finally, by the same argument as in the proof of Proposition \ref{prop:2}, we have \[A\cong \mathrm{Im}\,i_1^*\cong G((\w H^*(\mathcal {Z}_K;\kk)).\] The proposition follows immediately.
\end{proof}

\begin{lem}\label{lem:1}
Suppose there is a short exact sequence of finitely generated abelian groups
\[0\to B\to C\to A\to 0.\]
If $C\cong A\oplus B$, then the exact sequence splits.
\end{lem}

\begin{proof}
If we prove that the exact sequence is zero in $\mathrm{Ext}(A,B)$ (here we view $\mathrm{Ext}(A,B)$ as the set of equivalence classes of extension of $A$ by $B$), then the lemma follows.

As we know, every finitely generated abelian group can be expressed as
\[G=\mathbb{Z}^m\oplus G_{p_1}\oplus\cdots\oplus G_{p_n},\]
where $p_i$ is prime and $G_{p_i}$ is the $p_i$-primary component of $G$, i.e.,
\[G_{p_i}=(\mathbb{Z}_{p_i})^{m_1}\oplus(\mathbb{Z}_{p_i^2})^{m_2}\oplus\cdots\oplus (\mathbb{Z}_{p_i^t})^{m_t}.\]
From homological algebra theory
\[\mathrm{Ext}\big(\bigoplus_i A_i,\bigoplus_j B_j\big)=\bigoplus_i\bigoplus_j\mathrm{Ext}(A_i,B_j)\]
for $i$ and $j$ finite. On the other hand, $\mathrm{Ext}(\mathbb{Z},G)=0$ for any group $G$, and $\mathrm{Ext}(G_p,G_q)=0$ if $p\neq q$, where $p$ and $q$ are prime.
Thus we need only prove the cases that $A=\mathbb{Z}_{p^k}$, and $B=\mathbb{Z}_{p^l}$ or $\mathbb{Z}$.

In the case  $A=\mathbb{Z}_{p^k}$ and $B=\mathbb{Z}_{p^l}$, each nonzero class of $\mathrm{Ext}(A,B)$ has the form
\[0\to \mathbb{Z}_{p^l}\to \mathbb{Z}_{p^{k+l}}\to \mathbb{Z}_{p^k}\to 0.\]
In the case $A=\mathbb{Z}_{p^k}$ and $B=\mathbb{Z}$, each nonzero class of $\mathrm{Ext}(A,B)$ has the form
\[0\to \mathbb{Z}\to \mathbb{Z}\to \mathbb{Z}_{p^k}\to 0.\]
In either case, $C$ is not isomorphic to $A\oplus B$. This completes the proof of the lemma.
\end{proof}

Now let us use the preceding results to complete the proof of Theorem \ref{thm:5}.

\begin{proof}[Proof of Theorem \ref{thm:5}]
The cohomology groups with coefficients in $\kk$ will be implicit throughout the proof.
Let $K',\,[m],\,[m+1],\,[s]$ be the same notations as in Proposition \ref{prop:7}. Then there are
simplicial inclusions $\partial\sigma*\mathrm{link}_K\sigma\stackrel{h}{\hookrightarrow}\mathrm{S}_\sigma K\stackrel{i}{\hookrightarrow} K'$ which induce homomorphisms
$\w {\mathcal {H}}^*(K')\xrightarrow{i^*}\w {\mathcal {H}}^*(\mathrm{S}_\sigma K)\xrightarrow{h^*}\w\HH_{[m+1]}^*(\partial\sigma*\mathrm{link}_K\sigma)$.

Consider the homomorphisms of cohomology groups of full subcomplexes
\[i_I^*: \w H^j(K'_I)\to\w H^j((\mathrm{S}_\sigma K)_I),\ I\subseteq [m+1],\ 0\leq j\leq n-1.\]
We will analyze $i^*_I$ and $h^*_I$ in four cases:
\begin{enumerate}[(1)]
\item $\sigma\not\subset I$. In this case $K'_I=(\mathrm{S}_\sigma K)_I$, so $i_I^*$ is an isomorphism for $0\leq j\leq n-1$. Meanwhile, it is easy to verify that $h^*_I=0$ by the null homotopic assumption, so Im\,$i_I^*=\mathrm{Ker}\,h^*_I$.
\item \label{case:02}$\sigma\subseteq I$ and $I\subsetneq[m]$.
In this case $K'_I=K_I=(\mathrm{S}_\sigma K)_I\bigcup_{(\partial\sigma*\mathrm{link}_K\sigma)_I}(\mathrm{star}_K\sigma)_I$.
 Consider the Mayer-Vietoris sequence
\begin{align*}
\cdots\xrightarrow{\delta}\w H^j(K_I')\to\w H^j((\mathrm{S}_\sigma K)_I)\oplus \w H^j((\mathrm{star}_K\sigma)_I)
\to&\w H^j((\partial\sigma*\mathrm{link}_K\sigma)_I)\\&\xrightarrow{\delta}\w H^{j+1}(K_I')\to\cdots
\end{align*}
We assert that this long exact sequences breaks up into short exact sequences (note that $\w H^*\big((\mathrm{star}_K\sigma)_I\big)=0$)
\[0\to\w H^j(K_I')\xrightarrow{i_I^*}\w H^j((\mathrm{S}_\sigma K)_I)
\xrightarrow{h_I^*}\w H^j((\partial\sigma*\mathrm{link}_K\sigma)_I)\to0\] and these short exact sequences all split.
To prove this assertion, first we show that
\[\w H^j((\mathrm{S}_\sigma K)_I)\cong \w H^j(K_I')\oplus\w H^j((\partial\sigma*\mathrm{link}_K\sigma)_I)\]
for all $j\geq0$.
Since $\mathrm{S}_\sigma K$ is Gorenstein*, then by applying Alexander duality theorem we have that
\[\w H^j((\mathrm{S}_\sigma K)_I)\cong \w H_{n-j-2}((\mathrm{S}_\sigma K)_{[m+1]\setminus I}).\]
Clearly, $(\mathrm{S}_\sigma K)_{[m+1]\setminus I}=K'_{[m+1]\setminus I}$. Formula \eqref{eq:3} implies that
\[\w H_{n-j-2}(K'_{[m+1]\setminus I})\cong \w H_{n-j-2}(K_{[m]\setminus I})
\oplus\w H_{n-j-3}((\mathrm{link}_K\sigma)_{[s]\setminus I}).\]
Since $K$ and $\partial\sigma*\mathrm{link}_K\sigma$ are Gorenstein* complex of dimension $n-1$ and $n-2$ (Proposition \ref{thm:6}) respectively,
and $(\mathrm{link}_K\sigma)_{[s]\setminus I}=(\partial\sigma*\mathrm{link}_K\sigma)_{[s]\setminus I}$, then by applying Alexander duality again, we have that
\[\w H_{n-j-2}(K_{[m]\setminus I})\cong \w H^j(K_I)=\w H^j(K_I'),\ \text{ and }\]
\[\w H_{n-j-3}((\mathrm{link}_K\sigma)_{[s]\setminus I})\cong\w H^j((\partial\sigma*\mathrm{link}_K)_I).\]
Combining these isomorphisms together, we obtain that
\[\w H^j((\mathrm{S}_\sigma K)_I)\cong \w H^j(K_I')\oplus\w H^j((\partial\sigma*\mathrm{link}_K\sigma)_I).\]

Next we prove that $i_I^*$ is injective for all $j\geq0$. Then the desired result follows by Lemma \ref{lem:1}. Let $G$, $G'$ and $G''$ be the torsion subgroups of $\w H^j(K_I')$, $\w H^j((\mathrm{S}_\sigma K)_I)$ and $\w H^j((\partial\sigma*\mathrm{link}_K\sigma)_I)$ respectively. Then we have
$G'\cong G\oplus G''$, and so $|G'|=|G|\cdot|G''|$.
Notice that $G\xrightarrow{i'} G'\xrightarrow{h'}G''$ is exact, where $i'$ (resp., $h'$) is the restriction of $i_I^*$
(resp., $h_I^*$) to $G$ (resp., $G'$). Thus $i'$ has to be injective.
On the other hand, Ker$\,i_I^*$ has to be a finite group since
\[\mathrm{rank}\,\w H^j((\mathrm{S}_\sigma K)_I)=\mathrm{rank}\,\w H^j(K_I')+\mathrm{rank}\,\w H^j((\partial\sigma*\mathrm{link}_K\sigma)_I).\]
Hence Ker$\,i_I^*=\mathrm{Ker}\,i'=0$.
\item $I=[m]$. In this case $\w H^*((\mathrm{S}_\sigma K)_I)=0$, and
$\w H^*(K'_I)=\w H^{n-1}(K)\cong \kk.$ In other words, $i^*(([\mathcal {Z}_K]\otimes 1,0))=0$,
where $([\mathcal {Z}_K]\otimes 1,0)$ is a generator of $\w {H}^{m+n}(\mathcal {Z}_{K'})$ in terms of the isomorphism of Proposition \ref{prop:7}.
\item $\{v\}\cup\sigma\in I$. In this case
\[\w H^*(K'_I,(\mathrm{S}_\sigma K)_I)
\cong\w H^*\big(\mathrm{cone}\,(\mathrm{star}_K\sigma)_I,\mathrm{cone}\,(\partial\sigma*\mathrm{link}_K\sigma)_I\big)=0\]
by excision, so $i^*_I$ is an isomorphism. $h^*_I=0$ is clear.
\end{enumerate}

Combining arguments all above, we have a $\kk$-module isomorphism
\[\w{\mathcal {H}}^*(\mathrm{S}_\sigma K)\cong A\oplus B,\]
where $A=\mathrm{Im}\,i^*=\mathrm{Ker}\,h^*\cong \w H^*(\mathcal {Z}_{K'})/(([\mathcal {Z}_{K}]\otimes 1,0))$ is an ideal of $\w\HH^*(\mathrm{S}_\sigma K)$,
\[B\subseteq \bigoplus_{\sigma\subseteq I\subsetneq[m]}\w H^*((\mathrm{S}_\sigma K)_I);\]
\[B\cong h^*(B)=\mathrm{Im}\,h^*=\bigoplus_{\sigma\subseteq I\subsetneq[m]}\w H^*((\partial\sigma*\mathrm{link}_K\sigma)_I).\]
From the proof of Proposition \ref{prop:7}, we have $A=A_1\times A_2$, where $A_1\cong G(\w H^*(\mathcal {Z}_K))$,
$A_2\cong \w H^*(\mathcal {Z}_{L})\subset\bigoplus_{\sigma\cap I=\varnothing}\w H^*((\mathrm{S}_\sigma K)_I)$. This implies $A_2B\subset A_1$, and so $A_1$ is an ideal of $\w{\mathcal {H}}^*(\mathrm{S}_\sigma K)$.
Since $H^*(\mathcal {Z}_{\mathrm{link}_K\sigma})$ is torsion free, both $A_2$ and $B$ are torsion free.

Now let us verify the ring structure of $\w H^*(\mathcal {Z}_{\mathrm{S}_\sigma K})$. We only prove the case that $\kk=\mathbb{Z}$, the field case is similar and easier.
Since $K$ is Gorenstein*, $\ZZ_K$ is a manifold. Thus the cup product pairing is nonsingular (see appendix \ref{subsec:A1}) for $\ZZ_K$ when torsion in $H^*(\mathcal {Z}_K)$ is factored out. Thus as a direct consequence of Proposition \ref{prop:A4}, we can take $B$ such that $A_1B=0$.
It remains to show the multiplication structure between $A_2$ and $B$.
If $a_2\in A_2$, $b\in B$, then $a_2\in\bigoplus_{v\in I}\w H^*((\mathrm{S}_\sigma K)_I)$, and so
$a_2*b\in\bigoplus_{(\{v\}\cup\sigma)\subseteq I}\w H^*((\mathrm{S}_\sigma K)_I)$, i.e., $a_2*b\in A_1$.
We assert that if $a_2\in (A_2)^i$ (here we use the notation $\Aa^i=\Aa\cap\HH^i$ for a graded subalgebra of $\HH^*$), $b\in B^j$ and $a_2*b\neq0$, then $i+j=n$.
Otherwise, if the order of $a_2*b$ is infinite, then by nonsingularity of the cup product pairing of $A_1$, there is a element $a_1\in(A_1)^{n-i-j}$ such that $a_1*a_2*b\neq0$. Since $a_1*a_2=0$, this is a contradiction. If $a_2*b$ has finite order $m$.
We consider the homomorphism $\eta:G(\w H^*(\mathcal {Z}_K))\to G(\w H^*(\mathcal {Z}_K;\mathbb{Z}_m))$ induced by the map $\mathbb{Z}\to\mathbb{Z}_m$ reducing coefficients mod $m$. Then the nonsingularity of the cup product pairing for $\mathbb{Z}_m$ coefficient (see Proposition \ref{prop:A3}) implies that there is a element $a_1\in(A_1;\mathbb{Z}_m)^{n-i-j}$ such that $a_1*\eta(a_2*b)\neq0$. However $a_1*\eta(a_2)=0$ since Proposition \ref{prop:10} holds for any coefficient. Still a contradiction.

It is easy to verify that
\[\mathrm{rank}\,(A_2)^1=\sum_{j=1}^{m-s}\tbinom{m-s}{j}=\sum_{j=0}^{m-s-1}\tbinom{m-s}{j}=\mathrm{rank}\,B^{n-1},\]
and
\[\mathrm{rank}\,(A_2)^i= k_i\cdot\sum_{j=0}^{m-s}\tbinom{m-s}{j}=k_{n-i}\cdot\sum_{j=0}^{m-s}\tbinom{m-s}{j} =\mathrm{rank}\,B^{n-i},\ \text{ for } i>1,\]
where
\[k_i=\mathrm{rank}\,{\mathcal {H}}^{i-1}(\mathrm{link}_K\sigma)=\mathrm{rank}\,{\mathcal {H}}^{n-|\sigma|-i+1}(\mathrm{link}_K\sigma)=\mathrm{rank}\,{\mathcal {H}}^{n-i}(\partial\sigma*\mathrm{link}_K\sigma)\]
(these equalities follow by formula \eqref{eq:4} and Alexander duality of $\mathrm{link}_K\sigma$).

Now from the Poincar\'e duality of $\mathcal {Z}_{\mathrm{S}_\sigma K}$,
there exists a basis $\{\alpha_{i,\,r}\mid 1\leq r\leq \mathrm{rank}\,(A_2)^i\}$ of $(A_2)^i$
and a basis $\{\beta_{j,\,s}\mid 1\leq s\leq \mathrm{rank}\,B^{j}\}$ of $B^{j}$, such that
\[\alpha_{i,\,r}*\beta_{j,\,s}=
\begin{cases}
[\mathcal {Z}_{\mathrm{S}_\sigma K}]&\ \text{ if } i+j=n,\,r=s,\\
0&\ \text{ otherwise.}
\end{cases}\]
The ring isomorphism in Theorem \ref{thm:5} follows by a straightforward calculation using all formulae above.
\end{proof}

\section{Indecomposability of $\ZZ_K$ for $K$ a flag $2$-sphere}\label{sec:5}
Although the combinatorics of a simplicial complex encodes the topology of the associated moment-angle complex completely, and due
to Bahri, Bendersky, Cohen and Gitler \cite{BBCG1}, the decomposition of $\ZZ_K$ after a suspension can be entirely described by the geometric realization of the full subcomplexes of $K$, it is still hard to describe the homeomorphism class or even the homotopy type of $\ZZ_K$ itself except for some very particular cases of $K$.
For these special cases, the topology of $\ZZ_K$ has been studied in many works: the simplest case is that $\ZZ_K$ has the homotopy type of a wedge of spheres \cite{GT07,IK13,GPTW16}, and for $\ZZ_K$ a manifold, the simplest case is that $\ZZ_K$ has the homeomorphism class of a connected sum of sphere products  \cite{M79,GL13,FCMW16}. Gitler and L\'opez de~Medrano \cite{GL13} have shown that there are
infinite number of $\ZZ_K$ which can be decomposed as the connected sum of other nontrivial manifolds (see more examples of moment-angle manifolds with this property in \cite{CFW16}).
A natural question is that
\begin{que}
For which Gorenstein* complex $K$, the associated moment-angle manifold $\ZZ_K$ can not be decomposed as the connected sum of two nontrivial manifolds?
\end{que}
In this section, we answer this question for a special class of simplicial complexes, that is, the flag $2$-spheres. We shall begin with some lemmata on the combinatorial and
algebraic properties of this kind of simplicial complexes.
\begin{lem}\label{lem:2}
Let $K$ be a triangulation of $D^2$ with $m$ ($m>3$) vertices, and let $\partial K$ be the boundary of $K$, $\mathcal {S}$ be the vertex set of $\partial K$. If $K$ is flag and $K_{\mathcal {S}}=\partial K$, then for any $(v_1,v_2)\in MF(K)$, there exists a subset $I$ of $[m]$, such that $\{v_1,v_2\}\subset I$ and $K_I$ is the boundary of a polygon.
\end{lem}
\begin{proof}
If $v_1, v_2\in \mathcal {S}$, then we can choose $I=\mathcal {S}$. Therefore we always assume $\{v_1,v_2\}\not\subset \mathcal {S}$ in the proof.

We prove it by induction on $m$. Since $K$ is a flag complex, $m=4$ is impossible, then we star with the case that $m=5$. In this case $K$ is the join of the boundary of a $4$-gon and a vertex, then the statement of the lemma obviously holds. For the induction step, taking $v$ to be a vertex in $\mathcal {S}\setminus\{v_1,v_2\}$, put $L=\mathrm{star}_K{(v)}$. Let $l_v$ be the vertex number of $L$.

First we consider the case $l_v=4$ (note that $l_v\geq4$) for some $v\in \mathcal {S}\setminus\{v_1,v_2\}$. In this case $L$ has the form shown in Figure \ref{fig:0}, where $v',v''\in \mathcal {S}$.
\begin{figure}[!ht]
\vspace{8pt}
\begin{overpic}[scale=0.4]{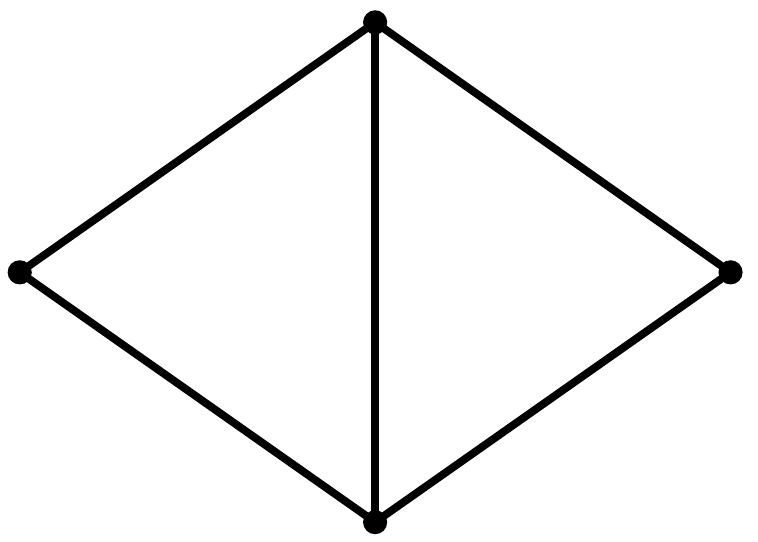}
\put(46,72){$v$}
\put(46,-8){$u$}
\put(100,32){$v''$}
\put(-10,32){$v'$}
\end{overpic}
\vspace{8pt}
\caption{$L$ with $4$ vertices}\label{fig:0}
\end{figure}
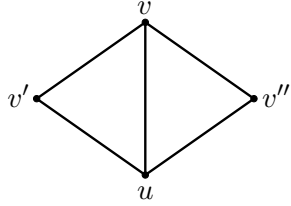
If there is no vertex $u'\in \mathcal {S}\setminus\{v,v',v''\}$ such that $(u,u')\in K$, then it is easy to check that $K'=K_{[m]\setminus \{v\}}$ also satisfies the hypotheses of the lemma. Since $K'$ has $m-1$ vertices, then by induction there exists a subset $I\subset [m]\setminus\{v\}$ such that $(v_1,v_2)\subset I$ and $K'_{I}$ is the boundary of a polygon. Note that $K'$ is a full subcomplex of $K$,
so $K'_I=K_I$, and so $I$ is the desired subset for $K$.
If there exists a vertex $u'\in \mathcal {S}\setminus\{v,v',v''\}$ such that $(u,u')\in K$, we may assume $|\mathcal {S}|\geq 5$
(otherwise $K=\text{cone}\,K_{\Ss}$ is the easy case at the beginning of the proof). Thus there is no vertex $v_0$ in $K$ such that $(v_0,v')$ and $(v_0,v'')\in K$ (If such $v_0$ exists, then $v_0=u'$ and $|\Ss|=4$, contradict the assumption that $|\mathcal {S}|\geq5$). Denote by $\Delta^2$ the $2$-dimensional complex consisting of all subsets of $\{v',v'',u\}$, and let
\[K'=(K\setminus L)\cup \Delta^2.\]
Then $K'$ is a triangulation of $D^2$ with $m-1$ vertices.
It is easily verified that $K'$ satisfies all hypotheses of the lemma, hence by induction there exists a subset $I'\subset [m]\setminus\{v\}$ such that $(v_1,v_2)\subset I'$ and $K'_{I'}$ is the boundary of a polygon. If $\{v',v''\}\not\subset I'$, take $I=I'$, otherwise take $I=I'\cup\{v\}$, then $I$ is the desired subset.

Now we consider the case that $l_v\geq 5$ for all $v\in\mathcal {S}\setminus\{v_1,v_2\}$. Firstly, let us give an order on $\mathcal {S}$, e.g. $\mathcal {S}=\{w_1,w_2,\dots w_n\}$ by counting clockwise, and let $L_i=\mathrm{star}_K{(w_i)}$, $\mathcal {V}_i$ be the vertex set of $L_i$, $l_i=|\mathcal {V}_i|$, $1\leq i\leq n$. So if $w_i\neq v_1,v_2$, then $l_i\geq5$ and $L_i$ is shown in Figure \ref{fig:1}.
\begin{figure}[!ht]
\vspace{8pt}
\begin{overpic}[scale=0.45]{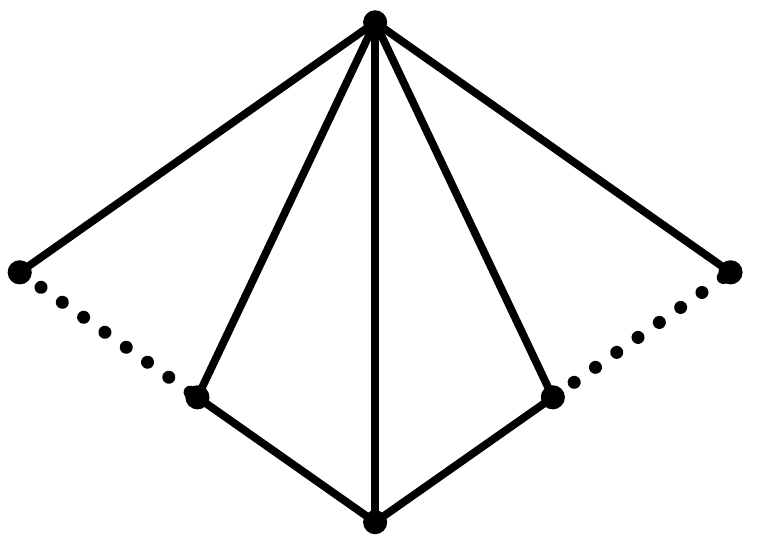}
\put(46,72){$w_i$}
\put(46,-8){$u_i$}
\put(100,32){$w_{i+1}$}
\put(-26,32){$w_{i-1}$}
\end{overpic}
\vspace{8pt}
\caption{}\label{fig:1}
\end{figure}

For $w_i\neq v_1,v_2$, if there are no vertices
\[u_i\in \mathcal {V}_i\setminus\{w_{i-1},w_{i},w_{i+1}\}\ \text{ and }\ w_{k_i}\in \mathcal {S}\setminus\{w_{i-1},w_{i},w_{i+1}\}\]
such that $(u_i,w_{k_i})\in K$, then it is easy to see that $K'=K_{[m]\setminus \{w_i\}}$ satisfies the hypotheses of the lemma. Hence by the same arguments as above, we get the desired subset $I\in [m]\setminus\{w_i\}$. Next we prove that this
kind of $w_i$ always exists, then we complete the proof of the lemma.

Suppose on the contrary that $u_i$ and $w_{k_i}$ always exist (such that $(u_i,w_{k_i})\in K$) for each $w_i\neq v_1,v_2$, and suppose without loss of generality that $w_1\neq v_1,v_2$, then $2<k_1<n$. By the assumption $\{v_1,v_2\}\not\subset \mathcal {S}$ at the beginning of the proof, we have \[\{v_1,v_2\}\cap\{w_1,w_2,\dots,w_{k_1-1}\}=\varnothing\quad\text{or}\]
\[\{v_1,v_2\}\cap\{w_{k_1+1},w_{k_1+2},\dots,w_n\}=\varnothing.\]
Provided without loss of generality that $\{v_1,v_2\}\cap\{w_1,w_2,\dots,w_{k_1-1}\}=\varnothing$. Since $K_{I_1}$ (where $I_1=\{w_1,u_1,w_{k_1}\}$) is a triangulation of $D^1$, $K_{I_1}$ separates $K$ into two simplicial complexes $K_1,\,K_1'$ which are both triangulations of $D^2$. Suppose $K_1$ is the one contains $\{w_1\},\{w_2\},\dots, \{w_{k_1}\}$. If $u_2=u_1$, then $k_1>3$ by the hypothesis $l_2\geq5$ and the flagness of $K$, thus we can rechoose $w_{k_2}$ if needed such that $w_{k_2}=w_{k_1}>3$. If $u_2\neq u_1$, then $w_{k_2}$ must belong to $K_1$, so $3<k_2\leq k_1$. Proceeding inductively, we get that $i+1<k_i\leq k_{i-1}$ for $1\leq i\leq k_1-1$. But taking $i=k_1-1$ we have $k_1<k_1$, a contradiction.
\end{proof}

\begin{cor}\label{cor:2}
If $K$ is a flag $2$-sphere (with $m$ vertices), then for any $(v_1,v_2)\in MF(K)$, there exists a subset $I$ of $[m]$, such that $\{v_1,v_2\}\subset I$ and $K_I$ is the boundary of a polygon.
\end{cor}
\begin{proof}
Choose any vertex $v\in [m]\setminus\{v_1,v_2\}$, then $K_{[m]\setminus\{v\}}$ is apparently a triangulation of $D^2$ and satisfies the hypotheses in Lemma \ref{lem:2}. Hence the statement follows immediately.
\end{proof}
\begin{lem}\label{prop:8}
If $K$ is a flag $2$-sphere, then the Hochster ring $\mathcal {H}^*(K)$ is generated by $\mathcal {H}^1(K)$,
i.e., $\w{\mathcal {H}}^*(K)/[\w{\mathcal {H}}^*(K)]^2=\mathcal {H}^1(K)$.
\end{lem}
\begin{proof}
Suppose the vertex set of $K$ is $[m]$. Since $\mathcal {H}^*(K)$ is a Poincar\'e algebra (note that $\mathcal {H}^*(K)$ is torsion free), then $\mathcal {H}^3(K)=\w H^2(K)=\mathbb{Z}$ is generated by $\mathcal {H}^1(K)\oplus\mathcal {H}^2(K)$. Thus we need only prove that $\mathcal {H}^2(K)$ is generated by $\mathcal {H}^1(K)$, i.e., for any $I\subset [m]$,
if $a\in \w H^1(K_I)$, then $a=\sum_{i}b_i*c_i$, where $b_i\in \w H^0(K_{J_i})$, $c_i\in \w H^0(K_{L_i})$ with $J_i\cap L_i=\varnothing$ and $J_i\cup L_i=I$ for each $i$.

First we prove this for the case $\w H^1(K_I)=\mathbb{Z}$ with a generator $a'$.
Let $\lambda$ be a generator of $\w H^2(K)=\mathbb{Z}$, and let $\wh I=[m]\setminus I$.
Then by Poincar\'e duality of $\mathcal {H}^*(K)$, $\w H^0(K_{\wh I})=\mathbb{Z}$ with a generator $a$ so that $a*a'=\lambda$.
This implies that $K_{\wh I}$ has two path-components: $K_{\wh I}'$ and $K_{\wh I}''$.
Take $\{v_1\}\in K_{\wh I}'$ and $\{v_2\}\in K_{\wh I}''$, then $(v_1,v_2)\in MF(K_{\wh I})\subset MF(K)$. Hence by Corollary \ref{cor:2}, there exists a subset $J\subset [m]$, such that $\{v_1,v_2\}\subset J$ and $K_J$ is the boundary of a polygon. Let $K_1=K_J$, $K_2=K_{\wh I\cup J}$. Consider the simplicial inclusion $\varphi: K_1\to K_2$ which induces a homomorphism of Hochster rings
\[\varphi_{\mathcal {H}}^*:\mathcal {H}^*(K_2)\to \mathcal {H}^*_{\wh I\cup J}(K_1)\]
(View $K_1$ as a simplicial complex on $\wh I\cup J$). It is clear that
\[\varphi^*_{\wh I}: \w H^0((K_2)_{\wh I})\to \w H^0((K_1)_{\wh I})\] is a monomorphism (in fact Im$\varphi^*_{\wh I}$  is a direct summand of $H^0((K_1)_{\wh I})$)
and \[\varphi^*_{J\setminus\wh I}: \w H^0((K_2)_{J\setminus\wh I})\to \w H^0((K_1)_{J\setminus\wh I})\] is an isomorphism.
Since $\mathcal {H}^*_{\wh I\cup J}(K_1)$ is a Poincar\'e algebra (note $K_1$ is a triangulation of $S^1$), there exists an element $b$ of $\w H^0((K_1)_{J\setminus\wh I})$
such that $\varphi^*_{\wh I}(a)*b=\xi$, where $\xi$ is a generator of $\w H^1(K_1)=\mathbb{Z}$. This implies that
$a*(\varphi^*_{J\setminus\wh I})^{-1}(b)$ is one of the generators of $\w H^1(K_2)$ ($\varphi_{\mathcal {H}}^*$ is surjective on $\mathcal {H}^2$).
Thus by applying the Poincar\'e duality of $\mathcal {H}^*(K)$, there is an element $c\in \w H^0(K_{I'})$, where $I'=[m]\setminus(\wh I\cup J)$, such that \[a*(\varphi^*_{J\setminus\wh I})^{-1}(b)*c=\lambda=a*a'.\] It follows that
\[(\varphi^*_{J\setminus\wh I})^{-1}(b)*c=a'.\] Then we get the desired result.

For the general case $\w H^1(K_I)=\mathbb{Z}^n$, similarly we have $\w H^0(K_{\wh I})=\mathbb{Z}^n$. So $K_{\wh I}$ has $n+1$ path-components, say
$L_0,L_1,\dots,L_n$. Give a basis $\{a_i\}_{1\leq i\leq n}$ of $\w H^0(K_{\wh I})$ defined by $a_i=\sum_{(v)\in L_i}(v)$. It determines a dual
basis $\{a_i'\}_{1\leq i\leq n}$ of $\w H^1(K_I)$, i.e., $a_i*a_j'=\lambda$ for $i=j$ and zero otherwise. If $(v_0,v_1)\in MF(K_{\wh I})\subset MF(K)$
with $\{v_0\}\in L_0$ and $\{v_1\}\in L_1$, then by applying Corollary \ref{cor:2} agian, there exists a subset $J\subset [m]$ such that $(v_0,v_1)\in J$ and $K_J$ is the boundary of a polygon. As in the preceding paragraph, let $K_1=K_J$, $K_2=K_{\wh I\cup J}$. Consider the simplicial inclusion $\varphi: K_1\to K_2$ which induces a homomorphism of Hochster rings
\[\varphi_{\mathcal {H}}^*:\mathcal {H}^*(K_2)\to \mathcal {H}^*_{\wh I\cup J}(K_1).\]
In similar fation there exists an element $b$ of $\w H^0((K_1)_{J\setminus\wh I})$
such that $\varphi^*_{\wh I}(a_i)*b=\xi$, where $\xi$ is a generator of $\w H^1(K_1)=\mathbb{Z}$, for $i=1$ and zero otherwise. This implies that
$a_1*(\varphi^*_{J\setminus\wh I})^{-1}(b)$ is one of the generators of $\w H^1(K_2)$.
Appealing to the Poincar\'e duality of $\mathcal {H}^*(K)$, there is an element $c\in \w H^0(K_{I'})$ (where $I'=[m]\setminus(\wh I\cup J)$) such that
\[a_1*(\varphi^*_{J\setminus\wh I})^{-1}(b)*c=\lambda=a_1*a_1'\quad \text{and}\] \[a_i*(\varphi^*_{J\setminus\wh I})^{-1}(b)*c=0\quad \text{ for } i\neq1.\] It follows that
$(\varphi^*_{J\setminus\wh I})^{-1}(b)*c=a_1'$, and so $a_1'$ is generated by $\mathcal {H}^1(K)$.
Similarly, $a_i'$ is generated by $\mathcal {H}^1(K)$ for each $i$, and the lemma follows.
\end{proof}

Note that Lemma \ref{prop:8} dose not hold for general flag complexes. To see this, consider the following example:
\begin{exmp} Let $K$ be a $2$-dimensional flag complex shown in Figure \ref{fig:2}.
\begin{figure}[!ht]
\includegraphics[scale=0.5]{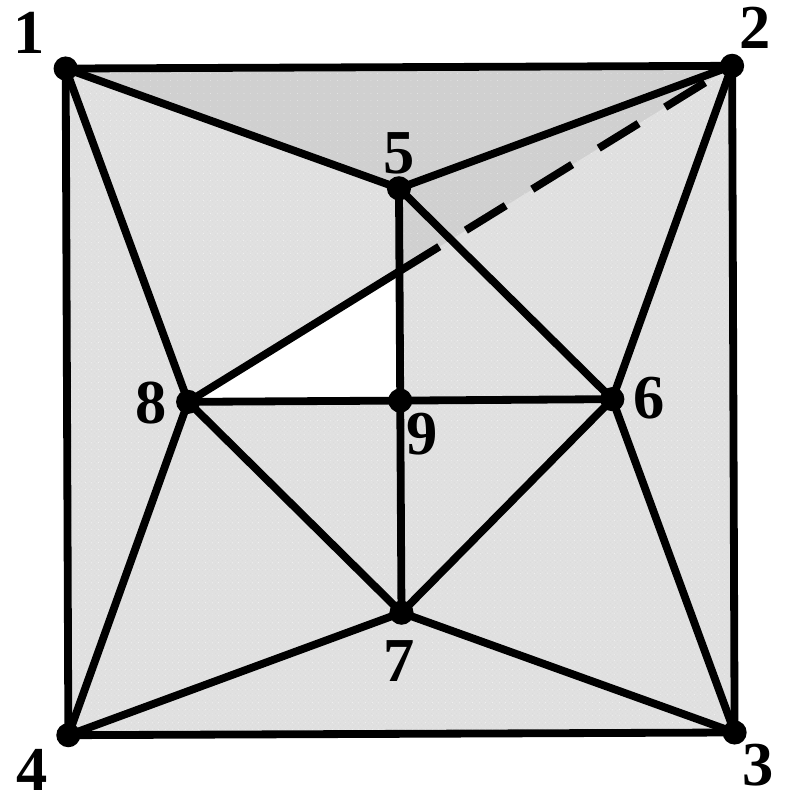}
\caption{}\label{fig:2}
\end{figure}

$K$ contains nine vertices and eleven $2$-simplices:
\[\begin{split}
\{(1,2,5),(1,2,8),(1,4,8),(2,3,6),(2,5,6),(3,4,7),&\\
(3,6,7),(4,7,8),(5,6,9),(6,7,9),(7,8,9)&\}.
\end{split}\]
It is easy to see that $|K|\simeq S^1$, so $\w H^1(K)\cong \mathbb{Z}$. However $\w H^1(K)=\mathcal {H}^2(K)$ is not generated by $\mathcal {H}^1(K)=\w H^0(K)$, since
a straightforward observation shows that  for any division $I\cup J=[9]$, $\w H^0(K_I)=0$ or $\w H^0(K_J)=0$.
\end{exmp}

\begin{Def}\label{def:5}
A ring $\mathcal {R}$ is called \emph{decomposable} if there exists nonzero rings $\mathcal {R}_1$ and $\mathcal {R}_2$  such that $\mathcal {R}\cong\mathcal {R}_1\times \mathcal {R}_2$. Otherwise, $\mathcal {R}$ is called \emph{indecomposable}. If $\mathcal {R}$ is a graded ring, and the corresponding graded decomposition exists, then
$\mathcal {R}$ is called \emph{graded decomposable}. Otherwise, $\mathcal {R}$ is called \emph{graded indecomposable}.

A $\kk$-algebra $A$ is called \emph{decomposable} if there exists nonzero $\kk$-algebra $A_1$ and $A_2$ such that $A\cong A_1\times A_2$. Otherwise, $A$ is called \emph{indecomposable}. The others are defined similarly for $\kk$-algebras.
\end{Def}

\begin{thm}\label{thm:7}
Let $K$ be a flag $2$-sphere. Then $\w {H}^*(\mathcal {Z}_{K})/([\mathcal {Z}_{K}])$ is a graded indecomposable ring, where $[\mathcal {Z}_{K}]$ is the top class of $\w {H}^*(\mathcal {Z}_{K})$.
\end{thm}

\begin{Def}\label{def:6}
An $n$-dimensional manifold $M^n$ is called \emph{prime} if $M^n=M_1^n\# M_2^n$ implies $M_1=S^n$ or $M_2=S^n$.
\end{Def}
The following theorem as a consequence of Theorem \ref{thm:7} answers the question at the beginning of this section for flag $2$-spheres.
\begin{thm}\label{thm:8}
Let $K$ be a flag $2$-sphere. Then $\mathcal {Z}_{K}$ is a prime manifold.
\end{thm}
\begin{proof}
Suppose $\mathcal {Z}_{K}=M_1\#M_2$. Since $\mathcal {Z}_{K}$ is always simply connected (cf. \cite[Corollary 6.19]{BP02}), then by Van Kampen's theorem $M_1$ and $M_2$ are both simply connected. Hence by Poincar\'e conjecture, if $M_i$ has the homology of a sphere, then it actually homeomorphic to the sphere. Note that
\[\w {H}^*(\mathcal {Z}_{K})/([\mathcal {Z}_{K}])=\w H^*(M_1)/([M_1])\times\w H^*(M_2)/([M_2]).\]
Then the theorem follows from Theorem \ref{thm:7} immediately.
\end{proof}
The proof of Theorem \ref{thm:7} is separated into several lemmata, in which we always assume that $K$ is a flag $2$-sphere with vertex set $[m]$.
\begin{lem}\label{lem:3}
If there is an (ungraded) isomorphism of rings
\[\phi:\w{\mathcal {H}}^*(K)/\mathcal {H}^3(K)\xrightarrow{\cong}\mathcal {R}_1\times \mathcal {R}_2\] for two nonzero rings $\mathcal {R}_1$ and $\mathcal {R}_2$, then there exist two subset $I_1,\,I_2\in [m]$ such that $K_{I_1}$ and $K_{I_2}$ are triangulations of $S^1$, and
$\phi(\w H^1(K_{I_i}))\subset \mathcal {R}_i$, $i=1,2$.
\end{lem}
\begin{proof}
Note first that
\[\w{\mathcal {H}}^*(K)/\mathcal {H}^3(K)=\mathcal {H}^1(K)\oplus\mathcal {H}^2(K)\text{ (as group)}.\]
Let $p:\mathcal {H}^1(K)\oplus\mathcal {H}^2(K)\to \mathcal {H}^1(K)$ be the projection map.
We claim that $p$ restricted to $\mathcal {R}_1$ and $\mathcal {R}_2$ are both nonzero. Otherwise suppose $p\mid\mathcal {R}_1=0$, then
$p\mid\mathcal {R}_2$ must be surjective. Since by Lemma \ref{prop:8}, $\w{\mathcal {H}}^*(K)/\mathcal {H}^3(K)$ is generated by
$\mathcal {H}^1(K)$, then $\w{\mathcal {H}}^*(K)/\mathcal {H}^3(K)$ is generated by $\mathcal {R}_2$, a contradiction.

Suppose $r_1\in \mathcal {R}_1$ such that $p_I\phi^{-1}(r_1)\neq 0$ for some $I\subset [m]$, where $p_I$ is the composition of $p$ and the projection
\[p'_I:\mathcal {H}^1(K)=\bigoplus_{I\in[m]}\w H^0(K_I)\to \w H^0(K_I).\]
Assume $(v_1,v_2)\in MF(K_I)\subset MF(K)$. Thus by Lemma \ref{lem:2}, there exists a subset $J\subset [m]$ such that $(v_1,v_2)\subset J$ and $K_J$ is the boundary of a polygon. Applying the same argument as in the proof of Lemma \ref{prop:8}, there exists an element
$a\in \w H^0(K_{J\setminus I})$ such that $a*p_I\phi^{-1}(r_1)\neq0$. Hence $a*\phi^{-1}(r_1)\neq0$. If $\phi(a)=\gamma_1+\gamma_2$ with $\gamma_1\in \mathcal {R}_1$ and $\gamma_2\in \mathcal {R}_2$, then we can write
\[\phi^{-1}(\gamma_1)=a_1+b,\quad \phi^{-1}(\gamma_2)=a_2-b\] with $a_1,a_2\in \w H^0(K_{J\setminus I})$, $a_1+a_2=a$ and $b\in\bigoplus_{I'\neq J\setminus I}\w H^*(K_{I'})$.  Set $b=b_1+b_2$, where $b_1\in \mathcal {H}^1(K)$,
$b_2\in \mathcal {H}^2(K)$.

If $b_1=0$, then we claim that $a_1=a$ and $a_2=0$. Otherwise $a_1=0$ and $a_2=a$ or $a_1,a_2\neq0$. In the first case,
$a*\phi^{-1}(r_1)=\phi^{-1}(\gamma_2)*\phi^{-1}(r_1)\neq0$ (note $\mathcal {H}^1(K)*\mathcal {H}^2(K)=0$ in $\w{\mathcal {H}}^*(K)/\mathcal {H}^3(K)$), so we have $\gamma_2r_1\neq0$, contrary to the fact that $\mathcal {R}_1\mathcal {R}_2=0$.
For the second case, from the
Poincar\'e duality of $\mathcal {H}^*(K_J)$, there exist two elements $a_1',a_2'\in \w H^0(K_{I\cap J})$ such that $a_1*a_1'=a_2*a_2'\neq0$ in
$\w H^1(K_J)=\mathbb{Z}$, so\[\phi^{-1}(\gamma_1)*a_1'=\phi^{-1}(\gamma_2)*a_2'\neq0,\] and so $\gamma_1\phi(a_1')=\gamma_2\phi(a_2')\neq0$.
This implies that $\mathcal {R}_1\cap\mathcal {R}_2\neq0$, a contradiction. When $b\in\mathcal {H}^2(K)$ and $a=a_1$, then from the fact that
$a*a'\neq0$ in $\w H^1(K_J)=\mathbb{Z}$ for some $a'\in \w H^0(K_{I\cap J})$ and \[\phi(a*a')=\phi((a+b)*a')=\gamma_1\phi(a')\in \mathcal {R}_1,\]
we get the desired result $\phi(\w H^1(K_J))\subset \mathcal {R}_1$ (put $I_1=J$). Next we prove that $b_1$ must be zero.

If $b_1\neq0$, without loss of generality we assume $b=b_1\in \w H^0(K_{I'})$ for some $I'\neq J\setminus I$, then there exists an element $b'\in \w H^1(K_{\wh {I'}})$, where $\wh {I'}=[m]\setminus I'$, such that $b*b'\neq0$ in $\w H^2(K)=\mathbb{Z}$ by the
Poincar\'e duality of $\mathcal {H}^*(K)$. Lemma \ref{prop:8} implies that $b'=\sum_i c_i*e_i$ with $c_i\in \w H^0(K_{J_i})$,
$e_i\in \w H^0(K_{J'_i})$, $J_i\cap J'_i=\varnothing$ and $J_i\cup J'_i=\wh {I'}$ for each $i$. Thus $b*c_i*e_i\neq0$ for
some $i$. Since $I'\neq J\setminus I$, then $\wh {I'}\cap (J\setminus I)\neq\varnothing$, and then $J_i\cap (J\setminus I)\neq\varnothing$ or
$J'_i\cap(J\setminus I)\neq\varnothing$. Assume $J_i\cap (J\setminus I)\neq\varnothing$, then
\[\phi^{-1}(\gamma_1)*c_i=(a_1+b)*c_i=b*c_i=(b-a_2)*c_i=-\phi^{-1}(\gamma_2)*c_i\neq0.\]
This implies that $\phi^{-1}(\mathcal {R}_1)\cap\phi^{-1}(\mathcal {R}_2)\neq0$, a contradiction. So $b_1=0$.

Similarly, by considering $r_2\in \mathcal {R}_2$ such that $p_I\phi^{-1}(r_2)\neq 0$ for some $I\subset [m]$, we get a subset $I_2\subset [m]$ such that $\phi(\w H^1(K_{I_2}))\subset \mathcal {R}_2$ and $K_{I_2}$ a triangulation of $S^1$.
\end{proof}
\begin{lem}\label{lem:4}
If there is a graded isomorphism
\[\phi:\w {H}^*(\mathcal {Z}_{K})/([\mathcal {Z}_{K}])\xrightarrow{\cong}\mathcal {R}_1\times \mathcal {R}_2\] for two nonzero graded rings $\mathcal {R}_1$ and $\mathcal {R}_2$, and if there is a subset $I\subset [m]$ such that $K_I$ is the boundary of a polygon and
$\phi(\w H^1(K_I))\subset \mathcal {R}_1$ (here $\w H^1(K_I)$ is viewed as a subgroup of $\w {H}^*(\mathcal {Z}_{K})/([\mathcal {Z}_{K}])$ under the isomorphism given in Theorem \ref{thm:1}), then for any $\sigma\in MF(K_I)$, $\phi(\w H^0(K_\sigma))\subset \mathcal {R}_1$.
\end{lem}
\begin{proof}
Let $\lambda_\sigma$ be a generator of $\w H^0(K_\sigma)=\mathbb{Z}$ for each $\sigma\in MF(K)$. Set $\phi(\lambda_\sigma)=r_\sigma'+r_\sigma''$ with $r_\sigma'\in \mathcal {R}_1$ and $r_\sigma''\in \mathcal {R}_2$. Since $\phi$ preserves grading and each
 element of $\w {H}^3(\mathcal {Z}_{K})$ belongs to $\bigoplus_{\sigma\in MF(K)}\w H^0(K_\sigma)$, then we can suppose
\[\phi^{-1}(r_\sigma')=\sum_{\tau\in MF(K)}k_\tau\cdot\lambda_\tau,\ \ \phi^{-1}(r_\sigma'')=\sum_{\tau\in MF(K)}l_\tau\cdot\lambda_\tau,\quad k_\tau, l_\tau\in \mathbb{Z}.\]
Since $\phi^{-1}(r_\sigma')+\phi^{-1}(r_\sigma'')=\lambda_\sigma$, then
\[k_\sigma+l_\sigma=1,\ \text{ and }\ k_\tau=-l_\tau\ \text{ for }\ \tau\neq\sigma.\]

If $\sigma\in MF(K_I)$, first we prove that $l_\sigma=0$.
Suppose on the contrary that $l_\sigma\neq0$, then $\phi^{-1}(r_\sigma'')*\lambda'\neq0$ for some $\lambda'\in\w H^0(K_{I\setminus\sigma})=\mathbb{Z}$.
Set $\phi(\lambda')=r_1+r_2$ with $r_1\in\RR_1$, $r_2\in\RR_2$. Since $\phi^{-1}(r_\sigma'')*\lambda'\neq0$, we have $r_\sigma''\phi(\lambda')=r_\sigma''(r_1+r_2)\neq0$, and then $r_\sigma''r_2\neq0$. By the same argument as in the proof of Lemma \ref{lem:3}, we have $\phi^{-1}(r_2)=\lambda'+b$ for some $b\in \mathcal {H}^2(K)$.
Now, since $\lambda_\sigma*(\lambda'+b)\neq0\in \w H^1(K_I)$ ($\lambda_\sigma*b=0$ in $\w {H}^*(\mathcal {Z}_{K})/([\mathcal {Z}_{K}])$ for dimension reasons), then
\[\phi(\lambda_\sigma)\phi(\lambda'+b)\neq0\in\phi(\w H^1(K_I))\subset \mathcal {R}_1.\]
On the other hand, notice that $\RR_2$ is an ideal of $\RR_1\times\RR_2$, so \[\phi(\lambda_\sigma)\phi(\lambda'+b)=\phi(\lambda_\sigma)r_2\in\mathcal {R}_2.\] This implies that $\mathcal {R}_1\cap\mathcal {R}_2\neq0$, a contradiction.

Next we prove that $k_\tau=l_\tau=0$ for all $\tau\neq\sigma$. Let $\lambda_{\wh\tau}$ be a generator of
$\w H^1(K_{\wh\tau})=\mathbb{Z}$, where $\wh\tau=[m]\setminus\tau$.
Then $\lambda_\tau*\lambda_{\wh\tau}$ is a generator of $\w {H}^2(K)$. By Lemma \ref{prop:8}, $\lambda_{\wh\tau}=\sum_i c_i*e_i$ with $c_i\in \w H^0(K_{J_i})$, $e_i\in \w H^0(K_{J'_i})$, $J_i\cap J'_i=\varnothing$ and $J_i\cup J'_i=\wh {\tau}$ for each $i$. Thus $\lambda_\tau*c_i*e_i\neq0$ (in $\w{\mathcal {H}}^*(K)$) for
some $i$. Since $\tau\neq \sigma$, then $\wh {\tau}\cap \sigma\neq\varnothing$, and then $J_i\cap \sigma\neq\varnothing$ or
$J'_i\cap\sigma\neq\varnothing$. Assume $J_i\cap \sigma\neq\varnothing$, then $\lambda_\sigma*c_i=0$. So if $l_\tau\neq0$ for some $\tau\neq\sigma$,  then
\[\phi^{-1}(r_\sigma')*c_i=-\phi^{-1}(r_\sigma'')*c_i\neq0.\]
This implies that $\phi^{-1}(\mathcal {R}_1)\cap\phi^{-1}(\mathcal {R}_2)\neq0$, a contradiction. So $\phi(\lambda_\sigma)=r_\sigma'\in \mathcal {R}_1$,
and the lemma follows.
\end{proof}
\begin{lem}\label{lem:5}
If $\sigma_1,\sigma_2\in MF(K)$ and $\sigma_1\cap\sigma_2=\{v\}$, then there exist three subsets $I_1,I_2,I_3\subset [m]$ such that $K_{I_j}$ is the boundary of a polygon, $j=1,2,3$, $\sigma_i\in MF(K_{I_i})$ and $MF(K_{I_3})\cap MF(K_{I_i})\neq\varnothing$, $i=1,2$.
\end{lem}
\begin{proof}
Suppose that $I_1$ and $I_2$ have been taken such that $\sigma_i\in MF(I_i)$ and $K_{I_i}$ is a triangulation of $S^1$, $i=1,2$ (Lemma \ref{lem:2} guarantees the existence of $I_1$ and $I_2$). If $MF(I_1)\cap MF(I_2)\neq\varnothing$, we can take $I_3=I_1$. Otherwise, let $I_3$ be the vertex set of $\mathrm{link}_K(v)$, then $K_{I_3}$ is a triangulation of $S^1$ by the flagness of $K$. It is easily verified that $MF(K_{I_3})\cap MF(K_{I_i})\neq\varnothing$, $i=1,2$.
\end{proof}
\begin{Def}
A simplicial complex $K$ is called a \emph{suspension complex}, if $K=\Sigma L$ ($\Sigma L=S^0*L$) for some complex $L$.
\end{Def}
\begin{lem}\label{lem:6}
If $K$ is not a suspension complex, then for any two different missing faces $\sigma_1,\sigma_2\in MF(K)$, there is a sequence of missing faces $\tau_1,\tau_2,\dots,\tau_n$ such that
$\tau_1=\sigma_1$, $\tau_n=\sigma_2$ and $\tau_i\cap\tau_{i+1}\neq\varnothing$, $i=1,\dots,n-1$.
\end{lem}
\begin{proof}
If $\sigma_1\cap\sigma_2\neq\varnothing$, there is nothing to prove, so we assume $\sigma_1\cap\sigma_2=\varnothing$. Set $\sigma_1=(v_1,v_1')$, $\sigma_2=(v_2,v_2')$.
If $\sigma_2\not\in MF(\mathrm{link}_K(v_1))\cap MF(\mathrm{link}_K(v'_1))$, then one of $(v_1,v_2),\,(v_1,v_2'),\,(v_1',v_2),\,(v_1',v_2')$ is a missing face of $K$, say $(v_1,v_2)\in MF(K)$. We then get the desired sequence: $\tau_1=\sigma_1,\tau_2=(v_1,v_2),\tau_3=\sigma_2$. Similarly we can prove the case that $\sigma_1\not\in MF(\mathrm{link}_K(v_2))\cap MF(\mathrm{link}_K(v'_2))$.

If $\sigma_2\in MF(\mathrm{link}_K(v_1))\cap MF(\mathrm{link}_K(v'_1))$, since $K$ is not a suspension complex, there exists a vertex $u$ of $\mathrm{link}_K(v_1)$ such that
$(v_1',u)$ is a missing face of $K$. Take a missing face $\sigma_3\in MF(\mathrm{link}_K(v_1))$ so that $u\in\sigma_3$. Since $\mathrm{link}_K(v_1)$  is a triangulation of $S^1$, thus if $\mathrm{link}_K(v_1)$ has more than $5$ vertices, then $\sigma_2$ and $\sigma_3$ can obviously be connected by a sequence of missing faces of $\mathrm{link}_K(v_1)$, which satisfies the condition in the lemma.  Thus we get the desired
sequence again: \[\tau_1=\sigma_1,\tau_2=(v_1',u),\tau_{3}=\sigma_3,\dots,\tau_n=\sigma_2.\]

The remaining case is that:
\begin{align*}
&\sigma_2\in MF(\mathrm{link}_K(v_1))\cap MF(\mathrm{link}_K(v'_1))\ \text{ and }\\
&\sigma_1\in MF(\mathrm{link}_K(v_2))\cap MF(\mathrm{link}_K(v'_2))\ \text{ and }\\
&\mathrm{link}_K(v_1)\cong \mathrm{link}_K(v'_1)\cong \mathrm{link}_K(v_2)\cong\mathrm{link}_K(v_2')\cong S^0*S^0.
\end{align*}
It is easy to see that $K\cong S^0*S^0*S^0$ in this case, contrary to the hypotheses, so the lemma holds.
\end{proof}
Now let us use the preceding results to complete the proof of Theorem \ref{thm:7}.
\begin{proof}[proof of Theorem \ref{thm:7}]
Suppose on the contrary that there is a graded ring isomorphism
\[\phi:\w {H}^*(\mathcal {Z}_{K})/([\mathcal {Z}_{K}])\xrightarrow{\cong}\mathcal {R}_1\times \mathcal {R}_2\] for two nonzero graded rings $\mathcal {R}_1$ and $\mathcal {R}_2$. Then by Lemma \ref{lem:3}, there exist two subset $I_1,\,I_2\in [m]$ such that $K_{I_1}$ and $K_{I_2}$ are triangulations of $S^1$, and
$\phi(\w H^1(K_{I_1}))\subset \mathcal {R}_1$, $\phi(\w H^1(K_{I_2}))\subset \mathcal {R}_2$. Lemma \ref{lem:4} implies that
$\phi(\w H^0(K_{\sigma_1}))\subset \mathcal {R}_1$ (resp. $\phi(\w H^0(K_{\sigma_2}))\subset \mathcal {R}_2$.) for each $\sigma_1\in MF(K_{I_1})$ (resp. $\sigma_2\in MF(K_{I_2})$).

If $K$ is not a suspension complex, Lemma \ref{lem:6} says that $\sigma_1\in MF(K_{I_1})$ and $\sigma_2\in MF(K_{I_2})$
can be connected by a sequence of missing faces $\tau_1,\tau_2,\dots,\tau_n$ such that
$\tau_1=\sigma_1$, $\tau_n=\sigma_2$ and $\tau_i\cap\tau_{i+1}\neq\varnothing$, $i=1,\dots,n-1$.
Let $\lambda_{\tau_i}$ be a generator of $\w H^0(K_{\tau_i})=\mathbb{Z}$. Then by combining the conclusion of Lemma \ref{lem:4} and Lemma \ref{lem:5} together, we have that
\[\phi(\lambda_{\tau_i})\in \mathcal {R}_1\cap \mathcal {R}_2\quad \text{for all } 1\leq i\leq n.\]
This is a contradiction.

If $K$ is a suspension complex, i.e., $K=S^0*L$ so that $L$ is the boundary of a $m$-gon ($m>3$), then
\[\ZZ_K\approx S^3\times(\overset{m-2}{\underset{j=1}\#}j\tbinom{m-1}{j+1}S^{j+2}\times S^{m-j}).\]
Hence $\w {H}^*(\mathcal {Z}_{K})/([\mathcal {Z}_{K}])$ is clearly indecomposable.
\end{proof}

Following the same procedure as in the proof of Theorem \ref{thm:7}, we get the following corollary which is needed later.
\begin{cor}\label{cor:6}
If $K$ is a flag $2$-sphere, then $\w {H}^*(\mathcal {Z}_{K};\kk)/([\mathcal {Z}_{K}])$ is a graded indecomposable $\kk$-algebra.
\end{cor}

When $K$ is a flag $n$-sphere with $n\geq3$, it seems that $\w {H}^*(\mathcal {Z}_{K})/([\mathcal {Z}_{K}])$ is also an indecomposable ring (we can not find a counterexample),
although the arguments in the proof for the case $n=2$ can not be applied to this case exactly. So we give the following
\begin{conj}
If $K$ is a flag $n$-dimensional ($n\geq3$) Gorenstein* complex, then $\w {H}^*(\mathcal {Z}_{K})/([\mathcal {Z}_{K}])$ is an indecomposable ring, and consequently
$\mathcal {Z}_{K}$ is a prime manifold.
\end{conj}

Note that if a simplicial $2$-sphere $K$ is not flag and $K\neq\partial\Delta^3$ (which means that $K$ is the connected sum of two simplicial $2$-spheres), then $\w {H}^*(\mathcal {Z}_{K})/([\mathcal {Z}_{K}])$ is not an indecomposable ring by Corollary \ref{cor:3}. Hence when $K$ is a nontrivial (i.e., $K\neq\partial\Delta^3$) simplicial $2$-sphere, the flagness of $K$ is an `if and only if' condition for $\mathcal {Z}_{K}$ to be prime. Naturally, We may propose the following
\begin{prob}
Find a `if and only if' condition such that $\mathcal {Z}_{K}$ is prime for higher dimensional Gorenstein* $K$.
\end{prob}

\section{Decomposition of $H^*(\ZZ_K)$ for simplicial $2$-spheres}\label{sec:6}
In this section, we give a brief discussion about the cohomological rigidity of moment-angle manifolds associated to simplicial $2$-spheres (see \cite{FMW16} for a more specific study about this topic). Cohomology ring does not distinguish closed smooth manifolds in general but for moment-angle manifolds, based on the special toric action on $\ZZ_K$ and by checking some examples, some toric topologist proposed the following (cf. \cite{{B08},CMS11}):
\renewcommand{\theQue}{\Alph{Que}}
\begin{Que}\label{que:1}
Suppose $\mathcal {Z}_{K_1}$ and $\mathcal {Z}_{K_2}$ are two moment-angle manifolds such that
\[H^*(\mathcal {Z}_{K_1})\cong H^*(\mathcal {Z}_{K_2})\] as graded rings. Are $\mathcal {Z}_{K_1}$ and $\mathcal {Z}_{K_1}$ homeomorphic?
\end{Que}

The moment-angle manifolds giving the positive answer to the question are called \emph{cohomologically rigid}.
Buchstaber asked another question in his lecture note \cite[Problem 7.6]{B08} (here we restrict his question to Gorenstein* complexes):

\begin{Que}
Let $K_1$ and $K_2$ be two Gorenstein* complexes (may have different dimension), and let $\mathcal {Z}_{K_1}$ and
$\mathcal {Z}_{K_2}$ be their respective moment-angle manifolds. When a graded ring
isomorphism $H^*(\mathcal {Z}_{K_1};\kk)\cong H^*(\mathcal {Z}_{K_2};\kk)$
implies a combinatorial equivalence $K_1\approx  K_2$?
\end{Que}
Let us call the Gorenstein* complexes giving the positive answer to the
question \emph{$B$-rigid over $\kk$} (if $\kk=\mathbb{Z}$, simply refer to it as \emph{$B$-rigid}). It is clear that  $K$ is $B$-rigid implies that $\mathcal {Z}_{K}$ is cohomologically rigid.

\begin{Def}
If a simplicial $(n-1)$-sphere $K$ can be expressed as a connected sum of two simplicial spheres,
then $K$ is called \emph{reducible}. Otherwise, $K$ is called \emph{irreducible}. Furthermore if $K$ is irreducible and $K\neq\Delta^n$, then $K$ is called
\emph{essentially irreducible}.
\end{Def}
Hence any simplicial sphere $K$ can be decomposed into
\[K=K_1\#K_2\#,\dots,\#K_n\]
such that each $K_i$ is irreducible, and it is easily verified that this decomposition is unique up to a permutation of $[n]$.
In particular if $K$ is a simplicial $2$-sphere, then the condition that $K$ is flag is equivalent to saying $K$ to be \emph{essentially irreducible}. From Theorem \ref{thm:4} we get the following
\begin{prop}
Let $K\in \mathcal {C}(K_1\#K_2\#\cdots\#K_m)$ for irreducible simplicial $n$-spheres $K_1,K_2,\dots,K_m$. If $K$ is $B$-rigid, then $K$ is the only element in \[\mathcal {C}(K_1\#K_2\#\cdots\#K_m).\]
\end{prop}
For the case $K$ is a simplicial $2$-sphere, Choi and Kim \cite{CK11} gave a "$\Longleftrightarrow $" condition for $K$ to be the only element in $\mathcal {C}(K_1\#K_2\#\cdots\#K_m)$ by using a result of Fleischner and Imrich \cite[Theorem 3]{FI79}. The following notation
is adopted from \cite{CK11}. Let $T_4,\, C_8,\, O_6,\, D_{20}$ and $I_{12}$ be the boundary of the five Platonic solids (subscripts indicating vertex numbers): the tetrahedron, the
cube, the octahedron, the dodecahedron and the icosahedron respectively; $\xi_1(C_8)$ and $\xi_1(D_{20})$ be the \emph{first-subdivision} of $C_8$ and $D_{20}$ resp. (see Figure \ref{fig:3}); $\xi_2(C_8)$ and $\xi_2(D_{20})$ be the \emph{second-subdivision} of $C_8$ and $D_{20}$ resp. (see Figure \ref{fig:4}); $B_n$ be the suspension of the boundary of a $(n-2)$-gon, called a \emph{bipyramid} (see Figure \ref{fig:5}).
\begin{figure}[!ht]
\centering
\subfloat{\includegraphics[scale=1.4]{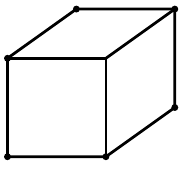}}\quad
\subfloat{\includegraphics[scale=1.4]{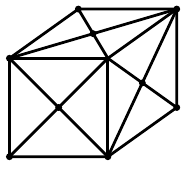}}\quad
\subfloat{\includegraphics[scale=1.4]{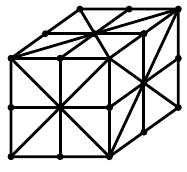}}\quad
\caption{$C_8$, $\xi_1(C_8)$ and $\xi_2(C_8)$.}\label{fig:3}
\end{figure}\vspace{-0.5cm}
\begin{figure}[!ht]
\centering
\subfloat{\includegraphics[scale=0.9]{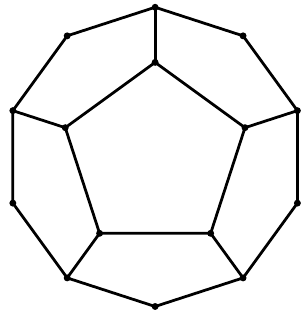}}\quad
\subfloat{\includegraphics[scale=0.9]{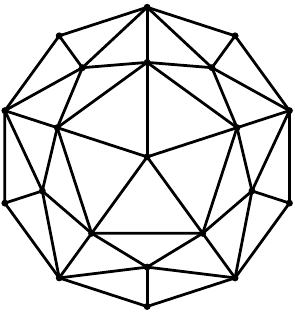}}\quad
\subfloat{\includegraphics[scale=0.9]{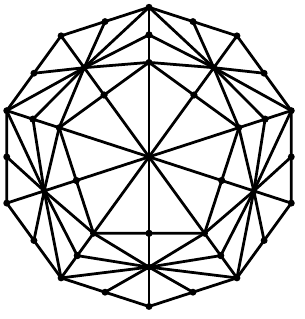}}\quad
\caption{$D_{20}$, $\xi_1(D_{20})$ and $\xi_2(D_{20})$.}\label{fig:4}
\end{figure}
\begin{figure}[!ht]
\includegraphics[scale=1.1]{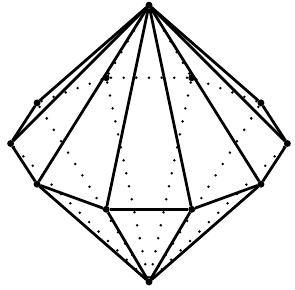}
\caption{$B_n$}\label{fig:5}
\end{figure}
\begin{thm}[{\cite[Theorem 1.2]{CK11}}]\label{thm:d}
Let $K$ be a reducible simplicial $2$-sphere. If $K$ is the only element in $\mathcal {C}(K_1\#K_2\#\cdots\#K_m)$ (where each $K_i$ is irreducible),
then $K=T_4\#T_4\#T_4$ or $K=K_1\#K_2$, where
\begin{align*}
&K_1\in\{T_4,O_6,I_{12}\},\\
&K_2\in\{T_4,O_6,I_{12}, \xi_1(C_8), \xi_2(C_8), \xi_1(D_{20}), \xi_2(D_{20})\}\cup\{B_n : n\geq7\}.
\end{align*}
\end{thm}
In the remainder, we focus our attention on the simplicial $2$-spheres, starting with the following theorem.
\begin{thm}\label{thm:9}
Let $K$ and $K'$ be simplicial $2$-spheres. If
\[K=K_1\#K_2\#,\dots,\#K_n,\quad K'=K'_1\#K'_2\#,\dots,\#K'_{n'}\]
such that each $K_i$ and $K_i'$ are irreducible, and if $H^*(\mathcal {Z}_{K};\kk)\cong H^*(\mathcal {Z}_{K'};\kk)$ (as graded rings), then
$n=n'$ and there is a permutation $j\curvearrowright j'$ such that
\[H^*(\mathcal {Z}_{K_j};\kk)/([\mathcal {Z}_{K_j}])\cong H^*(\mathcal {Z}_{K'_{j'}};\kk)/([\mathcal {Z}_{K'_{j'}}]),\ 1\leq j\leq n.\]
\end{thm}
We first separate off a
\begin{lem}\label{lem:7}
Let $K$ be a flag $2$-sphere, then $G^n(\w H^*(\mathcal {Z}_{K};\kk))/([\mathcal {Z}_{K}]\otimes v_{[n]})$, $n\geq1$, is a graded indecomposable $\kk$-algebra.
\end{lem}
\begin{proof}
Note that \[G^n(\w H^*(\mathcal {Z}_{K};\kk))/([\mathcal {Z}_{K}]\otimes v_{[n]})\cong (\w H^*(\mathcal {Z}_{K};\kk)/([\mathcal {Z}_{K}]))\otimes\Lambda_\kk[n].\]
Then the proof just follows the line of the proof of Theorem \ref{thm:7}.
\end{proof}
\begin{proof}[Proof of Theorem \ref{thm:9}]
Let $m$ be the vertex number of $K$, and $m_i$ be the vertex number of $K_i$.
First we do the case $\kk$ is a field. From Theorem \ref{thm:4}, we know that
\[\w H^*(\mathcal {Z}_K;\kk)/([\mathcal {Z}_K])\cong M_1\times M_2\times\cdots\times M_k\times\kk^l,\]
where each $M_i\cong G^{n_i}\big(\w {H}^*(\mathcal {Z}_{K_i};\kk)\big)/([\mathcal {Z}_{K_i}]\otimes v_{[n_i]})$ with $K_i$ flag, $n_i=m-m_i$ and each
$\kk$ summand in the above formula has trivial multiplication.
Lemma \ref{lem:7} says that $M_i$ are all graded indecomposable. Thus according to Theorem \ref{prop:13} in the appendix, $k,l$ are determined by $\w {H}^*(\mathcal {Z}_K;\kk)$.

Clearly, $K_i$ contains $(3m_i-6)$ $1$-faces, and so has $\tbinom{m_i}{2}-(3m_i-6)$ missing faces.
In other words, \[\mathrm{dim}_\kk(M_i^3)=\binom{m_i}{2}-(3m_i-6)=\frac{m_i^2-7m_i}{2}+6\] (each missing face corresponds to a generator of $\w {H}^3(\mathcal {Z}_{K_i};\kk)$). Since $f(x)=\frac{x^2-7x}{2}+6$ is a strictly monotone increasing function on $x\geq4$, thus $m_i$ ($1\leq i\leq k$) is determined by $\mathrm{dim}_\kk(M_i^3)$. On the other hand, the vertex number $m$ of $K$ is determined by the top dimension $d$ of
$H^*(\mathcal {Z}_K;\kk)$: $m=d-3$, so the cardinality $k'$ of $\{j\in[n]\mid K_j=\partial\Delta^3\}$ equals $d-3-\sum_{i=1}^k m_i$. Therefore
$n=k+k'$ is determined by $H^*(\mathcal {Z}_K;\kk)$, i.e., $n=n'$. Since $\w H^*(\mathcal {Z}_K;\kk)/([\mathcal {Z}_K])\cong \w H^*(\mathcal {Z}_{K'};\kk)/([\mathcal {Z}_{K'}])$, then by combining Theorem \ref{thm:4}, Corollary \ref{cor:6}, Theorem \ref{prop:13} and Lemma \ref{lem:7}, we have up to a permutation
\[G^{n_i}(\w {H}^*(\mathcal {Z}_{K_i};\kk))/([\mathcal {Z}_{K_i}]\otimes v_{[n_i]})\cong G^{n_i}(\w {H}^*(\mathcal {Z}_{K'_i};\kk))/([\mathcal {Z}_{K'_i}]\otimes v_{[n_i]}).\]
Since
\[G^{n_i}(\w H^*(\mathcal {Z}_{K_i};\kk))/([\mathcal {Z}_{K_i}]\otimes v_{[n_i]})\cong (\w H^*(\mathcal {Z}_{K_i};\kk)/([\mathcal {Z}_{K_i}]))\otimes\Lambda_\kk[n_i],\]
it follows that $\w {H}^*(\mathcal {Z}_{K_i};\kk)/([\mathcal {Z}_{K_i}])\cong\w {H}^*(\mathcal {Z}_{K'_i};\kk)/([\mathcal {Z}_{K'_i}])$.

For the integral case, the theorem can be proved in the same way, just by using Corollary \ref{cor:a} instead of Theorem \ref{prop:13}.
\end{proof}
\begin{Def}\label{def:7}
A simplicial sphere $K$ is called \emph{strongly $B$-rigid} if whenever there is another simplicial sphere $L$ such that
$H^*(\mathcal {Z}_K)/([\mathcal {Z}_K])\cong H^*(\mathcal {Z}_L)/([\mathcal {Z}_L])$, then $K\approx L$.
\end{Def}
For flag spheres, we make the following
\begin{conj}\label{conj:2}
Every flag $2$-sphere is strongly $B$-rigid, or more generally, every flag Gorenstein* complex is strongly $B$-rigid.
\end{conj}
In \cite{FMW16}, we have shown that this conjecture is true if  $K$ is a flag $2$-sphere without \emph{$4$-belt} (an \emph{$n$-belt} of $K$ is a full subcomplex isomorphic to the boundary of an $n$-gon).
Notice that if Conjecture \ref{conj:1} and Conjecture \ref{conj:2} are both true, then Theorem \ref{thm:9}, together with Proposition \ref{prop:9} below, give a positive answer to Question \ref{que:1} for moment-angle manifolds associated to simplicial $2$-spheres.
\begin{prop}\label{prop:9}
Let $\mathcal {Z}_K$ be a moment-angle manifold associated to a simplicial $2$-sphere $K$. If $H^*(\mathcal {Z}_K)\cong H^*(\mathcal {Z}_L)$ for a
moment-angle manifold $\mathcal {Z}_L$, then $L$ is a simplicial $2$-sphere.
\end{prop}
Before proving Proposition \ref{prop:9}, we introduce a fundamental fact from the polytope theory, known as the \emph{Lower
Bound Conjecture} (LBC), which was first proved by Barnette \cite{B73} for simplicial polytopes, and generalized to triangulated manifolds by Kalai
\cite{K87}.
\begin{thm}[LBC]
Let $K$ be a triangulated homology $(n-1)$-manifold with $m$ vertices, and let $e$ be the number of edges ($1$-simplices) of $K$. Then $e\geq mn-\binom{n+1}{2}$.
%\item If $e=mn-\binom{n+1}{2}$, then $K$ is a stacked $(n-1)$-sphere (connected sum of several $\partial\Delta^n$).
\end{thm}
\begin{proof}[proof of Proposition \ref{prop:9}]
The proof is by using LBC to show that if $\mathrm{dim}L>2$, then
\[\mathrm{rank}\,H^3(\mathcal {Z}_L)\neq\mathrm{rank}\,H^3(\mathcal {Z}_K).\]
Let $m$ and $l$ be the vertex number of $K$ and $L$ respectively ($m\geq4$ is obvious). From the analysis in the proof of Theorem \ref{thm:9}, $\mathrm{rank}\,H^3(\mathcal {Z}_K)$ equals the number of two-element missing faces of $K$, and so
$\mathrm{rank}\,H^3(\mathcal {Z}_K)=\binom{m}{2}-(3m-6)$. If $\mathrm{dim}\,L=1$, then $l=m+1$, and
$\mathrm{rank}\,H^3(\mathcal {Z}_L)=\tbinom{m+1}{2}-(m+1)$. Thus \[\mathrm{rank}\,H^3(\mathcal {Z}_L)-\mathrm{rank}\,H^3(\mathcal {Z}_K)=3m-7>0.\] If $\mathrm{dim}\,L=n>2$,
since the top dimension $d$ of $H^*(\mathcal {Z}_L)$ equals $n+l+1=m+3$, then $n+2\leq l<m$. By LBC we have
\[\mathrm{rank}\,H^3(\mathcal {Z}_L)\leq \binom{l}{2}-l(n+1)+\binom{n+2}{2}=\frac{l(l-2n-3)}{2}+\binom{n+2}{2}.\]
A straightforward calculation shows that the function $f(x)=\frac{x(x-2n-3)}{2}$ is strictly monotone increasing on $x\geq n+2$.
Thus \[\mathrm{rank}\,H^3(\mathcal {Z}_L)<\binom{m}{2}-m(n+1)+\binom{n+2}{2}=\binom{m}{2}-\frac{(n+1)(2m-2-n)}{2}.\]
Since the function $g(x)=\frac{(x+1)(2m-2-x)}{2}$ is strictly monotone increasing on $x\leq m-2$, then
\[\mathrm{rank}\,H^3(\mathcal {Z}_L)<\binom{m}{2}-(3m-6)=\mathrm{rank}\,H^3(\mathcal {Z}_K).\]
This is a contradiction, so it must have $n=2$.
\end{proof}
Now we go to another question: When a reducible $2$-sphere $K$ is $B$-rigid? Theorem \ref{thm:d} gives a necessary condition for $K$ to be $B$-rigid.
\begin{thm}\label{thm:10}
Let $K$ be a reducible simplicial $2$-sphere. If $K$ is $B$-rigid,
then $K=T_4\#T_4\#T_4$ or $K=K_1\#K_2$, where
\begin{align*}
&K_1\in\{T_4,O_6,I_{12}\},\\
&K_2\in\{T_4,O_6,I_{12}, \xi_1(C_8), \xi_2(C_8), \xi_1(D_{20}), \xi_2(D_{20})\}\cup\{B_n : n\geq7\}.
\end{align*}
\end{thm}
In fact this is also a sufficient condition for $K$ to be $B$-rigid, this is shown in \cite{FMW16} by proving that the irreducible $2$-spheres appear in Theorem \ref{thm:d} are all strongly $B$-rigid.

\section*{Appendix}
\setcounter{section}{1}
\setcounter{thm}{0}
\renewcommand{\thesection}{\Alph{section}}

\subsection{Poincar\'e duality and cup product pairing}\label{subsec:A1}
Let $\RR$ be a commutative ring. For a closed $\RR$-orientable $n$-manifold $M$ ($[M]$ its fundamental class), consider the cup product pairing
\[H^k(M;\RR)\times H^{n-k}(M;\RR)\longrightarrow  \RR,\qquad (\phi,\psi)\mapsto (\phi\s\psi)[M]\]
Such a bilinear pairing $A\times B\to \RR$ is said to be \emph{nonsingular} if the maps $A\to \mathrm{Hom}_\RR(B;\RR)$
and $B\to \mathrm{Hom}_\RR(A;\RR)$, obtained by viewing the pairing as a function of each variable
separately, i.e., $a\mapsto\rho_a$, $\rho_a(b)=(a\s b)[M]$, are both isomorphisms.

As we kown, the cup product pairing is nonsingular for closed $\RR$-orientable
manifolds when $\RR$ is a field, or when $\RR=\mathbb{Z}$ and torsion in $H^*(M;\mathbb{Z})$ is factored out. There is a generalization of this result to the case that
when $\RR=\mathbb{Z}_m$. First we need the following result which is a generalization of universal coefficient theorem for cohomology.

\begin{thm}\label{thm:A1}
Let $\RR$ be a commutative ring. If $\RR$ is an injective $\RR$-module itself, then for any space $X$, there is an isomorphism
\[h: H^k(X;\RR)\cong \mathrm{Hom}_{\RR}(H_k(X;\RR),\RR).\]
\end{thm}
\begin{proof}
Let $(C_*(X),d)$ be the singular chain complex of $X$. Then
\[H^*(X;\RR)=H^*(\mathrm{Hom}(C_*(X),\RR),\delta)=H^*(\mathrm{Hom}_{\RR}(C_*(X)\otimes\RR,\RR),\delta),\]
wherein $\delta$ is the dual coboundary map.
From the proof universal coefficient theorem, we know that there is a natural homomorphism
\[h:H^*(\mathrm{Hom}_{\RR}(C_*(X)\otimes\RR,\RR),\delta)\to \mathrm{Hom}_{\RR}(H_*(C_*(X)\otimes\RR,d),\RR).\]
The elements of $\mathrm{Hom}_{\RR}(H_k(X;\RR),\RR)$ can be represented by $\RR$-homomorphisms $\varphi:{\mathrm{Ker}\,d_k}\to\RR$ that vanish on ${\mathrm{Im}\,d_{k+1}}$. Since $\RR$ is injective, $\varphi$ can be extended to an $\RR$-homomorphism $\bar\varphi:{C_k\otimes\RR}\to\RR$ that still vanish on $\mathrm{Im}\,d_{k+1}$.
Thus we have an $\RR$-homomorphism $\mathrm{Hom}_{\RR}(H_k(X;\RR),\RR)\to\mathrm{Ker}\,\delta_k$. Composing this with the quotient map
$\mathrm{Ker}\,\delta_k\to H^k(X;\RR)$ gives a homomorphism from $\mathrm{Hom}_{\RR}(H_k(X;\RR),\RR)$ to $H^k(X;\RR)$. If we follow this map by $h$ we get the identity map on $\mathrm{Hom}_{\RR}(H_k(X;\RR),\RR)$. Hence by the same reasoning as in the proof of universal coefficient theorem, we have the split exact sequence
\[0\to\mathrm{Ext}_\RR(H_{k-1}(X;\RR),\RR)\to H^k(X;\RR)\xrightarrow{h}\mathrm{Hom}_{\RR}(H_k(X;\RR),\RR)\to 0.\]
Note that $\mathrm{Ext}_\RR(H_{k-1}(X;\RR),\RR)\equiv 0$ since $\RR$ is injective. We get the desired result.
\end{proof}

According to Baer Criterion, $\mathbb{Z}_m$ is an injective $\mathbb{Z}_m$-module, so we immediately get the following
\begin{cor}
For any integer $m\geq2$, there is an isomorphism
\[H^k(X;\mathbb{Z}_m)\cong \mathrm{Hom}_{\mathbb{Z}_m}(H_k(X;\mathbb{Z}_m),\mathbb{Z}_m).\]
\end{cor}

Let $\eta: H^*(M;\mathbb{Z})\to H^*(M;\mathbb{Z}_m)$ (resp. $\eta':H_*(M;\mathbb{Z})\to H_*(M;\mathbb{Z}_m)$) be the homomorphism induced by the map $\mathbb{Z}\to\mathbb{Z}_m$ reducing coefficients mod $m$. It is easy to check that $\mathrm{Im}\,\eta\cong H^{*}(M;\mathbb{Z})\otimes\mathbb{Z}_m$ (resp. $\mathrm{Im}\,\eta'\cong H_{*}(M;\mathbb{Z})\otimes\mathbb{Z}_m$) is a direct summand of $H^*(M;\mathbb{Z}_m)$ (resp. $H_*(M;\mathbb{Z}_m)$). The proof of the following theorem is just by following the way in \cite[Proposition 3.38]{H02}.

\begin{prop}\label{prop:A3}
(i) If $\RR$ is an injective $\RR$-module itself, then the cup product pairing is nonsingular for closed $\RR$-orientable $n$-manifolds.

(ii) For closed orientable manifolds, the cup product pairing (with $\mathbb{Z}_m$ coefficient) restricted to $\mathrm{Im}\,\eta\times(H^*(M;\mathbb{Z}_m)/\mathrm{Ker}\,\w h)$ and $(H^*(M;\mathbb{Z}_m)/\mathrm{Im}\,\eta)\times\mathrm{Ker}\,\w h$ are both nonsingular, where $\w h$ is the map in the
split exact sequence of universal coefficient theorem:
\[0\to\mathrm{Ext}(H_{k-1}(M;\mathbb{Z}),\mathbb{Z}_m)\to H^k(M;\mathbb{Z}_m)\xrightarrow{\w h}\mathrm{Hom}(H_k(M;\mathbb{Z}),\mathbb{Z}_m)\to 0.\]
\end{prop}

\begin{proof}
For (i), consider the composition
\[H^{n-k}(M;\RR)\xrightarrow{h}\mathrm{Hom}_{\RR}(H_{n-k}(M;\RR),\RR)\xrightarrow{D^*}\mathrm{Hom}_{\RR}(H^k(M;\RR);\RR),\]
%\[H^{n-k}(M;\mathbb{Z}_m)\xrightarrow{h}\mathrm{Hom}_{\mathbb{Z}_m}(H_{n-k}(M;\mathbb{Z}_m),\mathbb{Z}_m)\xrightarrow{D^*}\mathrm{Hom}_{\mathbb{Z}_m}(H^k(M;\mathbb{Z}_m);\mathbb{Z}_m),\]
where $h$ is the map in Theorem \ref{thm:A1} and $D^*$ is the Hom dual of the Poincar\'e duality map
$D:H^k\to H_{n-k},\ \phi\mapsto [M]\f\phi$. The composition $D^*h$ sends $\psi\in H^{n-k}(M;R)$ to the homomorphism
$\phi\mapsto \psi([M]\f\phi)=\phi\s\psi([M])$. The first statement follows by Theorem \ref{thm:A1} immediately.

For (ii), consider the following commutative diagram
\[\begin{CD}
\mathrm{Ext}(H_{k-1}(M;\mathbb{Z}),\mathbb{Z})@>>>H^k(M;\mathbb{Z})@>h >>\mathrm{Hom}(H_k(M;\mathbb{Z}),\mathbb{Z})\\
@V\eta VV @V\eta VV @V\eta VV  \\
\mathrm{Ext}(H_{k-1}(M;\mathbb{Z}),\mathbb{Z}_m)@>>>H^k(M;\mathbb{Z}_m)@>\w h>>\mathrm{Hom}(H_k(M;\mathbb{Z}),\mathbb{Z}_m)
\end{CD}\]
It is easy to see that $\eta$ restricted to $\mathrm{Ker}\,h=\mathrm{Ext}(H_{k-1}(M;\mathbb{Z}),\mathbb{Z})$ is surjective.
Hence for any $\phi\in\mathrm{Im}\,\eta$, $\psi\in\mathrm{Ker}\,\w h$, there exist $\phi'\in H^{n-k}(M;\mathbb{Z})$ and $\psi'\in\mathrm{Ker}\,h$ such that $\phi=\eta(\phi')$ and $\psi=\eta(\psi')$. On the other hand, for such $\phi',\psi'$, $[M]\f\phi'\in H_k(M;\mathbb{Z})$, so $\phi'\s\psi'([M])=\psi'([M]\f\phi')=0$, i.e. $\phi'\s\psi'=0$. Thus $\phi\s\psi=\eta(\phi'\s\psi')=0$. So the cup product pairings in (ii) are well defined.

Now we consider the following commutative diagram
\[\begin{CD}
H^{k}(M;\mathbb{Z})@>D:=[M]\smallfrown  >>H_{n-k}(M;\mathbb{Z})\\
@V\eta  VV @V\eta' VV\\
H^{k}(M;\mathbb{Z}_m)@>D:=[M]\smallfrown >>H_{n-k}(M;\mathbb{Z}_m)
\end{CD}\]
Since $M$ is orientable, the maps in the rows are isomorphisms. It is easy to check that the map in the bottom row restricted to $\mathrm{Im}\,\eta\xrightarrow{[M]\f}\mathrm{Im}\,\eta'$ is an isomorphism. Therefore since the following diagram is commutative,
\[
\xymatrix{H^{k}(M;\mathbb{Z}_m)\ar[dr]_{\w h}\ar[r]^-{h}_-{\cong}&\mathrm{Hom}(H_k(M;\mathbb{Z}_m),\mathbb{Z}_m)\ar[d]^{(\eta')^*}\\
&\mathrm{Hom}(H_k(M;\mathbb{Z}),\mathbb{Z}_m)}
\]
we have \[H^{k}(M;\mathbb{Z}_m)/\mathrm{Ker}\,\w h=\mathrm{Hom}(H_k(M;\mathbb{Z}_m),\mathbb{Z}_m)/\mathrm{Ker}\,(\eta')^*=\mathrm{Hom}(\mathrm{Im}\,\eta',\mathbb{Z}_m).\] Composing this with the Hom dual of the Poincar\'e duality map $D^*:\mathrm{Hom}(\mathrm{Im}\,\eta',\mathbb{Z}_m)\to \mathrm{Hom}(\mathrm{Im}\,\eta,\mathbb{Z}_m)$,
we get the desired isomorphism $H^{k}(M;\mathbb{Z}_m)/\mathrm{Ker}\,\w h\to \mathrm{Hom}(\mathrm{Im}\,\eta,\mathbb{Z}_m)$.
Meanwhile, the composition
\[\mathrm{Im}\,\eta\xrightarrow{D}\mathrm{Im}\,\eta'\xrightarrow{\lambda}\mathrm{Hom}(\mathrm{Hom}(\mathrm{Im}\,\eta',\mathbb{Z}_m),\mathbb{Z}_m)=\mathrm{Hom}(H^{k}(M;\mathbb{Z}_m)/\mathrm{Ker}\,\w h,\,\mathbb{Z}_m)\]
gives another isomorphism, where $\lambda$ is the double dual of $\mathrm{Im}\,\eta$, i.e. $\lambda(x)(\phi)=\phi(x)$ for $x\in \mathrm{Im}\,\eta', \,\phi\in\mathrm{Hom}(\mathrm{Im}\,\eta',\mathbb{Z}_m)$.

The nonsingularity for $(H^*(M;\mathbb{Z}_m)/\mathrm{Im}\,\eta)\times\mathrm{Ker}\,\w h$ is similar.
\end{proof}

Let $M$ be an orientable $n$-manifold, $\Aa$ be a graded ideal of $H^*(M)$ satisfies that $\Aa$ is a direct summand of $H^*(M)$ and the cup product pairing restricted to $\Aa/Torsion(\Aa)$ is nonsingular. So there exists a graded subgroup $\BB\subset H^*(M)$ such that $H^*(M)=\Aa\oplus\BB$ (as groups). However, the choice of $\BB$ satisfying the above condition is not unique, and the cup product of $\Aa$ and $\BB$ is not zero in general.

\begin{prop}\label{prop:A4}
Let $M,\,\Aa,\,\BB$ be as above, $\mathcal {T}=Torsion(H^*(M;\mathbb{Z}))$.  Suppose $m$ is the exponent of $\mathcal {T}$. If there exists a subring $\Gamma$ of $H^*(M;\mathbb{Z}_m)$ such that $\eta(\Aa)\subset\Gamma$ and $H^*(M;\mathbb{Z}_m)=\Gamma\oplus\eta(\BB)$ for some $\BB$, then $\BB$ can be chosen properly such that $\Aa\s \BB=0$.
\end{prop}

\begin{proof}
Firstly, let us arbitrarily take $\BB\subset H^*(M)$ so that $H^*(M)=\Aa\oplus\BB$. Set $\w\Aa=\Aa/Torsion(\Aa)$,
and chose a homogenous basis $\{b_i\}_{i=1}^s$ for $\BB$, $b_{i}\in\BB^{k_i}$. As defined at the beginning of the appendix, $\rho_{b_i}\in\mathrm{Hom}(\w\Aa^{n-k_i},\mathbb{Z})$.
Since the cup product pairing on $\w\Aa$ is nonsingular, then there exists an $a_i\in\Aa^{k_i}$ such that $\rho_{a_i}=\rho_{b_i}$.
Let $b_i'=b_i-a_i$, then $b_i'\s \Aa^{n-k_i}=0$. We claim that $b_i'\s \Aa\subset Torsion(\Aa)$. Otherwise $b_i'\s a\neq0\in\w\Aa^{k_i+j}$ for some $a\in\Aa^{j}$, $j<n-k_i$.
Still from the nonsingularity of $\w\Aa$, there is an $a'\in \Aa^{n-k_i-j}$ such that $b_i'\s a\s a'\neq0$, this is a contradiction. Let $\BB'=\langle b_i'\rangle_{i=1}^s$ be the free abelian group generated by $b_i'$, then apparently $H^*(M)=\Aa\oplus\BB'$.

Now we consider the cohomology ring with coefficient $\mathbb{Z}_m$. It is easy to see that $\eta(\Aa)$ is a direct summand of $\Gamma$. Suppose $\Gamma=\CC\oplus\eta(\Aa)$, then
$\CC=H^*(M;\mathbb{Z}_m)/\mathrm{Im}\,\eta$. Thus the restriction map $\rho_{\eta(b_i')}\mid{\CC}$ is an element of $\mathrm{Hom}_{\mathbb{Z}_m}(H^*(M;\mathbb{Z}_m)/\mathrm{Im}\,\eta,\mathbb{Z}_m)$.
According to Proposition \ref{prop:A3} (ii), there is an element $\gamma_i\in\mathrm{Ker}\,\w h$ such that $\rho_{\gamma_i}\mid\CC=\rho_{\eta(b_i')}\mid\CC$. From the proof of Proposition \ref{prop:A3} (ii)
we know that $\mathrm{Ker}\,\w h=\eta(\mathrm{Ker}\,h)$. $\mathrm{Ker}\,h=\mathcal {T}$ by the assumption that $\mathcal {T}$ has exponent $m$.
Hence there exists an $a_i'\in \mathcal {T}$ such that $\eta(a_i')=\gamma_i$.
Now let $b_i''=b_i'-a_i'$. Clearly $b_i''\s \Aa^{n-k_i}=0$, so $\rho_{\eta(b_i'')}\mid\eta(\Aa)=0$. Meanwhile, $\rho_{\eta(b_i'')}\mid\CC=\rho_{\eta(b_i')}-\rho_{\gamma_i}\mid\CC=0$. So $\rho_{\eta(b_i'')}\mid\Gamma=0$. We claim that $b_i''\s\Aa=0$ for each $i$. Otherwise $b_i''\s a\neq0\in Torsion(\Aa^{k_i+j})\subset\mathrm{Ker}\,h$ for some $a\in\Aa^{j}$, $j<n-k_i$. It follows that $\eta(b_i''\s a)\neq0\in\mathrm{Ker}\,\w h$. By applying Proposition
\ref{prop:A3} (ii) again, we can find $c\in\CC^{n-k_i-j}$ such that $\eta(b_i'')\s\eta(a)\s c\neq0\in H^n(M;\mathbb{Z}_m)$. However, since $\Gamma$ is a subring, $c\s\eta(a)\in\Gamma$, contradict to the fact $\rho_{\eta(b_i'')}\mid\Gamma=0$. Finally, let $\BB''=\langle b_i''\rangle_{i=1}^s$, then we have $H^*(M)=\Aa\oplus\BB''$ and $\Aa\s\BB''=0$.
\end{proof}

\subsection{Algebra version of Krull-Schmidt Theorem}
This appendix is devoted to generalizing a famous theorem in group and module theory, which is called ``Krull-Schmidt theorem'' (for a while the theorem was also called the ``Wedderburn-Remak-Krull-Schmidt theorem'') to the ring or algebra case. This theorem can be stated as below (cf. \cite{J80} p.115). A module $M$ is \emph{decomposable} if $M=M_1\oplus M_2$ where $M_i\neq0$. Otherwise, $M$ is called \emph{indecomposable}.
\begin{K-S}
Let $M$ be a module that is both artinian
and noetherian and let $M=M_1\oplus M_2\oplus\cdots\oplus M_n=N_1\oplus N_2\oplus\cdots\oplus N_m$ where the
$M_i$ and $N_j$ are indecomposable. Then $m=n$ and there is a permutation $i\curvearrowright i'$
such that $M_i\cong N_{i'}$, $1\leq i\leq n$.
\end{K-S}
A natural and similar question is that: Suppose a ring $\RR$ is a
direct product of a finite number of indecomposable rings. Are the components unique?
Now by following the line of the proof of this theorem and by adding some supplementary arguments, we can show that there is a similar condition to assure this.

If $f$ is an endomorphism of a ring $\RR$, we define \[f^\infty \RR=\bigcap_{n=0}^\infty f^n(\RR)\ \text{ and }\
f^{-\infty}0=\bigcup_{n=0}^\infty\mathrm{Ker}\,f^n.\]

The following result is the ring version of Fitting's Lemma (cf. \cite{J80} p.113).
\begin{lem}\label{lem:a}
Let $f$ be an endomorphism of a ring $\RR$ that is both artinian and noetherian. If $f^n(\RR)$ is an ideal of $\RR$ for every $n\geq0$, then we have the Fitting decomposition
\[\RR=f^\infty \RR\times f^{-\infty}0.\]
\end{lem}
\begin{proof}
Since $\RR$ is artinian, there is an integer $s$ such that $f^s(\RR)=f^{s+1}(\RR)=\cdots=f^\infty\RR$. Since $\RR$ is noetherian, there exists $t$ such that
$\mathrm{Ker}\,f^t=\mathrm{Ker}\,f^{t+1}=\cdots=f^{-\infty}0$. Let $r=\mathrm{max}\{s,t\}$, so $f^\infty\RR=f^r(\RR)$ and $f^{-\infty}0=\mathrm{Ker}\,f^{r}$.
If $z\in f^{\infty}\RR\cap f^{-\infty}0$, so $z=f^ry$ for some $y\in \RR$. Then $0=f^rz=f^{2r}y$ and $y\in\mathrm{Ker}\,f^{2r}=\mathrm{Ker}\,f^{r}$. Hence $z=f^ry=0$.
Thus $f^{\infty}\RR\cap f^{-\infty}0$.

Now let $x\in\RR$. Then $f^rx\in f^r(\RR)=f^{2r}(\RR)$, so $f^rx=f^{2r}y$ for some $y\in\RR$. Then $f^r(x-f^ry)=0$ and so $z=x-f^ry\in f^{-\infty}0$. Then $x=f^ry+z$ and $f^ry\in f^{\infty}\RR$. Hence there is an additively splitting $\RR=f^\infty \RR\oplus f^{-\infty}0$. From the hypothesis that $f^r(\RR)$ is an ideal of $\RR$, we get the desired decomposition.
\end{proof}
Now we give the algebra version of Krull-Schmidt Theorem  as follows.
\begin{thm}\label{prop:13}
Let $R,\,Q$ be two finitely dimensional $\kk$-algebras with $\kk$ a field and let
\begin{align*}
R&=R_1\times R_2\times\cdots\times R_n,\\
Q&=Q_1\times Q_2\times\cdots\times Q_m,
\end{align*}
where the $R_i$ and $Q_j$ are indecomposable $\kk$-algebras. If $R\cong Q$, then $m=n$ and there is a permutation $i\curvearrowright i'$ such that $R_i\cong Q_{i'}$, $1\leq i\leq n$.
Moreover, if $R,\,Q$ are isomorphic graded $\kk$-algebras and the $R_i$ and $Q_j$ are graded indecomposable, then the isomorphisms $R_i\cong Q_{i'}$ are graded isomorphisms.
\end{thm}
\begin{proof}
Let $i_s: R_s\to R$ and $j_t: Q_t\to Q$ be the respective inclusions and let $p_s: R\to R_s$, $q_t: Q\to Q_t$ be the respective projections.
Suppose $\phi:R\to Q$ is an isomorphism. Define $e_{t}=q_t\phi i_1:R_1\to Q_t$ and $f_{t}=p_1\phi^{-1}j_t:Q_t\to R_1$. Then $g_{t}=f_{t}e_{t}$
is an endomorphism of $R_1$. Note that
\[\sum_{t=1}^m g_{t}=p_1\phi^{-1}(\sum_{t=1}^mj_tq_t)\phi i_1=p_1\phi^{-1}1_R\phi i_1=1_{R_1}.\]
So we have
\[R_1=g_{1}(R_1)+g_{2}(R_1)+\cdots+g_{m}(R_1)\text{ (as $\kk$-module)}.\]
On the other hand, since $Q_i\cdot Q_j=0$ ($i\neq j$), then for $x_i=g_i(y_i)\in g_{i}(R_1)$, $x_j=g_j(y_j)\in g_{j}(R_1)$ we have that
\begin{equation}\label{eq:a1}
x_ix_j=p_1\phi^{-1}(y_i)\cdot p_1\phi^{-1}(y_j)=p_1\phi^{-1}(y_iy_j)=0.
\end{equation}
This implies that $g_{t}(R_1)$ is an  ideal of $R_1$ for all $t$. Similarly, we have that ${g_{t}}^k(R_1)$ is an ideal of $R_1$ for all $k>0$.

We claim that one of $g_t$ is an automorphism of $R_1$. We prove this by induction on $m$. If $m=1$, there is nothing to prove. For the induction step,
we may assume $g_1$ is nilpotent (otherwise $g_1$ is an automorphism by Lemma \ref{lem:a} and the indecomposability of $R_1$).  It is easy to see by equation \eqref{eq:a1} that $h=1-g_1=g_2+\cdots+g_m$ is an endomorphism of $R_1$. Since ${g_1}^r=0$ for some $r$, then \[(1-g_1)(1+g_1+\cdots+{g_1}^{r-1})=1=(1+g_1+\cdots+{g_1}^{r-1})(1-g_1).\]
Thus $h$ is an automorphism of $R_1$. Let $g_i'=g_ih^{-1}$, then $g_2'+\cdots+g'_m=1$. Hence by induction there is a $t$ ($2\leq t\leq m$) such that $g_t'$ is an automorphism of $R_1$, and so is $g_t$.

Assume $g_t$ is an automorphism of $R_1$, now we prove that $e_{t}:R_1\to Q_{t}$ is an isomorphism. It is clear that $e_t$ is a monomorphism and $f_{t}$ restricted to $e_{t}(R_1)$ is an isomorphism. It follows that $Q_{t}$ additively splits as $e_{t}(R_1)\oplus \mathrm{Ker}\,f_{t}$.
Moreover, by a similar analysis as before, we have $e_{t}(R_1)$ is an ideal of $Q_{t}$. Thus there is a multiplicative splitting $Q_{t}=e_{t}(R_1)\times \mathrm{Ker}\,f_{t}$.
The fact that $Q_{t}$ is indecomposable implies that $Q_{t}=e_{t}(R_1)$, i.e. $e_{t}$ is an isomorphism.

Define $e_{s,t}=q_t\phi i_s:R_s\to Q_t$ (thus $e_{1,t}=e_t$ in terms of the former notation).
Let \[\NN_{i}=\{t\in [m]\mid e_{i,t}\text{ is an isomorphism}\}.\]
By the preceding argument, $\NN_{i}\neq\varnothing$ for each $i\in [n]$. Next we show that if $R_i$ has nontrivial multiplication then $\NN_{i}\cap\NN_{j}=\varnothing$ for $j\neq i$.
To see this, suppose $R_i\cong Q_t$ has nontrivial multiplication, then there are two elements $a,b\in Q_t$ such that $ab\neq0$. Suppose on the contrary that there exists $j\neq i$ such that $e_{j,t}$ is an isomorphism, then there are $x\in R_i,\ y\in R_j$ so that $q_t\phi(x)=a,\ q_t\phi(y)=b$. However $ab=q_t\phi(xy)=0$ since $xy=0$, a contradiction.

Suppose $n'$ is the number of $R_i$ which has nontrivial multiplication, and $m'$ is the number of $Q_j$ which has nontrivial multiplication.
Hence $\NN_{i}\cap\NN_{j}=\varnothing$ ($j\neq i$) implies that $n'\leq m'$. From the symmetry of $R$ and $Q$, we also have $m'\leq n'$, then $m'=n'$. Thus up to a permutation of $n'$ there is an $1$-to-$1$ correspondence between the indecomposable
factors with nontrivial multiplication of $R$ and $Q$: $R_i\cong Q_{i}$. On the other hand, since $\kk$ is a field, any indecomposable $\kk$-algebra with trivial multiplication is an $1$-dimensional $\kk$-vector space. Thus the theorem follows by the fact that the dimensions of $R$ and $Q$ over $\kk$ are equal.

The graded case can be proved in the same way.
\end{proof}
\begin{cor}\label{cor:a}
Let $R\cong Q$ be torsion-free rings with finite ranks and let
\begin{align*}
R&=R_1\times R_2\times\cdots\times R_n,\\
Q&=Q_1\times Q_2\times\cdots\times Q_m,
\end{align*}
where the $R_i$ and $Q_j$ are indecomposable rings.
If the $R_i\otimes\kk$ and $Q_j\otimes\kk$ are indecomposable $\kk$-algebras whenever $\kk=\mathbb{Q}$ or $\mathbb{Z}_p$ ($p$ is a prime), then $m=n$ and there is a permutation $i\curvearrowright i'$ such that $R_i\cong Q_{i'}$, $1\leq i\leq n$.
\end{cor}
\begin{proof}
Since $R\cong Q$, then $R\otimes\mathbb{Q}\cong Q\otimes\mathbb{Q}$. According to Theorem \ref{prop:13}, $m=n$ and $R_i\otimes\mathbb{Q}\cong Q_i\otimes\mathbb{Q}$ up to a permutation of $[n]$. Let $e_{i,j}:R_i\to Q_j$ be as in the proof of Theorem \ref{prop:13}. Then by the arguments in the proof of Theorem \ref{prop:13},
if $R_i$ has nontrivial multiplication, then $e_{i,i}\otimes1_\mathbb{Q}$ is an isomorphism and $e_{i,j}\otimes1_\mathbb{Q}$ is not an isomorphism for $i\neq j$.
Thus if $R_i$ has nontrivial multiplication, then
\begin{enumerate}
\item $\mathrm{Ker}\,e_{i,i}=0$ and $\mathrm{rank}\,e_{i,i}(R_i)=\mathrm{rank}\,Q_i$, which implies that $\mathrm{coker}\,e_{i,i}$ is finite;
\item $\mathrm{Ker}\,e_{i,j}\neq0$ or $\mathrm{rank}\,e_{i,j}(R_i)<\mathrm{rank}\,Q_j$ for $i\neq j$, which implies that $e_{i,j}\otimes1_{\mathbb{Z}_p}$ is not an isomorphism for $i\neq j$ and any prime $p$.
\end{enumerate}
On the other hand, from the proof of Theorem \ref{prop:13} we know that there must be a $j\in [m]$ such that $e_{i,j}\otimes1_{\mathbb{Z}_p}$ is an isomorphism. So $e_{i,i}\otimes1_{\mathbb{Z}_p}$ is actually an isomorphism. However if $\mathrm{coker}\,e_{i,i}\neq0$, say $\mathbb{Z}_p$, then $\mathrm{Ker}\,e_{i,i}\otimes1_{\mathbb{Z}_p}\neq0$, a contradiction. Hence $\mathrm{coker}\,e_{i,i}=0$ and $e_{i,i}$ is an isomorphism.

Therefore there is an $1$-to-$1$ correspondence between the indecomposable
factors with nontrivial multiplication of $R$ and $Q$: $R_i\cong Q_{i}$. On the other hand, note that any torsion-free indecomposable ring with trivial multiplication is additively isomorphic to $\mathbb{Z}$. Thus the corollary follows.
\end{proof}

\bibliography{M-A}
\bibliographystyle{amsplain}

%\cite{m79,BBCG10,BLV13,F15,M89,D08,C14,M63}

\end{document}